\documentclass[11pt]{article}
 \usepackage{amsmath}
\usepackage{amsfonts}
\usepackage{color}
\usepackage{srcltx}
\usepackage{color}
\usepackage{float}
\usepackage{fullpage}
\usepackage{multirow}
\usepackage{graphicx}
\usepackage{url}
\usepackage{amsthm}
\def\cA{{\mathcal A}}
\def\cB{{\mathcal B}}

\def\Prox{{\hbox{\rm Prox}}}

\def\bE{{\mathbb{E}}}

\def\Prob{{\hbox{\rm Prob}}}
 
\newcommand{\transp}{{\scriptscriptstyle \top}}
\newtheorem{thm}{\bf{Theorem}}[section]
\newtheorem{lemma}[thm]{\bf{Lemma}}

\newtheorem{df}[thm]{\bf{Definition}}
\newtheorem{cor}[thm]{\bf{Corollary}}

\newtheorem{prop}[thm]{\bf{Proposition}}

\newtheorem{ex}[thm]{\bf{Example}}
\newtheorem{rem}[thm]{\bf{Remark}}

\def\argmin{{\mathop{\hbox{\rm argmin}\,}}}

\title{Multistep stochastic mirror descent for risk-averse convex stochastic programs based on extended polyhedral risk measures}
\author{Vincent Guigues\\ FGV/EMAp,\\ 22250-900 Rio de Janeiro, Brazil\\ {\tt vguigues@fgv.br}}
\date{}

\begin{document}
\maketitle

%\address{Vincent Guigues:
%            FGV/EMAp, 190 Praia de Botafogo, Rio de Janeiro, Brazil, {\tt vguigues@fgv.br}\\
%            UFRJ, Escola Polit\'ecnica, Departamento de Engenharia Industrial\\
%             Ilha do Fund\~ao, CT, Bloco F, Rio de Janeiro, Brazil}

\begin{abstract}
We consider risk-averse convex stochastic programs expressed in terms of extended polyhedral
risk measures.
We derive computable
confidence intervals on the optimal value of such stochastic programs
using the Robust Stochastic
Approximation and the Stochastic Mirror Descent (SMD) algorithms.
When the objective functions are uniformly convex, 
we also propose a multistep extension of the Stochastic Mirror Descent
algorithm and obtain confidence intervals on both the optimal values
and optimal solutions.
Numerical simulations show that 
our 
confidence intervals are much less conservative and are quicker to compute than previously obtained confidence intervals for SMD
and that
the multistep
Stochastic Mirror Descent algorithm can obtain a {\em{good}} approximate solution much quicker than
its nonmultistep counterpart.
\end{abstract}



\section{Introduction} \label{intro}

Consider the convex stochastic optimization problem
\begin{equation} \label{pbpartcase}
\left\{
\begin{array}{l}
\min \; f(x):=\mathcal{R}\left[ g(x,\xi) \right],\\
x \in X,
\end{array}
\right.
\end{equation}
where $\xi \in L_p(\Omega, \mathcal{F}, \mathbb{P}; \mathbb{R}^s)$ is
a random vector with support $\Xi$ and with
\begin{itemize}
\item $g: E \times \mathbb{R}^s \rightarrow \mathbb{R}$ a Borel function which is convex in $x$
for every $\xi$ and $\mathbb{P}$-summable in $\xi$ for every $x$;
\item $X$ a closed and bounded convex set in a Euclidean space $E$; and
\item $\mathcal{R}$ an extended polyhedral risk measure \cite{guiguesromisch10}.
\end{itemize}

Given a sample $\xi_1, \ldots, \xi_N$ from the distribution of $\xi$, 
our goal is to obtain {\em{online nonasymptotic computable}} confidence intervals for the optimal value of 
\eqref{pbpartcase} using as estimators of the optimal value
{\em{variants}} of the {\em{Stochastic Mirror Descent}} (SMD) algorithm.
By computable confidence interval, we mean a confidence interval that does not depend on unknown quantities.
For instance, the confidence intervals from \cite{nemjudlannem09} and \cite{nestioud2010}  are obtained using SMD and a variant of SMD 
but are not computable since they require the evaluation of the objective function $f$ at the approximate solution and typically
for problems of form \eqref{pbpartcase} this evaluation cannot be performed exactly.
The terminology online, taken from \cite{nemlansh09}, refers to the fact that the confidence intervals
are computed in terms of the sample $\xi^N=(\xi_1, \ldots, \xi_N)$ used to solve problem \eqref{pbpartcase},
whereas offline confidence intervals use an additional sample $\xi^{\tilde N}=(\xi_{N+1}, \ldots, \xi_{N+\tilde N})$ independent on $\xi^N$.
Contrary to asymptotic confidence intervals that are valid as the sample size tends to infinity, nonasymptotic
confidence bounds use probability inequalities that are valid for all sample sizes, but they can
be more conservative for this reason.

Before deriving a confidence interval on the optimal value of stochastic program \eqref{pbpartcase},
we need to define
an estimator of this optimal value. A natural estimator is the empirical estimator which is obtained replacing the risk measure
in the objective function by its empirical estimation.\footnote{Note, however, that in this case a solution method still needs to be
specified to solve the corresponding approximate problem.}
In the case of risk-neutral convex problems (when $\mathcal{R}=\mathbb{E}$ is the expectation),
asymptotic and consistency properties of this estimator have been studied extensively.
The asymptotic distribution of the empirical estimator is obtained
using the Delta method (see \cite{romdelta2005}, \cite{shadenrbook}) and the Functional Central
Limit Theorem. This distribution and the consistency of the estimator were derived in 
\cite{dupawets1988}, \cite{shap1989}, \cite{shap1991}
\cite{kingrock1993}, \cite{pflug1995asy}, \cite{baymorton06}, \cite{baymorton11},
\cite{pierrelouisbayak12}. In \cite{makmortonwood} the confidence intervals are built using a multiple replication procedure
while a single replication is used in \cite{baymorton06}.
The paper \cite{Chiralaksanakulmorton04} deals more specifically with the computation
of asymptotic confidence intervals for the optimal value 
of risk-neutral multistage stochastic programs. These results were extended to some stochastic programs
with integer recourse in \cite{ahmedshapiro2002} and \cite{eichrom07}.

Less papers have focused on the determination of nonasymptotic confidence intervals on the optimal value
of a stochastic convex program.
This problem was  however studied  in \cite{pflug99} for risk-neutral convex problems using Talagrand 
inequality (\cite{talagrand1994}, \cite{talagrand1996}).
Similar results, using large-deviation type results are obtained in
\cite{shap2000} and in \cite{kleywegt2001}, \cite{ahmedshapiro2002} for integer models.
Instead of using the empirical estimator, the optimal value of \eqref{pbpartcase}
can be estimated using algorithms for stochastic convex optimization such as the Stochastic
Approximation (SA) \cite{monroe51}, the Robust Stochastic Approximation (RSA) \cite{polyak90}, \cite{polyakjud92}, or the Stochastic 
Mirror Descent (SMD) algorithm  \cite{nemjudlannem09}. 
This approach is used in \cite{nemjudlannem09} and \cite{nemlansh09} where nonasymptotic confidence
intervals on the optimal value of a stochastic convex program are derived. 

The SMD algorithm applied to
stochastic programs minimizing the Conditional Value-at-Risk (CVaR, introduced in \cite{ury2}) of a cost function
was studied in \cite{nemlansh09}.
However, we are not aware of papers deriving confidence intervals for the optimal values of stochastic risk-averse convex programs
expressed in terms of large classes of risk measures, namely law invariant coherent or extended polyhedral risk measures (EPRM).

In this context, the contributions of this paper are the following:
\begin{itemize}
\item[(A)] the description and convergence analysis of Stochastic Mirror Descent is based on three important assumptions: 
(i) convexity of the objective function, (ii) a stochastic oracle provides stochastic subgradients,
and (iii) bounds on some exponential moments are available.
We extend the SMD algorithm to solve risk-averse stochastic programs that minimize
an EPRM of the cost. We provide conditions on these
risk measures such that the aforementioned conditions (i), (ii), and (iii) hold and give 
a formula for stochastic subgradients of the objective function in this situation.
Examples of EPRM satisfying these conditions are the expectation, the CVaR, some spectral risk measures, the optimized certainty equivalent,
the expected utility with piecewise affine utility function, and any linear combination of these.
We also observe that
such stochastic programs can be reformulated as risk-neutral stochastic programs with additional variables and constraints,
making the SMD for risk-neutral problems directly applicable to these reformulations.
\item[(B)]  We provide conditions ensuring that assumptions (i), (ii), and (iii)
are satisfied for two-stage stochastic risk-neutral programs and give again formulas for stochastic subgradients of the
objective function in this case.
\item[(C)] We define a new computable nonasymptotic online confidence interval on the optimal value of a risk-neutral
stochastic convex program using SMD.
Numerical simulations show that this confidence interval is much less conservative than the online confidence interval from \cite{nemlansh09}
and is more quickly computed.
\item[(D)] We apply the ideas of the multistep {\em{method of dual averaging}} described in \cite{nestioud2010}
to propose a multistep Stochastic Mirror Descent algorithm. We also analyse the convergence of this variant of SMD
and provide computable confidence intervals on the optimal value using this algorithm (contrary to  \cite{nestioud2010} where 
for the stochastic {\em{method of dual averaging}} the
confidence intervals were not computable). We present the results of numerical simulations showing the interest
of the multistep variant of SMD on two stochastic (uniformly) convex optimization problems.
\item[(E)] We study the convergence of SMD when the objective function is uniformly convex.
\end{itemize}
More precisely, the outline of the study is as follows. 
In Section \ref{problemsconsidered}, we introduce (in Subsection \ref{generalassumptions}) the assumptions
on the class of problems \eqref{pbpartcase} considered.
In this section we also provide examples of two important classes
of problems satisfying these assumptions: two-stage risk-neutral stochastic convex
programs (Subsection \ref{applitwostage}) and some risk-averse stochastic convex programs expressed in terms of EPRM (Subsection \ref{raverse}).
Since problem \eqref{pbpartcase} can be expressed, eventually after some reformulation (see Section \ref{problemsconsidered}),
as a risk-neutral stochastic convex program, we then explain in Sections \ref{riskneutral} and \ref{mssmd}  
how to obtain a nonasymptotic confidence interval for the optimal value
of \eqref{pbpartcase} in the case when $\mathcal{R}=\mathbb{E}$ is the expectation.
Various algorithms are considered. In Section \ref{riskneutral}, we consider the
RSA algorithm (Subsection \ref{rsa}) and the SMD algorithm
(Subsection \ref{smd}). In each case, on the basis
of an independent sample $(\xi_1,\ldots,\xi_N)$ of $\xi$, the algorithm produces
an approximate optimal value $g^N$ for \eqref{pbpartcase} and a confidence interval for that
optimal value. 
In the particular case when the objective function $f$ is uniformly convex,
we additionally
provide confidence intervals for the optimal solution of \eqref{pbpartcase}.
Applying the techniques discussed in \cite{nestioud2010} to the SMD algorithm, multistep
versions of the Stochastic Mirror Descent algorithm are proposed and studied in Section \ref{mssmd}
in the case when $f$ is uniformly convex.
Confidence intervals for the optimal value of \eqref{pbpartcase} obtained using these
multistep algorithms are also given.
In Section \ref{numsim} numerical simulations illustrate our results: we show that our confidence intervals
are less conservative than previously obtained confidence intervals for SMD
and we show
the interest of the multistep variant of SMD over its traditional, nonmultistep, implementation.
Finally, in Section \ref{applirmfuture}, we comment on future directions of research.

We use the following notation.
For a vector $x \in \mathbb{R}^n$, $x^{+}$ is the vector with $i$-th component given by $x^{+}(i)=\max(x(i),0)$.
We denote by $f'(x)$ one of the subgradient(s) of convex function $f$ at $x$.
For a norm $\|\cdot\|$ of a Euclidean space $E$
associated to a scalar product $\langle \cdot , \cdot \rangle$, the norm $\|\cdot\|_{*}$ conjugate to $\|\cdot\|$
is given by
$$
\|y\|_{*} = \max_{x : \|x\| \leq 1} \,\langle x , y \rangle.
$$
We denote the $\ell_p$ norm of a vector $x$ in $\mathbb{R}^n$ by $\|x\|_p$.
The closed ball of center $x_0$ and radius $R$ is denoted by
$B(x_0, R)$.
By $\Pi_{Y}$, we denote the metric projection
operator onto the set $Y$, i.e., $\Pi_{Y}(x)=\mbox{arg min}_{y \in Y}\;\|y-x\|_2$.
For a nonempty set $X \subseteq \mathbb{R}^{n}$, the polar cone $X^{*}$ is defined
by $X^{*}=\{x^{*} : \left\langle x, x^{*} \right\rangle \leq 0, \; \forall x \in X\}$, where
$\left\langle \cdot , \cdot \right\rangle$ is the standard scalar product on $\mathbb{R}^{n}$.
By
$\xi^t=(\xi_1,\ldots,\xi_t)$, we denote the history of the process $(\xi_t)$
up to time $t$ and by $\mathcal{F}_t$ the sigma-algebra generated by
$\xi^t$. We will denote the Hessian matrix of $f$ at $x$ by $f''(x)$.
Finally, unless stated otherwise, all relations between random variables are
supposed to hold almost surely.

\section{Class of problems considered and assumptions} \label{problemsconsidered}

Consider problem \eqref{pbpartcase} with $\mathcal{R}$ an EPRM:
\begin{df} \cite{guiguesromisch10} \label{defoneperiod}
Let $(\Omega, \mathcal{F}, \mathbb{P})$ be a probability space and
let $K(z)=(K_{1}(z),\ldots$, $K_{n_{2,2}}(z))^\transp$ for given functions\footnote{
The number of components $n_{2, 2}$ of $K$ could be denoted by $n$ to alleviate notation.
We chose to use, as in \cite{guiguesromisch10}, the notation $n_{2, 2}$
where these one-period EPRM are seen as special cases 
of multiperiod ($T$-periods) EPRM  for which additional parameters $n_{t, 1}, n_{t, 2}, t=3, \ldots,T$ are needed.
The same observation applies for the notation used for matrices $B_{2, 1}$ and $B_{2, 0}$.}
$K_1,\ldots,K_{n_{2,2}}:\mathbb{R} \rightarrow \mathbb{R}$. 
A risk measure $\mathcal{R}$ on $L_p(\Omega,\mathcal{F},\mathbb{P})$
with $p \in [2,\infty)$ is called \mbox{\emph{extended}} \mbox{\emph{polyhedral}} if there exist
matrices $A_1, A_2, B_{2, 0}, B_{2, 1}$, and vectors $a_1, a_2, c_1, c_2$
such that for every random variable $Z \in L_p(\Omega,\mathcal{F},\mathbb{P})$
\begin{equation} \label{rmpolymod}
\mathcal{R}(Z)= \left\{
\begin{array}{l}
\displaystyle{\inf} \; c_1^\transp y_1 + \mathbb{E}[c_2^\transp y_2]\\
y_1 \in \mathbb{R}^{k_1},\;y_2 \in L_p(\Omega,\mathcal{F}, \mathbb{P};\mathbb{R}^{k_{2}}),\\
A_1 y_1 \leq a_1,\;A_2 y_2 \leq a_2 \;a.s.,\\
B_{2, 1} y_1 + B_{2, 0} y_2 =  K(Z) \;a.s.
\end{array}
\right.
\end{equation}
\end{df}
In what follows, we make the following assumption on $K$ in \eqref{rmpolymod}:
\begin{itemize}
\item[(A0')] The function $K(z)$ is affine: $K(z)=z k_2+{\tilde k}_2$ for some vectors $k_2, {\tilde k}_2$.
\end{itemize}
Representation \eqref{rmpolymod} can alternatively be written
\begin{equation} \label{compact}
\mathcal{R}(Z)= \left\{
\begin{array}{l}
\displaystyle{\inf}_{y_1} \; c_1^\transp y_1 + \mathbb{E}[ \mathcal{Q}(y_1, Z)]\\
A_1 y_1 \leq a_1,
\end{array}
\right.
\end{equation}
where the recourse function $\mathcal{Q}(y_1, z)$ is given by
\begin{equation} \label{newsecondstagepoly}
\mathcal{Q}(y_1, z)= \left\{
\begin{array}{l}
\displaystyle{\inf}_{y_2} \;c_2^\transp y_2\\
A_2 y_2 \leq a_2\\
B_{2, 0} y_2 = z k_2+{\tilde k}_2-B_{2, 1} y_1.
\end{array}
\right.
\end{equation}
In other words, $\mathcal{R}(Z)$ is the optimal value of a two-stage stochastic program where
$Z$ appears in the right-hand side of the second-stage problem. It follows that we can re-write \eqref{pbpartcase} as
\begin{equation} \label{compact}
\left\{
\begin{array}{l}
\displaystyle{\inf_{y_1, x}} \; c_1^\transp y_1 + \mathbb{E}\Big[ \mathcal{Q}\Big(y_1, g(x, \xi)\Big)\Big]\\
A_1 y_1 \leq a_1,\;x \in X,
\end{array}
\right.
\end{equation}
with $\mathcal{Q}(\cdot, \cdot)$ given by \eqref{newsecondstagepoly}. This problem
is of the form 
\eqref{pbpartcase} with $\mathcal{R}$ the expectation and with
$x, g(x, \xi)$, and $X$ respectively replaced by
${\tilde x}=(y_1, x)$, ${\tilde g}({\tilde x}, \xi)=c_1^\transp y_1 + \mathcal{Q}\Big(y_1, g(x, \xi)\Big)$,
and ${\tilde X}=\{{\tilde x}=(y_1, x) : x \in X, A_1 y_1 \leq a_1 \}$. 

For this reason, in Sections \ref{riskneutral} and \ref{mssmd}, we focus on risk-neutral stochastic problems of the form
\begin{equation} \label{defpbriskneutral}
\left\{
\begin{array}{l}
\min \; f(x):=\mathbb{E}\left[ g(x,\xi) \right],\\
x \in X.
\end{array}
\right.
\end{equation}
However, our analysis is based on some assumptions on $f, X$, and $\xi$, to be described in the next
section. When reformulating risk-averse problem \eqref{pbpartcase} under the form \eqref{defpbriskneutral},
introducing additional variables and constraints, one has to make  some assumptions on the problem
structure and on the EPRM in such a way that this reformulation \eqref{defpbriskneutral} of the problem
satisfies our assumptions. This issue is addressed in Subsection \ref{raverse}.

\subsection{Assumptions}\label{generalassumptions}

For problem \eqref{defpbriskneutral}, in addition to the assumptions on $f$ and $X$ mentioned in the introduction, we make the following assumptions:\\
\par {\textbf{Assumption 1.}} All subgradients of the objective function are bounded on $X$:
$$
\mbox{there exists }0 \leq L <+\infty \mbox{ such that }\|f'(x)\|_{*} \leq L \mbox{   for every }x \in X.
$$

\par Note that Assumption 1 holds if $f$ is finite in a neighborhood of $X$.\\

\par {\textbf{Stochastic Oracle.}}
We assume that samples of $\xi$ can be generated and
the existence  of a {\sl stochastic oracle}: at $t$-th call to the oracle, $x\in X$ being the query point,
the oracle returns $g(x,\xi_t)\in \mathbb{R}$ and a measurable selection
$G(x,\xi_t)$ of a stochastic subgradient $G(x,\xi_t) \in \partial_x g(x, \xi_t)$, where $\xi_1,\xi_2,...$ is an i.i.d
sample of $\xi$.
We treat  $g(x,\xi)$ as an estimate of
$f(x)$ and $G(x, \xi)$ as an estimate of a subgradient of $f$ at $x$.\\

\par {\textbf{Assumption 2.}} Our estimates are {\sl unbiased}:
$$
\forall x\in X: f(x)=\bE_\xi \left[ g(x,\xi) \right] \ \;\mbox{ and }\;\ f'(x):=\bE_\xi \left[ G(x,\xi) \right] \in\partial f(x).
$$

From now on, we set
\begin{equation}\label{defdeltaD}
\delta(x,\xi)=g(x,\xi)-f(x),\,\,\Delta(x,\xi)=G(x,\xi)-f'(x),
\end{equation}
so that
$$
\bE_\xi \left[ \delta(x,\xi) \right] =0,\,\,\bE_\xi \left[ \Delta(x,\xi) \right] =0.
$$
 In the sequel, we assume that the observation errors of our oracle satisfy some assumptions
(introduced in \cite{nemjudlannem09}) additional to having zero means.
Specifically, our {\sl minimal} assumption is the following:\\
\par {\textbf{Assumption 3.}} For some $M_1, M_2\in(0,\infty)$ and for all  $x\in X$
\begin{equation}\label{assw}
\begin{array}{lrcl}
(a)&\bE\Big[ \delta^2(x,\xi) \Big] &\leq& M_1^2, \vspace*{0.1cm} \\
(b)&\bE \Big[ \|\Delta(x,\xi)\|_*^2 \Big] &\leq& M_2^2 .
\end{array}
\end{equation}
\par Under our minimal assumption, we will obtain an upper bound on the average error on the
optimal value of \eqref{pbpartcase}. To obtain a confidence interval on this optimal value, we will need a stronger assumption:\\
\par {\textbf{Assumption 4.}} For some $M_1, M_2 \in (0,\infty)$ and for all $x\in X$ it holds that
\begin{equation}\label{asss}
\begin{array}{lrcl}
(a)&\bE\Big[\exp\{\delta^2(x,\xi)/M_1^2\}\Big] &\leq& \exp\{1\}, \vspace*{0.1cm} \\
(b)&\bE \Big[\exp\{\|\Delta(x,\xi)\|_*^2/M_2^2\} \Big] &\leq& \exp\{1\}.\\
\end{array}
\end{equation}
Note that condition \eqref{asss} is indeed stronger than condition \eqref{assw}: if a random variable
$Y$ satisfies $\mathbb{E}\Big[ \exp\{Y\} \Big] \leq \exp\{1\}$ then by Jensen inequality, using the concavity
of the logarithmic function, $\mathbb{E}\Big[ Y \Big] = \mathbb{E}\Big[ \ln \Big( \exp\{Y\}  \Big) \Big] \leq \ln\Big( \mathbb{E}\Big[\exp\{Y\}  \Big] \Big) \leq 1$.

For a given confidence level, a smaller confidence interval can be obtained under an even stronger assumption:\\
\par {\textbf{Assumption 5.}} For some $M_1, M_2 \in (0,\infty)$ and for all $x\in X$ it holds that
\begin{equation}\label{asst}
\begin{array}{lrcl}
(a)&\bE\Big[ \exp\{\delta^2(x,\xi)/M_1^2\}\Big] &\leq& \exp\{1\},\vspace*{0.1cm} \\
(b)&\|\Delta(x,\xi)\|_*&\leq& M_2\mbox{\ almost surely}.
\end{array}
\end{equation}
Observe that the validity of (\ref{asst}) for all $x\in X$ and some $M_1, M_2$
implies the validity of (\ref{asss}) for all $x\in X$ with the same $M_1, M_2$.

The computation of the confidence intervals on the optimal value of \eqref{pbpartcase} using the SMD and multistep SMD
algorithms presented in Sections \ref{riskneutral} and \ref{mssmd} requires the knowledge of constants $L, M_1$, and $M_2$ satisfying the assumptions above. 
For instance, the best (smallest) constants $M_1, M_2$ satisfying Assumption 4 are $M_1 = \sup_{x \in X} \pi[\delta(x, \cdot)]$ and $M_2=\sup_{x \in X} \pi[\|\Delta(x,\cdot)\|_{*}]$ where $\pi$ is the Orlicz semi-norm given by
$$
\pi[h]=\inf\big\{M \geq 0:\; \mathbb{E} \{\exp\{h^2(\xi)/M^2\} \} \leq\exp\{1\}\big\}.
$$
For many problems of form \eqref{pbpartcase} with $\mathcal{R}=\mathbb{E}$ the expectation operator, upper bounds on these best constants can be computed analytically, see
for instance \cite{nemjudlannem09}, \cite{nemlansh09}, \cite{ioudnemgui15}.

\subsection{Two-stage stochastic convex programs} \label{applitwostage}

Consider the case when \eqref{pbpartcase} is a two-stage risk-neutral stochastic convex program, i.e.,
$\mathcal{R}=\mathbb{E}$ is the expectation, $x$ is the first-stage decision variable,
$f(x)=f_1(x)+\mathbb{E}_{\xi}[\mathcal{Q}(x, \xi)]$ where $\mathcal{Q}(x, \xi)$ is the second-stage cost given by
\begin{equation}\label{secondstagecost}
\mathcal{Q}(x, \xi)=\left\{
\begin{array}{l}
\min_{y}\;f_2(x, y, \xi)\\
y \in \mathcal{S}(x,\xi)=\{y : g_2(x, y, \xi) \leq 0, A x + B  y = \xi\}
\end{array}
\right.
\end{equation}
for some function $g_2$ taking values in $\mathbb{R}^m$ and
some random vector $\xi \in L_p(\Omega, \mathcal{F}, \mathbb{P})$ with $p \geq 2$ and support
$\Xi$. We make the following assumptions:
\begin{itemize}
\item[(A0)] $X$ is a nonempty, compact, and  convex set;
\item[(A1)] $f_1$ is convex, proper, lower semicontinuous, and is finite in a neighborhood of $X$;
\item[(A2)] for every $x \in X$ and $y \in \mathbb{R}^q$ the function   
$f_2(x, y, \cdot)$ is measurable and for every $\xi \in \Xi$, the function
$f_2(\cdot, \cdot, \xi)$ is differentiable and convex;
\item[(A3)] for every $\xi \in \Xi$, the function $g_2(\cdot, \cdot, \xi)$
is convex and differentiable;
\item[(A4)] for every $x \in X$ and for every $\xi \in \Xi$ the set
$\mathcal{S}(x,\xi)$ is compact and there exists
$y_{x, \xi} \in \mathcal{S}(x,\xi)$ such that $g_2(x, y_{x, \xi}, \xi)<0$.
\end{itemize}
With the notation of Section \ref{intro}, we have $f(x)=\mathbb{E}[g(x,\xi)]$ where
$g(x,\xi)=f_1(x)+\mathcal{Q}(x, \xi)$.
Assumptions (A1), (A2), and (A3) imply the convexity of $f$.
Assumptions (A2) and (A4) imply that for every $\xi \in \Xi$,
the second-stage cost $\mathcal{Q}(x, \xi)$ is finite
which implies the finiteness of $\delta(x, \xi)$ for every $x \in X$. Relations 
\eqref{assw}(a), \eqref{asss}(a), and \eqref{asst}(a)
in respectively Assumptions 3, 4, and 5 are thus satisfied.
Assumptions (A2), (A3), and (A4) imply that for every
$\xi \in \Xi$, the function $x \rightarrow \mathcal{Q}(x, \xi)$ is subdifferentiable
on $X$
with bounded subgradients at any $x \in X$. 
For fixed $x \in X$ and $\xi \in \Xi$, let $y(x, \xi)$ be an optimal solution
of \eqref{secondstagecost} and consider the dual problem 
\begin{equation}\label{dualpb}
\displaystyle \sup_{\lambda \in \mathbb{R}^s, \mu \geq 0 }\; \theta_{x, \xi}(\lambda, \mu)
\end{equation}
for the dual function
$$
\theta_{x, \xi}(\lambda, \mu)=\displaystyle \inf_{y \in \mathbb{R}^q} \;f_2(x,y, \xi) + \lambda^\transp (Ax+By-\xi) + \mu^\transp g_2(x,y, \xi).
$$
Let $(\lambda(x,\xi), \mu(x, \xi))$
be an optimal solution of \eqref{dualpb} (for problem  \eqref{secondstagecost}, 
$\lambda(x,\xi)$ and  $\mu(x, \xi)$ are optimal Lagrange multipliers
for respectively the equality and inequality constraints).
Then for any $x \in X$ and $\xi \in \Xi$, denoting by
$I(x, y, \xi):=\{i \in \{1,\ldots,m\} : g_{2, i}(x, y, \xi)=0\}$ the set of active inequality constraints
at $y$ for problem \eqref{secondstagecost},
$$
s(x,\xi)=\nabla_x f_2(x, y(x, \xi), \xi) + A^\transp \lambda(x, \xi) + \sum_{i \in I(x, y(x, \xi), \xi)} \mu_i(x, \xi) \nabla_x g_{2, i}(x, y(x, \xi), \xi)
$$
belongs to the subdifferential $\partial_x \mathcal{Q}(x, \xi)$ and is bounded (see \cite{guigues2014cvsddp} for instance for a proof).
As a result, for any $x \in X$, denoting by $s_1(x)$ an arbitrary element
from $\partial f_1(x)$, $f'(x):=\mathbb{E}[G(x, \xi)]$ is a subgradient of
$f$ at $x$
for
$G(x, \xi)=s_1(x)+s(x,\xi)$ and recalling that (A1) holds, $\|G(x, \xi)\|_{*}$ is bounded for any $x \in X$ and $\xi \in \Xi$. It follows that
Assumption 1 is satisfied as well as Relations 
\eqref{assw}(b), \eqref{asss}(b), and \eqref{asst}(b)
in respectively Assumptions 3, 4, and 5.

\subsection{Risk-averse stochastic convex programs} \label{raverse}

Consider reformulation \eqref{compact} of problem \eqref{pbpartcase}.
To guarantee the convexity of the objective function in this problem as well as Assumptions 1-5, we
make the following assumptions on $\mathcal{R}$ and $g$:
\begin{itemize}
\item[(A1')] Complete recourse: $Y_1:=\{y_1 : A_1 y_1 \leq a_1 \}$ is nonempty and bounded and $\{B_{2,0} y_2 : A_2 y_2 \leq a_2\}=\mathbb{R}^{n_{2,2}}$.
\item[(A2')] The feasible set 
\begin{equation} \label{dualfset}
\mathcal{D}=\{\lambda=(\lambda_1,\lambda_2) \in \mathbb{R}^{n_{2,2}} \small{\times} \mathbb{R}^{n_{2,1}} \;:\;\lambda_2 \leq 0,\;\;B_{2, 0}^{\transp} \lambda_1 + A_2^{\transp} \lambda_2 = c_2\}
\end{equation}
of the dual of the second-stage problem \eqref{newsecondstagepoly} is nonempty.
\item[(A3')] The set $\mathcal{D}$ given by \eqref{dualfset} is bounded.
\item[(A4')] For the set $\mathcal{D}$ given by \eqref{dualfset}, we have that
$\mathcal{D} \subseteq {\{ -k_2 \}}^{*} \small{\times} \mathbb{R}^{n_{2,1}}$.
\item[(A5')] For every $\xi \in \Xi$, the function $g(\cdot, \xi)$ is convex and lower semicontinuous on $X$ and
finite in a neighborhood of $X$.
\end{itemize}
If $X$ is closed, bounded, and convex, (A1') implies that ${\tilde X}$ is also
closed, bounded and convex. Moreover, we can show that 
assumptions (A1'), (A2'),
(A3'), (A4'), and (A5') imply that 
the objective function in  \eqref{compact}
is convex and has bounded subgradients:
\begin{lemma}\label{lemmaextpolyass} Consider the objective function 
$f(\tilde x)=c_1^\transp y_1 + \mathbb{E}\Big[ \mathcal{Q}\Big(y_1, g(x, \xi)\Big)\Big]$ of \eqref{compact} in variable ${\tilde x}=(y_1, x)$.
Assume that (A1'), (A2'), (A3'), (A4'), and (A5') hold. Then
\begin{itemize}
\item[(i)] $\mathcal{Q}\Big(y_1, g(x, \tilde \xi)\Big)$ is finite for every $\tilde \xi$ and every $\tilde x \in \tilde X =\{{\tilde x}=(y_1, x) : x \in X, A_1 y_1 \leq a_1 \}$;
\item[(ii)] for every $\tilde \xi \in \Xi$, the function
$\tilde x \rightarrow {\tilde{\mathcal{Q}}}_{\tilde \xi}(\tilde x)=\mathcal{Q}\Big(y_1, g(x, \tilde \xi)\Big)$ is convex and has bounded subgradients on $\tilde X$;
\item[(iii)] $f$ is convex and has bounded subgradients on $\tilde X$.
\end{itemize}
\end{lemma}
\begin{proof} Since (A1') holds, for every $y_1 \in Y_1$ and every $z \in \mathbb{R}$, the feasible
set of problem \eqref{newsecondstagepoly} which defines $\mathcal{Q}(y, z)$ is nonempty.
Due to (A2'), the feasible set of the dual of this problem is nonempty too. It follows that
both the primal and the dual have the same finite optimal value (this shows item (i)) and by duality we can
express $\mathcal{Q}(y_1, z)$ as the optimal value of the dual problem:
\begin{equation}\label{dualproblemq}
\mathcal{Q}(y_1, z)=\max_{(\lambda_1, \lambda_2) \in \mathcal{D}} \lambda_1^\transp (z k_2 + {\tilde k}_2 - B_{2, 1} y_1) + \lambda_2^\transp a_2
\end{equation}
with $\mathcal{D}$ given by  \eqref{dualfset}. Next, observe that $\mathcal{Q}(y_1, \cdot)$ is monotone:
\begin{equation} \label{qmonotone}
\forall y_1 \in Y_1,\;\forall z_1, z_2 \in \mathbb{R},\;z_1 \geq z_2 \Rightarrow \mathcal{Q}(y_1, z_1) \geq \mathcal{Q}(y_1, z_2).
\end{equation}
Indeed, if $z_1 \geq z_2$, for every $(\lambda_1, \lambda_2) \in \mathcal{D}$, since (A4') holds, we have
$\lambda_1^\transp k_2 \geq 0$ and 
$$
\lambda_1^\transp (z_1 k_2 + {\tilde k}_2 - B_{2, 1} y_1) + \lambda_2^\transp a_2 \geq \lambda_1^\transp (z_2 k_2 + {\tilde k}_2 - B_{2, 1} y_1) + \lambda_2^\transp a_2
$$
for every $y_1 \in Y_1$. Taking the maximum when $(\lambda_1, \lambda_2) \in \mathcal{D}$ in each side of the previous inequality gives
$\mathcal{Q}(y_1, z_1) \geq \mathcal{Q}(y_1, z_2)$. Now take $\tilde \xi$ a realization of $\xi$ and $\tilde x=(y_1, x), {\tilde x}_0=(y_1^{0}, x_0) \in \tilde X$.
Using the convexity of $g(\cdot, \tilde \xi)$, we have
$$
g(x, \tilde \xi) \geq g(x_0, \tilde \xi) +  G(x_0, \tilde \xi)^\transp (x-x_0)
$$
recalling that
$G(x_0, \xi)$ is a measurable selection of a stochastic subgradient of $g(\cdot, \xi)$ at $x_0$. Combining this inequality and
\eqref{qmonotone} gives
$$
{\tilde{\mathcal{Q}}}_{\tilde \xi}(\tilde x)=\mathcal{Q}
\Big(y_1, g(x, \tilde \xi)\Big) \geq \mathcal{Q} \Big(y_1,  g(x_0, \tilde \xi) + G(x_0, \tilde \xi)^\transp (x-x_0 )\Big)
$$
for every $y_1 \in Y_1$. Next, we have that $\mathcal{Q}(y_1, z)$ is convex and its subdifferential is given by
$$
\partial \mathcal{Q}(y_1, z)=\left\{\left(\begin{array}{cc}-B_{2, 1}^\transp \lambda_1\\\lambda_1^\transp k_2\end{array}\right) : (\lambda_1,\lambda_2) \in  \mathcal{D}_{y_1, z} \right\}
$$
where $\mathcal{D}_{y_1, z}$ is the set of optimal solutions to the dual problem \eqref{dualproblemq}.
Denoting by $(\lambda_1(y_1, z), \lambda_2(y_1, z))$ an optimal solution to \eqref{dualproblemq}, we then have
$$
{\tilde{\mathcal{Q}}}_{\tilde \xi}(\tilde x)=\mathcal{Q} \Big(y_1, g(x, \tilde \xi)\Big) \geq
{\tilde{\mathcal{Q}}}_{\tilde \xi}(\tilde x_0) + 
\left(
\begin{array}{c}
-B_{2, 1}^\transp \lambda_1(y_1^0, g(x_0, \tilde \xi))\\
\lambda_1(y_1^0, g(x_0, \tilde \xi))^\transp k_2 G(x_0, \tilde \xi)
\end{array}
\right)^\transp
\Big( \tilde x - \tilde x_0 \Big).
$$
It follows that for every $\tilde \xi$, ${\tilde{\mathcal{Q}}}_{\tilde \xi}(\cdot)$ is convex  and 
its subdifferential is given by
$$ 
\partial {\tilde{\mathcal{Q}}}_{\tilde \xi}(y_1^0, x_0) = 
\left\{\left(\begin{array}{cc}-B_{2, 1}^\transp \lambda_1\\\lambda_1^\transp k_2 G(x_0, \tilde \xi) \end{array}\right) : (\lambda_1,\lambda_2) \in \mathcal{D}_{y_1^0, \,g(x_0, \tilde \xi)} \right\}.
$$
Since $\mathcal{D}_{y_1^0, \,g(x_0, \tilde \xi)}$ is a subset of the bounded set $\mathcal{D}$
and since (A5') holds, all subgradients of ${\tilde{\mathcal{Q}}}_{\tilde \xi}(\cdot)$ are bounded for every $\tilde \xi \in \Xi$: we have proved (ii).
Item (iii) follows from (ii) and the fact that $f$ is finite in a neighborhood of $\tilde X$.\hfill
\end{proof}
It follows from Lemma \ref{lemmaextpolyass}-(iii) that Assumption 1 is satisfied.
We also have 
$\delta(\tilde x, \xi)=\mathcal{Q}_{\xi}(\tilde x)-\mathbb{E}[\mathcal{Q}_{\xi}(\tilde x)]$,
which is finite for every $\xi$ and $\tilde x \in \tilde X$ using Lemma \ref{lemmaextpolyass}-(i).
It follows that relations  \eqref{assw}(a), \eqref{asss}(a), and \eqref{asst}(a)
in respectively Assumptions 3, 4, and 5 are satisfied.
Finally Lemma \ref{lemmaextpolyass}-(ii) shows that relations  \eqref{assw}(b), \eqref{asss}(b), and \eqref{asst}(b)
in respectively Assumptions 3, 4, and 5 are also satisfied. This shows that we can use the developments of Sections 
\ref{rsa}, and  \ref{smd} to solve problem
\eqref{pbpartcase} and to obtain a confidence interval on its optimal value 
when $\mathcal{R}$ is an EPRM 
and when assumptions (A0'), (A1'), (A2'), (A3'), (A4'), and (A5') are satisfied.

Risk-averse stochastic programs expressed in terms of EPRMs
share many properties with risk-neutral stochastic programs. Moreover, many popular
risk measures can be written as EPRMs satisfying assumptions (A0'), (A1'), (A2'), (A3'), and (A4').
Examples of such risk measures are the CVaR, some spectral risk measures, the optimized certainty equivalent
and the expected utility with piecewise affine utility function.
We refer to Examples 2.16 and 2.17 in \cite{guiguesromisch10} for
a discussion on these examples. 
Conditions ensuring that an EPRM is
convex, coherent or consistent with second order stochastic dominance are
given in \cite{guiguesromisch10}.
Multiperiod versions of these risk measures are
also defined in \cite{guiguesromisch10}. 
In this context, a convenient property of the corresponding
risk-averse program is that we can
write dynamic programming equations and solve it,
in the case when the problem is convex, by decomposition using for instance Stochastic Dual Dynamic Programming (SDDP) \cite{pereira}; see \cite{guiguesromisch10}
for more details and examples of multiperiod EPRM. EPRM are an extension of the polyhedral risk measures
introduced in \cite{rom2} where the reader will find additional examples of (extended) polyhedral  risk measures.

Throughout the paper, we will use two (classes of) problems of form \eqref{pbpartcase} for which we will detail the computation of the
parameters necessary to obtain the confidence intervals on their optimal value given in Sections \ref{riskneutral} and \ref{mssmd}, in particular parameters $L, M_1$, and $M_2$ introduced 
in Section \ref{generalassumptions}. These problems are described in the next section.

\subsection{Examples}

We provide two classes of problems that will be used to illustrate our results.

\begin{enumerate}
\item The first class of problems writes
\begin{equation}\label{definstance1}
\left\{
\begin{array}{l}
\min f(x)=\mathbb{E}\Big[  \alpha_0 \xi^\transp x  + {\alpha_1\over2}\left( \left( {\xi^\transp x} \right)^2 + \lambda_0 \|x\|_2^2 \right)    \Big] \\
x \in X:=\{x \in \mathbb{R}^n : \sum_{i=1}^n x(i) = a, \;x( i ) \geq b, i=1,\ldots,n\}, 
\end{array}
\right.
\end{equation}
where $n \geq 3$, $\alpha_1, a > 0$, $b, \lambda_0 \geq 0$, with  $b<a/n$, and the support $\Xi$ of $\xi$ is a part of the unit box
$\{\xi=[\xi(1);...;\xi(n)]\in \mathbb{R}^n:\|\xi\|_\infty\leq 1\}$.\footnote{If $b=a/n$ then there is only one feasible point given by
$x_i= b, i=1,\ldots,n$, while if $b>a/n$ the problem is not feasible.} 

If $a=1$ and $b=0$, taking $\|\cdot\|=\|\cdot\|_1$, $\|\cdot\|_* =\|\cdot\|_{\infty}$,
straightforward computations (see \cite{ioudnemgui15}) show that 
Assumptions 1-5 are satisfied for this problem with 
$L=|\alpha_0|+\alpha_1(1+ \lambda_0)$, $M_1=2|\alpha_0|+0.5\alpha_1$, and $M_2=2|\alpha_0| + \alpha_1$.

If $a=1$ and $b=0$, taking $\|\cdot \| = \| \cdot \|_2 = \| \cdot  \|_*$, and $G(x,\xi)=\alpha_0 \xi + \alpha_1 (\xi \xi^\transp + \lambda_0 I ) x$,
we have for every $x \in X$ that
$$
\begin{array}{lll}
\|G(x,\xi) - \mathbb{E}[G(x,\xi)]\|_2   &\leq &  |\alpha_0| \|\xi - \mathbb{E}[\xi]   \|_2 + \alpha_1 \|(\xi \xi^\transp - \mathbb{E}[\xi \xi^\transp ] )x\|_2 \\
&\leq & 2 |\alpha_0| \sqrt{n} + \alpha_1  \sqrt{n}  \|\xi \xi^\transp - \mathbb{E}[\xi \xi^\transp ] \|_{\infty} \leq  2 \sqrt{n} ( |\alpha_0| + \alpha_1 ),
\end{array}
$$
$$
\begin{array}{lll}
\|\mathbb{E}[G(x,\xi)]\|_2   &\leq &  |\alpha_0 | \|\mathbb{E}[\xi]\|_2 + \alpha_1 \sqrt{n} \|\mathbb{E}[\xi \xi^\transp] x \|_{\infty}  +  \alpha_1\lambda_0 \|x\|_{1} \\
& \leq & |\alpha_0 | \sqrt{n} + \alpha_1 (\sqrt{n} +  \lambda_0 ),
\end{array}
$$
and Assumptions 1 and 5 hold with $L=|\alpha_0 | \sqrt{n} + \alpha_1 (\sqrt{n} +  \lambda_0 )$, $M_1=2|\alpha_0|+0.5\alpha_1$, and $M_2=2 \sqrt{n} ( |\alpha_0| + \alpha_1 ) $.
\item The second class of problems amounts to minimizing a linear combination of the expectation and the CVaR of some random linear function:
\begin{equation}\label{definstance2}
\left\{
\begin{array}{l}
\min f(x)=\alpha_0 \mathbb{E}[ \xi^\transp x ]  + {\alpha_1} CVaR_{\varepsilon}( \xi^\transp x )  \\
\sum_{i=1}^n x(i) = 1, \;x \geq 0, 
\end{array}
\right.
\end{equation}
where $\alpha_1, \alpha_0 \geq 0$, $0<\varepsilon<1$, 
the support $\Xi$ of $\xi$ is a part of the unit box
$\{\xi=[\xi(1);...;\xi(n)]\in \mathbb{R}^n:\|\xi\|_\infty\leq 1\}$,
 and 
 \[
 {\rm CVaR}_\varepsilon( \xi^\transp x)=\min_{x_0 \in \mathbb{R}}\, x_0 + \mathbb{E}\left[\varepsilon^{-1}[\xi^\transp x-x_0]^+\right]
 \]
is the Conditional Value-at-Risk of level $0<\varepsilon<1$; see \cite{ury2}.
Observing that $|\xi^\transp x| \leq 1$ a.s., problem \eqref{definstance2} is of form \eqref{defpbriskneutral} with
$X=\{x=[x(1);...;x(n);x(n+1)]\in\mathbb{R}^{n+1}: \;|x(n+1)|\leq 1,\, x(1),...,x(n)\geq 0,\,\sum_{i=1}^{n} x(i)=1\}$
and
\[
g(x,\xi)=\alpha_0\xi^\transp  [x(1);...;x(n)] + \alpha_1\left(x(n+1)+{1\over\epsilon}[\xi^\transp [x(1);...;x(n)]-x(n+1)]^+ \right).
\]
We will also consider a perturbed version of this problem given by
\begin{equation}\label{definstance3}
\left\{
\begin{array}{l}
\min \alpha_0 \mathbb{E}[ \xi^\transp x_{1:n}]  + {\alpha_1}\left( x(n+1) + \mathbb{E}\left[ \varepsilon^{-1}[\xi^\transp x_{1:n} -x(n+1)]^+  \right] \right) + \lambda_0 \|x_{1:n+1}\|_2^2 \\
-1 \leq x(n+1) \leq 1, \sum_{i=1}^n x(i) = 1, \;x(i) \geq 0, i=1,\ldots,n, 
\end{array}
\right.
\end{equation}
for $\lambda_0>0$ where $x_{1:n}=[x(1);...;x(n)]$.
For problem \eqref{definstance3}, taking $\|\cdot \| = \| \cdot \|_2 = \| \cdot  \|_*$, Assumptions 1 and 5 are satisfied (see \cite{ioudnemgui15})
with 
$L=\sqrt{\alpha_1^2 (1-\frac{1}{\varepsilon})^2   + n (\alpha_0 + \frac{\alpha_1}{\varepsilon})^2}+2 \lambda_0$,
$M_1=2(\alpha_0 + \frac{\alpha_1}{\varepsilon})$, and
$M_2=\sqrt{\left(\frac{\alpha_1}{\varepsilon}\right)^2  + 4n \left( \alpha_0 + \frac{\alpha_1}{\varepsilon} \right)^2 }$.

\end{enumerate}

Problems \eqref{definstance1} and \eqref{definstance3} have a penalty term in the objective to make the objective function 
strongly convex so that multistep SMD, as described in Section \ref{mssmd}, can be applied to these problems.

\section{Quality of the solutions using RSA and SMD} \label{riskneutral}

We consider the RSA and SMD algorithms to solve problem \eqref{defpbriskneutral}.

\subsection{Robust Stochastic Approximation algorithm}\label{rsa}

In this section, we use the scalar product $\langle x, y \rangle = x^\transp y$ 
and the corresponding norm $\|\cdot\|=\|\cdot\|_2$ with dual norm
$\| \cdot \|_{*}=\| \cdot \|_2$, meaning that
\eqref{assw}, \eqref{asss}, and \eqref{asst} hold  with $\|\cdot\|_* = \|\cdot\|_2$.
The Robust Stochastic Approximation algorithm solves \eqref{defpbriskneutral} as follows:\vspace*{0.2cm}\\
\rule{\linewidth}{1pt}
\par {\textbf{Algorithm 1: Robust Stochastic Approximation.}} \\
\par {\textbf{Initialization.}} Take $x_1$ in $X$. Fix the number of iterations $N-1$ and positive
deterministic stepsizes $\gamma_1, \ldots, \gamma_N$.
\par {\textbf{Loop.}} For $t=1,\ldots,N-1$, compute
\begin{equation}\label{rsaiterations}
\begin{array}{rcl}
x_{t+1}&=&\Pi_{X}(x_{t}-\gamma_t G(x_{t},\xi_t)).
\end{array}
\end{equation}
\par {\textbf{Outputs:}} 
\begin{equation}\label{outputsrsa}
\begin{array}{l}
x^N = \displaystyle \frac{1}{\Gamma_N} \sum_{\tau=1}^N \gamma_\tau x_\tau  \mbox{ and }g^N=\displaystyle \frac{1}{\Gamma_N}\left[ \sum_{\tau=1}^N  \gamma_\tau g(x_\tau,\xi_\tau)\right] \mbox{ with }\Gamma_N=\displaystyle \sum_{\tau=1}^N \gamma_\tau.
\end{array}
\end{equation}
\rule{\linewidth}{1pt}

Note that by convexity of $X$, we have $x^N \in X$ and
after $N-1$ iterations, $x^N$ is an approximate solution of \eqref{defpbriskneutral}.
The value $f(x^N) $ is an approximation of the optimal value
of \eqref{defpbriskneutral}, but it is not computable since $f$ is not known.
Denoting by $x_*$ an optimal solution of \eqref{defpbriskneutral}, we
introduce after $N-1$ iterations
the computable approximation\footnote{Note that the approximation depends on $(x_1,\ldots,x_N,\xi_1,\ldots,\xi_N, \gamma_1, \ldots, \gamma_N)$
so we could write $g^N(x_1,\ldots,x_N,\xi_1,\ldots,\xi_N, \gamma_1, \ldots, \gamma_N)$ but we choose, for the moment, to suppress this dependence to alleviate notation}
\begin{equation}\label{approxopt}
g^N=\frac{1}{\Gamma_N}\left[\displaystyle \sum_{\tau=1}^N  \gamma_\tau g(x_\tau,\xi_\tau)\right]
\end{equation}
of the optimal value $f(x_*)$ of \eqref{defpbriskneutral} obtained using the points generated by the algorithm
and information from the stochastic oracle.
Our goal is to obtain exponential bounds on large deviations of this estimate $g^N$ of $f(x_{*})$ from $f(x_{*})$ itself, i.e.,
a confidence interval on the optimal value of \eqref{defpbriskneutral} using the information provided by the RSA algorithm along iterations.
We need two technical lemmas. The first one gives an $O(1/\sqrt{N})$ upper bound on the
first
absolute moment of the estimation error (the average distance of $g^N$ to $f(x_*)$):
\begin{lemma} \label{lemmersa1} Let Assumptions 1, 2, and 3 hold and
assume that the number of iterations $N-1$ of the RSA algorithm is fixed in advance with stepsizes given by
\begin{equation}\label{stepsizes}
\gamma_\tau=\gamma = \frac{D_X}{\sqrt{2(M_2^2 + L^2)} \sqrt{N}},\;\tau=1,\ldots,N,
\end{equation}
where
\begin{equation}\label{maxdisttox1}
D_X=\max_{x \in X}\;\|x-x_1\|.
\end{equation}
Let $g^N$ be the approximation of $f(x_*)$ given by \eqref{approxopt}.
Then
\begin{equation} \label{upperboundmeanrsa}
\mathbb{E}\left[ \Big| g^N - f( x_*) \Big|  \right] \leq \frac{ M_1+ D_X \sqrt{2(M_2^2 + L^2)} }{\sqrt{N}}.
\end{equation}
\end{lemma}
\begin{proof} Recalling \eqref{stepsizes}, $\frac{\gamma_\tau}{\Gamma_N}=\frac{1}{N}$ and letting
\begin{equation} \label{defft}
f^N=\frac{1}{N} \displaystyle \sum_{\tau=1}^N f (x_\tau ),
\end{equation}
it is known (see \cite{nemjudlannem09}, Section 2.2) that under our assumptions
\begin{equation} \label{firstupperboundmeanrsa}
\bE\left[ f(x^N)-f(x_{*}) \right] \leq \bE\left[ f^N -f(x_{*}) \right] \leq \frac{D_X \sqrt{2(M_2^2 + L^2)}}{\sqrt{N}}.
\end{equation}
Since the main steps of the proof of \eqref{firstupperboundmeanrsa} will be useful for our further developments,
we rewrite them here.
Setting $A_\tau=\frac{1}{2}\|x_{\tau}-x_{*}\|_2^2$, we can show (see Section 2.1 in \cite{nemjudlannem09} for instance) that
\begin{equation} \label{rsaeq1}
\sum_{\tau=1}^{N} \; \gamma_{\tau} \langle G(x_{\tau}, \xi_{\tau}), x_{\tau}-x_{*} \rangle  \leq
A_1 + \frac{1}{2} \sum_{	\tau=1}^{N} \; \gamma_{\tau}^2 \| G(x_{\tau}, \xi_{\tau})\|_*^2.
\end{equation}
To save notation, let us set
\begin{equation} \label{shortnotation}
\delta_\tau=g(x_\tau,\xi_\tau)-f(x_\tau),\,\Delta_\tau=\Delta(x_\tau,\xi_\tau),\mbox{ and } G_\tau=G(x_\tau,\xi_\tau)=f'(x_\tau)+\Delta_\tau.
\end{equation}
Inequality \eqref{rsaeq1} can be rewritten
\begin{equation} \label{rsaeq2}
\sum_{\tau=1}^{N} \; \gamma_{\tau} \langle  f'(x_{\tau}), x_{\tau}-x_{*} \rangle  \leq \frac{D_X^2}{2} + \frac{1}{2} \sum_{\tau=1}^{N} \; \gamma_{\tau}^2 \| G_{\tau}\|_*^2  +\sum_{\tau=1}^{N} \; \gamma_{\tau} \langle \Delta_{\tau}, x_{*}-x_{\tau} \rangle.
\end{equation}
Taking into account that by convexity of $f$ we have $f(x_\tau)-f(x_*)\leq\langle f'(x_\tau),x_\tau-x_*\rangle$,
we get
\begin{equation} \label{rsaeq3}
\begin{array}{rcl}
f(x^N)-f(x_{*}) &\leq& f^N -f(x_{*}) = \displaystyle\frac{1}{\Gamma_N} \displaystyle \sum_{\tau=1}^N \gamma_\tau \Big(f(x_\tau)-f(x_*)\Big) \\
&\leq&\displaystyle  \frac{1}{\Gamma_N} \left[ \displaystyle \frac{D_X^2}{2} + \frac{1}{2} \displaystyle \sum_{\tau=1}^{N} \; \gamma_{\tau}^2 \| G_{\tau}\|_*^2  + \displaystyle \sum_{\tau=1}^{N} \; \gamma_{\tau} \langle \Delta_{\tau}, x_{*}-x_{\tau} \rangle\right]
\end{array}
\end{equation}
where the first inequality is due to the origin of $x^N$ and to the convexity of $f$.

Next, note that under Assumptions 1, 2, and 3,
\begin{equation} \label{bornesupgtau}
\mathbb{E}\Big[\|G_\tau\|_*^2\Big]=
\mathbb{E}\Big[\|f'(x_\tau)+\Delta_\tau\|_*^2\Big] \leq
2 \mathbb{E}\Big[\|f'(x_\tau)\|_*^2+\|\Delta_\tau\|_*^2 \Big] \leq 2 \Big[M_2^2 + L^2 \Big].
\end{equation}
Passing to expectations in \eqref{rsaeq3}, and taking into account that the conditional,
 $\xi^{\tau-1}:=(\xi_1,...,\xi_{\tau-1})$ being fixed,
 expectation of $\Delta_\tau$ is zero, while $x_\tau$ by construction is a deterministic function of $\xi^{\tau-1}$, we get
\begin{eqnarray}
\bE\Big[ f(x^N)-f(x_{*}) \Big] &\leq& \bE \Big[ f^N -f(x_{*}) \Big]  \leq {D_X^2 + \displaystyle \sum_{\tau=1}^N \gamma_\tau^2\bE \left[ \|G_\tau\|_*^2 \right] \over 2\Gamma_N }\nonumber \\
& \leq & \displaystyle \frac{1}{\Gamma_N }\Big[ \frac{D_X^2}{2} + (M_2^2 + L^2) \sum_{\tau=1}^N \gamma_\tau^2 \Big].  \label{eq102}
\end{eqnarray}
Using stepsizes \eqref{stepsizes}, we have $\Gamma_N={{D_X \sqrt{N}} \over {\sqrt{2(M_2^2 + L^2)}}}$. Plugging this value
of $\Gamma_N$ into \eqref{eq102}, we obtain the announced inequality \eqref{firstupperboundmeanrsa}.

We now show that
\begin{equation} \label{secondupperboundmeanrsa}
\bE\Big[ \Big| g^N - f^N \Big| \Big] \leq  \frac{M_1}{\sqrt{N}}.
\end{equation}
First, note that
\begin{equation}\label{gtft}
g^N-f^N={1\over N} \displaystyle \sum_{\tau=1}^N \delta_\tau.
\end{equation}
By the same argument as above, the conditional, $\xi^{\tau-1}$ being fixed, expectation of $\delta_\tau$ is $0$, whence
$$
\bE\Big[ \Big(\displaystyle \sum_{\tau=1}^N \delta_\tau\Big)^2\Big] =\sum_{\tau=1}^N \bE \Big[ \delta_\tau^2 \Big] \leq N M_1^2,
$$
where the concluding inequality is due to \eqref{assw}($a$). We conclude that
$$
\bE \Big[ \Big|g^N-f^N\Big| \Big] \leq {1\over N}\sqrt{\bE\left[ \left(\displaystyle  \sum_{\tau=1}^N \delta_\tau\right)^2\right]}\leq {1\over N}\sqrt{N M_1^2}={M_1\over \sqrt{N}},
$$
which is the announced inequality \eqref{secondupperboundmeanrsa}. Next, observe that by convexity of $f$,
$f^N \geq f(x^N)$ and since $x^N \in X$, we have $f(x^N) \geq f(x_{*})$, i.e.,
$f^N-f(x_{*}) \geq f(x^N)-f(x_{*}) \geq 0$, so that (\ref{firstupperboundmeanrsa}) and (\ref{secondupperboundmeanrsa}) imply
\begin{eqnarray*}
\bE \Big[ |g^N-f( x_* )| \Big] &\leq &\bE \Big[ |g^N-f^N| + |f^N - f( x_* ) | \Big]=\bE \Big[ |g^N-f^N| \Big] + \bE \Big[ f^N - f( x_* )  \Big] \\
& \leq &  \Big[ M_1+ D_X \sqrt{2(M_2^2 + L^2)} \Big] \frac{1}{\sqrt{N}},
\end{eqnarray*}
which achieves the proof of \eqref{upperboundmeanrsa}.\hfill
\end{proof}
To proceed, we need the following lemma:
\begin{lemma}\label{lemmalargedevexp} Let $\xi_1, \ldots, \xi_N$ be random vectors and associated
sigma algebras $\mathcal{F}_\tau = \sigma(\xi_1, \ldots, \xi_\tau ), \tau=1,\ldots,N$.
Let  $\eta_\tau, \tau=1,\ldots,N$, be a sequence of real-valued random variables
with $\eta_\tau$ $\mathcal{F}_\tau$-measurable.
Let $\bE_{|\tau-1}\left[ \cdot \right]$ be the conditional expectation $\bE \left[ \cdot | \xi^{\tau - 1}  \right]$
where $\xi^{\tau - 1}=(\xi_1,\ldots,\xi_{\tau -1})$.
Assume that
\begin{equation} \label{martin}
\bE_{|\tau-1}\Big[ \eta_\tau \Big] = 0,\;\;\bE_{|\tau-1}\Big[ \exp\{ \eta_\tau^2 \} \Big] \leq \exp\{ 1\}.
\end{equation}
Then, for any $\Theta>0$,
\begin{equation}\label{neqlemma2}
\begin{array}{ll}
 \mathbb{P}\left( \displaystyle \sum_{\tau=1}^N \eta_\tau >\Theta \, \sqrt{N} \right) \leq \exp\{-\Theta^2/4\}.
\end{array}
\end{equation}
\end{lemma}
\begin{proof} See the Appendix.\hfill
\end{proof}
We are now in a position to provide a confidence interval for the optimal value of \eqref{defpbriskneutral} using the RSA
algorithm:
\begin{prop} \label{confidenceintervalRSA} Assume
that the number of iterations $N-1$ of the RSA algorithm is fixed in advance with stepsizes given by
\eqref{stepsizes}. Let $g^N$ be the approximation of $f(x_*)$ given by \eqref{approxopt}.
Then
\begin{itemize}
\item[(i)] if Assumptions 1, 2, 3, and 4 hold, for any $\Theta>0$, we have
\begin{equation}\label{eqB3}
\mathbb{P}\left(\Big|g^N-f(x_{*})\Big|>\frac{K_1(X) + \Theta K_2(X)}{\sqrt{N}} \right)\leq 4 \exp\{1\}\exp\{-\Theta\}
\end{equation}
where the constants $K_1(X)$ and $K_2(X)$ are given by
$$
K_1(X)=\frac{D_X (M_2^2+2L^2)}{\sqrt{2 (M_2^2 + L^2 )}}
\mbox{ and }K_2(X)=\frac{D_X M_2^2}{\sqrt{2 (M_2^2 + L^2)}}+2 D_X M_2 + M_1,
$$
with $D_X$ given by \eqref{maxdisttox1}.
\item[(ii)] If Assumptions 1, 2, 3, and 5 hold,
\eqref{eqB3} holds with the right-hand side replaced by $(3+\exp\{1\}) \exp\{-\frac{1}{4} \Theta^2 \}$.
\end{itemize}
\end{prop}
\begin{proof} To prove (i), we shall first prove that for any $\Theta>0$,
\begin{equation} \label{firstconfintrsa}
\begin{array}{l}
 \mathbb{P}\left( f^N-f(x_{*}) > \frac{D_X}{\sqrt{2 (M_2^2 + L^2 ) \, N}} \left[M_2^2+2L^2 +\Theta\left[ M_2^2 + 2 M_2 \sqrt{2 ( M_2^2 + L^2)} \right] \right] \right)\\
\leq   2 \exp\{1\} \exp\{-\Theta\},
\end{array}
\end{equation}
where $f^N$ is given by \eqref{defft}.
Using Assumption 1, we have
$\|G_\tau\|_*^2=\|f'(x_\tau)+\Delta_\tau\|_*^2 \leq 2(\|f'(x_\tau)\|_*^2+\|\Delta_\tau\|_*^2) \leq
2(L^2+\|\Delta_\tau\|_*^2)$. Combined with \eqref{rsaeq3}, this implies that
\begin{equation}\label{eq110}
\begin{array}{rcl}
f^N-f(x_{*})& \leq  & \displaystyle \frac{1}{\Gamma_N}\left[\frac{D_X^2}{2}+\displaystyle \sum_{\tau=1}^N\gamma_\tau^2 \Big( L^2 + \|\Delta_\tau\|_*^2 \Big) \right]+ \frac{1}{\Gamma_N} \displaystyle \sum_{\tau=1}^N \gamma_\tau \langle\Delta_\tau,x_*-x_\tau\rangle \vspace*{0.1cm}\\
& \leq & \displaystyle \frac{D_X (M_2^2 + 2 L^2)}{\sqrt{2(M_2^2 + L^2)} \sqrt{N} } + \frac{D_X M_2^2}{\sqrt{2(M_2^2 + L^2)} \sqrt{N} } \cA  + \frac{2 D_X M_2}{N} \cB
\end{array}
\end{equation}
where
\begin{equation} \label{defAB}
\cA = {1\over N M_2^2} \displaystyle \sum_{\tau=1}^N\|\Delta_\tau\|_*^2 \;\;\mbox{  and  }\;\;\cB =  {1\over 2 D_X M_2} \displaystyle  \sum_{\tau=1}^N\langle \Delta_\tau,x_*-x_\tau\rangle.
\end{equation}
Setting $\zeta_\tau=\|\Delta_\tau\|_*^2/M_2^2$ and invoking \eqref{asss}($b$), we get
$\bE\Big[ \exp\{\zeta_\tau\}\Big] \leq \exp\{1\}$ for all $\tau\leq N$, whence,
due to
the convexity of the exponent,
$$
\bE\Big[\exp\{\cA\}\Big]=\bE\Big[\exp\{{1\over N}\sum_{\tau=1}^N\zeta_\tau\}\Big] \leq \frac{1}{N} \sum_{\tau=1}^{N}  \bE\Big[\exp\{ \zeta_\tau \}\Big] \leq   \exp\{1\}
$$ as well. As a result,
\begin{equation}\label{resultA}
\forall \, \Theta>0: \mathbb{P}\Big(\cA>\Theta\Big) \leq \exp\{-\Theta\}\bE\Big[\exp\left\{\cA\right\}\Big] \leq \exp\{1-\Theta\}.
\end{equation}
Now let us set $\eta_\tau={1\over 2 D_X M_2}\langle\Delta_\tau,x_*-x_\tau\rangle$, so that $\cB=\sum_{\tau=1}^N\eta_\tau$. Denoting by $\bE_{|\tau-1}$ the conditional,
$\xi^{\tau-1}$ being fixed, expectation, we have
$$
\bE_{|\tau-1}\Big[ \eta_\tau \Big] =0\ \;\mbox{ and }\;\ \bE_{|\tau-1}\Big[ \exp\{\eta_\tau^2\}\Big]  \leq\exp\{1\},
$$
where the first relation is due to $\bE_{|\tau-1}\Big[\Delta_\tau \Big]=0$ combined with the fact that $x_*-x_\tau$ is a deterministic function
of $\xi^{\tau-1}$,
and the second relation
is due to \eqref{asss}($b$) combined with the fact that $\|x_*-x_\tau\|\leq 2 D_X$.
Using Lemma \ref{lemmalargedevexp}, we obtain for any $\Theta>0$
\begin{equation}\label{neq13}
\begin{array}{ll}
\mathbb{P}\Big( \cB>\Theta \, \sqrt{N} \Big) \leq\exp\{-\Theta^2/4\}.
\end{array}
\end{equation}
Combining  \eqref{eq110}, \eqref{resultA}, and \eqref{neq13}, we obtain
for every $\Theta>0$
\begin{equation}\label{wegetB}
\begin{array}{l}
\mathbb{P}\left( f^N-f(x_{*})> \displaystyle \frac{D_X (M_2^2 + 2 L^2)}{\sqrt{2 (M_2^2 + L^2 ) \, N}} + \frac{\Theta}{\sqrt{N}} \left[ \frac{D_X M_2^2}{\sqrt{2 ( M_2^2 + L^2)}} + 2D_X M_2 \right] \right)\\
\leq \exp\{1-\Theta\}+\exp\{-\Theta^2/4\} \leq  2 \exp\{1\} \exp\{-\Theta\},
\end{array}
\end{equation}
which is \eqref{firstconfintrsa}.

Next,
$$
g^N -f^N={M_1 \over N}\left[\sum_{\tau=1}^N \chi_\tau\right],\,\,\chi_\tau={\delta_\tau\over M_1}.
$$
Observing that $\chi_\tau$ is a deterministic function of $\xi^\tau$ and that
$$
\bE_{|\tau-1}\Big[ \chi_\tau \Big]=0\ \;\mbox{ and }\;\ \bE_{|\tau-1}\Big[ \exp\{\chi_\tau^2 \}\Big] \leq\exp\{1\},\,1\leq\tau\leq N
$$ (we have used \eqref{asss}($a$)), we can use once again Lemma \ref{lemmalargedevexp} to obtain
for all $\Theta>0$:
$$
\mathbb{P}\left( g^N-f^N > \Theta {M_1 \over\sqrt{N}}\right) \leq \exp\{-\Theta^2/4\}
$$
and
$$
\mathbb{P} \left( g^N-f^N < -\Theta {M_1 \over\sqrt{N}}\right) \leq \exp\{-\Theta^2/4\}.
$$
Thus,
$$
\forall \; \Theta>0 : \mathbb{P}\left( \Big|g^N-f^N \Big|>\Theta{M_1 \over\sqrt{N}}\right) \leq 2\exp\{-\Theta^2/4\},
$$
which, combined with \eqref{firstconfintrsa} implies \eqref{eqB3}, i.e., item (i) of the lemma.

Finally, under Assumption 5, we have $\mathbb{P}\Big(\cA>1\Big)=0$, which combines
with (\ref{resultA}) to imply that
$$
\forall \Theta>0: \mathbb{P}\Big( \cA>\Theta \Big)  \leq \exp\{1-\Theta^2\},
$$
meaning that the right-hand side in (\ref{firstconfintrsa}) can be replaced with $\exp\{1-\Theta^2\}+\exp\{-\Theta^2/4\}$,
which proves item (ii).\hfill
\end{proof}
Setting 
\begin{equation}\label{defab}
a(\Theta, N) =\frac{\Theta M_1}{\sqrt{N}} \mbox{ and }b(\Theta,X,N)=\frac{K_1(X) + \Theta(K_2(X)-M_1)}{\sqrt{N}},
\end{equation}
we now combine the upper bound on $f(x_* )$
\begin{equation} \label{upperbound}
{\tt{Up}}_1(\Theta_1, N) = \frac{1}{N} \sum_{t=1}^N g(x_t, \xi_t) + a(\Theta_1, N)= g^N + a(\Theta_1, N),
\end{equation}
from \cite{nemlansh09} with the lower bound
\begin{equation}\label{lowerbound}
{\tt{Low}}_1(\Theta_2, \Theta_3, N) = g^N - b(\Theta_2, X, N)-a(\Theta_3, N),
\end{equation}
from Proposition \ref{confidenceintervalRSA} to obtain a new confidence interval on the optimal value $f(x_*)$:
\begin{cor}\label{corCIRSA} Let ${\tt{Up}}_1$ and ${\tt{Low}}_1$ be the upper and lower bounds given by respectively \eqref{upperbound} and
\eqref{lowerbound}. Then if Assumptions 1, 2, 3, and 5 hold, for any $\Theta_1, \Theta_2, \Theta_3>0$, we have
\begin{equation}\label{confintRSA}
\mathbb{P}\left(f(x_*) \in \Big[{\tt{Low}}_1(\Theta_2, \Theta_3, N) , {\tt{Up}}_1(\Theta_1, N)\Big]\right) \geq 1-e^{-\Theta_1^2/4}-e^{1-\Theta_2^2} - e^{-\Theta_2^2/4} - e^{-\Theta_3^2/4}.   
\end{equation}
If Assumptions 1, 2, 3, and 4 hold, then \eqref{confintRSA} holds with the term   $e^{1-\Theta_2^2}$ replaced by $e^{1-\Theta_2}$.
\end{cor}
\begin{proof}
Let Assumptions 1, 2, 3, and 5 hold.
Since $f(x_t) \geq f(x_*)$ almost surely, using Lemma \ref{lemmalargedevexp} we get
$$
\mathbb{P}\Big({\tt{Up}}_1(\Theta_1,N)  < f(x_* )   \Big) \leq 
\mathbb{P} \Big( \frac{1}{N}\sum_{t=1}^{N} \Big[ g(x_t, \xi_t )-f(x_t) \Big] <  - \frac{\Theta_1 M_1}{\sqrt{N}}   \Big) \leq
e^{-\Theta_1^2 /4}.
$$
Next, using the proof of Proposition \ref{confidenceintervalRSA}, we can define sets $S_1, S_2 \subset \Omega$ such that
under Assumptions 1, 2, 3, and 5 we have
$\mathbb{P}(S_1)\geq  1-e^{1-\Theta_2^2} - e^{-\Theta_2^2/4}$ (resp. $\mathbb{P}(S_2)\geq  1- e^{-\Theta_3^2/4}$)
and on $S_1$ (resp. on $S_2$) we have $f^ N -b(\Theta_2, X, N) \leq f(x_*)$ (resp. $g^N - f^ N \leq  a(\Theta_3,N) $).  
Now observe that on $S_1 \cap S_2$ we have $f(x_* ) \geq  {\tt{Low}}_1(\Theta_2, \Theta_3, N)$ which implies that
$$
\mathbb{P}(f(x_* ) \geq {\tt{Low}}_1(\Theta_2, \Theta_3, N)) \geq \mathbb{P}(S_1 \cap S_2) \geq 1-e^{1-\Theta_2^2} - e^{-\Theta_2^2/4} - e^{-\Theta_3^2/4}
$$
and \eqref{confintRSA} follows.\hfill
\end{proof}
\begin{rem}\label{choiceparametersCI} Let Assumptions 1, 2, 3, and 5 hold.
To equilibrate the risks, for the confidence interval 
$\Big[{\tt{Low}}_1(\Theta_2, \Theta_3, N) , {\tt{Up}}_1(\Theta_1, N)\Big]$ on $f( x_* )$ to have confidence level at least $0<1-\alpha<1$, we can take
$\Theta_1$ such that $e^{-\Theta_1^2/4} = \alpha/2$, i.e., $\Theta_1=2\sqrt{\ln(2/\alpha)}$, $\Theta_3$ such that $e^{-\Theta_3^2/4} = \alpha/4$, i.e., $\Theta_3=2\sqrt{\ln(4/\alpha)}$,
and compute by dichotomy $\Theta_2$ such that $e^{1-\Theta_2^2} + e^{-\Theta_2^2/4}=\frac{\alpha}{4}$.
\end{rem}
\begin{rem} If an additional sample $\bar \xi^{\tilde N}=(\bar \xi_1,\ldots, \bar \xi_{\tilde N})$ 
independent on $\xi^{N}=(\xi_1,\ldots, \xi_{N})$
is available, we can use the upper bound 
${\tt{Up}}_2(\Theta_1, N, \tilde N) = \frac{1}{\tilde N} \sum_{t=1}^{\tilde N} g(x^N, {\bar \xi}_t) + a(\Theta_1, {\tilde N})$
with $x^N$ given by \eqref{outputsrsa}, see \cite{nemlansh09}.
\end{rem}

\subsection{Stochastic Mirror Descent algorithm} \label{smd}

The algorithm to be described, introduced in \cite{nemjudlannem09}, is given by a {\sl proximal setup}, that is, by a norm $\|\cdot\|$ on $E$ and a {\sl distance-generating function} $\omega(x):X\to \mathbb{R}$.
This function should
\begin{itemize}
\item be convex and continuous on $X$,
\item admit on $X^o=\{x\in X:\partial \omega(x)\neq\emptyset\}$ a selection $\omega'(x)$ of subgradients, and
\item be compatible with $\|\cdot\|$, meaning that $\omega(\cdot)$ is strongly convex, modulus $\mu(\omega)>0$, with respect to
the norm $\|\cdot\|$:
$$
( \omega'(x)-\omega'(y))^\transp (x-y) \geq \mu( \omega) \|x-y\|^2\,\,\forall x,y\in X^o.
$$
\end{itemize}
The proximal setup induces the following entities:
\begin{enumerate}
\item the {\sl $\omega$-center of $X$} given by $x_\omega=\displaystyle \argmin_{x \in X} \, \omega(x)\in X^o$;
\item the {\sl Bregman distance} or prox-function
\begin{equation}\label{strong}
V_x(y)=\omega(y)-\omega(x)- (y-x)^\transp \omega'(x) \geq {\mu( \omega ) \over 2}\|x-y\|^2,
\end{equation}
for $x\in X^o$, $y\in X$ (the concluding inequality is due to the strong convexity of $\omega$);
\item the {\sl $\omega$-radius of $X$} defined as
\begin{equation}\label{defDomega}
D_{\omega, X}=\sqrt{2\Big[\max_{x \in X}\omega(x)-\min_{x \in X} \omega(x)\Big]}.
\end{equation}
Since $(x-x_\omega)^\transp \omega'(x_\omega) \geq0$ for all $x\in X$, we have
\begin{equation}\label{note}
\begin{array}{lcl}
\forall x\in X: {\mu( \omega )\over 2}\|x-x_\omega\|^2  & \leq  & V_{x_\omega}(x)=\omega(x)-\omega(x_\omega)-\underbrace{ (x-x_\omega)^\transp  \omega'(x_\omega)}_{\geq 0}\\
& \leq & \omega(x)-\omega(x_\omega)\leq {1\over 2} D_{\omega, X}^2,
\end{array}
\end{equation}
and
\begin{equation}\label{eq447}
\forall x\in X: \|x-x_\omega\| \leq \frac{D_{\omega, X}}{\sqrt{\mu(\omega ) }}.
\end{equation}
\item {\sl The proximal mapping}, defined by
\begin{equation} \label{defprox}
\Prox_x(\zeta)=\argmin_{y\in X} \{\omega(y)+y^\transp (\zeta-\omega'(x)) \} \;\;\; [x\in X^o,\zeta\in E],
\end{equation}
takes its values in $X^o$.\\

Taking $x_+=\Prox_x(\zeta)$, the optimality conditions for the optimization problem
$\min_{y\in X}\{\omega(y)+ y^\transp (\zeta-\omega'(x))\}$ in which $x_+$ is the optimal solution read
$$
\forall y\in X:  (y-x_+ )^\transp  (\omega'(x_+)+\zeta-\omega'(x)) \geq 0.
$$
Rearranging the terms, simple arithmetics show that this condition can be written equivalently as
\begin{equation}\label{eq448}
x_+=\Prox_x(\zeta)\Rightarrow \zeta^\transp (x_+-y ) \leq V_x(y)-V_{x_+}(y)-V_x(x_+)\,\,\forall y\in X.
\end{equation}
\end{enumerate}
\rule{\linewidth}{1pt}
\par {\textbf{Algorithm 2: Stochastic Mirror Descent.}}\\
\par {\textbf{Initialization.}} Take $x_1 = x_\omega$. Fix the number of iterations
$N-1$ and positive deterministic stepsizes $\gamma_1, \ldots, \gamma_N$.\\
\par {\textbf{Loop.}} For $t=1,\ldots,N-1$, compute
\begin{equation} \label{mdrecurrence}
\begin{array}{rcl}
x_{t+1}&=&\Prox_{x_t}(\gamma_t G(x_t,\xi_t)).
\end{array}
\end{equation}
\par {\textbf{Outputs:}} 
\begin{equation}\label{outputssmd}
\begin{array}{l}
x^N = \displaystyle \frac{1}{\Gamma_N} \sum_{\tau=1}^N \gamma_\tau x_\tau  \mbox{ and }g^N=\displaystyle \frac{1}{\Gamma_N}\left[ \sum_{\tau=1}^N  \gamma_\tau g(x_\tau,\xi_\tau)\right] \mbox{ with }\Gamma_N=\displaystyle \sum_{\tau=1}^N \gamma_\tau.
\end{array}
\end{equation}
\rule{\linewidth}{1pt}
The choice of $\omega$ depends on the feasibility set $X$. For the feasibility sets of problems
\eqref{definstance1} and \eqref{definstance2}, several distance-generating functions are of interest.
\begin{ex}[Distance-generating function for \eqref{definstance1} and \eqref{definstance2}]
For $\omega(x)=\omega_1(x)=\frac{1}{2}\|x\|_2^2$ and $\|\cdot\| = \|\cdot\|_2  = \|\cdot\|_{*}$, $\Prox_{x}(\zeta)=\Pi_{X}(x-\zeta)$ and the Stochastic Mirror Descent
algorithm is the RSA algorithm given by the recurrence \eqref{rsaiterations}. 
\end{ex}
\begin{ex}[Distance-generating function for problem \eqref{definstance1} with $a=1$ and $b=0$]
Let $\omega$
be the entropy function
\begin{equation} \label{entgenfunction}
\omega(x)=\omega_2(x)=\sum_{i=1}^n x(i) \ln(x(i))
\end{equation}
used in \cite{nemjudlannem09} with  $\|\cdot\|=\|\cdot\|_1$ and $\|\cdot\|_{*}=\|\cdot\|_{\infty}$. In this case, it is shown in \cite{nemjudlannem09}
that $x_{+}=\Prox_x(\zeta)$ is given by
$$
x_{+}(i)=\frac{x(i) e^{-\zeta(i)}}{\sum_{k=1}^n x(k) e^{-\zeta(k)}}, i=1,\ldots,n,
$$
and that we can take $D_{\omega_{2}, X}=\sqrt{2 \ln(n)}, \mu(\omega_{2})=1$, and $x_1=x_{\omega_{2}}=\frac{1}{n}(1,1,\ldots,1)^\transp$.
To avoid numerical instability in the computation of $x_{+}=\Prox_x(\zeta)$,
we compute instead $z_{+}=\ln(x_{+})$ from $z=\ln(x)$
using the alternative representation
$$
z_{+}=w-\ln\left(\sum_{i=1}^n e^{w(i)}\right) \mathbf{1} \mbox{ where }w=z-\zeta-\max_{i} [z(i)-\zeta(i)].
$$
\end{ex}
\begin{ex}[Distance-generating function for problem \eqref{definstance1} with $0<b<a/n$.] 
Let $\|\cdot\|=\|\cdot\|_1$, $\|\cdot\|_{*}=\|\cdot\|_{\infty}$, and as in \cite{ioudnemgui15}, \cite[Section 5.7]{MLOPT},
consider the distance-generating function  
\begin{equation}\label{omeganormp}
\omega(x)=\omega_3(x)={1\over p\gamma}\sum_{i=1}^n|x(i)|^p\;\mbox{ with }
p=1+1/\ln(n) \mbox{ and }\gamma= 
{1\over \exp(1) \ln(n)}.
\end{equation}
For every $x \in X$, since $p \rightarrow \|x\|_p$ is nonincreasing and $p>1$, we get
$\|x\|_p \leq \|x\|_1=a$ and $\max_{x \in X} \omega_3(x) \leq \frac{a^p}{p \gamma}$.
Next, using H\"older's inequality, for $x \in X$ we have $a=\sum_{i=1}^n x(i) \leq n^{1/q} \|x\|_p$
where $\frac{1}{p} + \frac{1}{q}=1$. We deduce that
$\min_{x \in X} \omega_3(x) \geq \frac{a^p}{p \gamma n^{1/\ln(n)}}$ and
that $D_{\omega_3,X} \leq \sqrt{\frac{2 a^p}{p \gamma}(1 - n^{-1/\ln(n)})}$.
We also observe that $D_X \leq \sqrt{2}(a- nb)$ and that $\mu(\omega_3)=\frac{\exp(1)}{n a^{2-p}}$: for $x, y \in X$ we have
$$
(\omega_3'(x) - \omega_3'(y))^\transp (x-y)  =  \frac{1}{\gamma} \sum_{i=1}^n (y(i) - x(i))(\varphi(y(i)) - \varphi(x(i) ))
 =  \frac{1}{\gamma} \sum_{i=1}^n  \varphi'( c_i ) (y(i) - x(i))^2 
$$
for some $0<c_i \leq a$
where $\varphi(x)=x^{p-1}$. Since $\varphi'(c_i) \geq \varphi'(a)=(p-1)a^{p-2}$, we obtain that
$(\omega_3'(x) - \omega_3'(y))^\transp (x-y) \geq \mu(\omega_3)\|y-x\|_1^2$ with $\mu(\omega_3)=\frac{\exp(1)}{n a^{2-p}}$.
In this context, each iteration of the SMD algorithm can be performed efficiently using Newton's method:
setting $x_{+}=\Prox_x(\zeta)$ and $z=\zeta - \omega_3'(x)$, $x_{+}$ is the solution of the optimization problem
$
\min_{y \in X}  \sum_{i=1}^n (1/ p \gamma) y(i)^p + z(i) y(i). 
$
Hence, there are Lagrange multiplers $\mu \geq 0$  and $\nu$ such that $\mu(i) (b - x_{+}(i)) = 0$,
$(1/\gamma)x_{+}(i)^{p-1} + z(i) - \nu - \mu(i)=0$ for $i=1,\ldots,n$, and $\sum_{i=1}^n x_{+}(i) =a$.
If $x_{+}(i)> b$ then $\mu(i)=0$ and $\nu -z(i)=(1/\gamma)x_{+}(i)^{p-1}>b^{p-1}/\gamma$, i.e.,
$
x_{+}(i) =\max( (\gamma(\nu - z(i)))^{\frac{1}{p-1}}, b ).
$
If $x_{+}(i)=b$ then $\mu(i) \geq 0$ can be written  $(1/\gamma)x_{+}(i)^{p-1}=\frac{1}{\gamma } b^{p-1} \geq \nu-z(i)$. 
It follows that in all cases 
$x_{+}(i) = \max( (\gamma(\nu - z(i)))^{\frac{1}{p-1}}, b )$. Plugging this relation into $\sum_{i=1}^n x_{+}(i) =a$, computing 
$x_{+}$ 
amounts to finding a root of the function
$f(\nu)=\sum_{i=1}^n \max( (\gamma(\nu - z(i)))^{\frac{1}{p-1}}, b ) -a$.
\end{ex}

In what follows, we provide confidence intervals for the optimal value of \eqref{defpbriskneutral} on the basis
of the points generated by the SMD algorithm, thus extending Proposition \ref{confidenceintervalRSA}.
We first need a technical lemma:
\begin{lemma}\label{lemmaMD} Let $e_1,...,e_N$ be a sequence of vectors from $E$, $\gamma_1,...,\gamma_N$ be nonnegative reals,
and let $u_1,...,u_N \in X$ be given by the recurrence
$$
\begin{array}{rcl}
u_1&=&x_\omega\\
u_{\tau+1}&=&\Prox_{u_\tau}(\gamma_\tau e_\tau),\,1\leq \tau\leq N-1.\\
\end{array}
$$
Then
\begin{equation}\label{lemmaMDineq}
\forall y\in X: \sum_{\tau=1}^N\gamma_\tau  e_\tau^\transp (u_\tau-y ) \leq {1\over 2}
D_{\omega, X}^2 + \frac{1}{2 \mu( \omega )}\sum_{\tau=1}^N \gamma_\tau^2\|e_\tau\|_*^2.
\end{equation}
\end{lemma}
\begin{proof} See the Appendix. \hfill
\end{proof}

Applying  Lemma \ref{lemmaMD} to $e_\tau=G(x_\tau,\xi_\tau)$ and in relation \eqref{lemmaMDineq}
specifying $y$ as  a minimizer $x_*$ of $f$ over $X$, we get:
$$
\sum_{\tau=1}^N\gamma_\tau  G(x_\tau,\xi_\tau)^\transp (x_\tau- x_*) \leq {1\over 2} D_{\omega, X}^2+\frac{1}{2 \mu( \omega )}\sum_{\tau=1}^N\gamma_\tau^2\|G(x_\tau,\xi_\tau)\|_*^2.
$$
Using notation \eqref{shortnotation} of the previous section, the above inequality can be rewritten
\begin{equation} \label{mdeq2}
\sum_{\tau=1}^N \; \gamma_{\tau}( x_{\tau}-x_{*} )^\transp  f'(x_{\tau})  \leq \frac{D_{\omega, X}^2}{2} + \frac{1}{2 \mu(\omega)} \sum_{\tau=1}^N \; \gamma_{\tau}^2 \| G_{\tau}\|_*^2  +\sum_{\tau=1}^N \; \gamma_{\tau} 
\Delta_{\tau}^\transp (x_{*}-x_{\tau}).
\end{equation}
We mentioned that when $\omega(x)=\frac{1}{2}\|x\|_2^2$, the SMD algorithm is the RSA algorithm of the
previous section. In that case, $\mu(\omega)=1$, $\|\cdot\|=\|\cdot\|_2$, $\|\cdot\|_{*}=\|\cdot\|_2$, and
\eqref{mdeq2} is obtained from inequality \eqref{rsaeq2} of the previous section for the RSA algorithm
substituting $D_X$ by $D_{\omega, X}$
(note that when choosing $x_1=x_{\omega}$ for the RSA algorithm, we have $D_X \leq D_{\omega, X}$
so for the RSA algorithm \eqref{rsaeq2} gives a tighter upper bound). We can now
extend the results of Lemma \ref{lemmersa1} and Proposition \ref{confidenceintervalRSA} to the SMD
algorithm:
\begin{lemma}\label{controleespmd} Let  Assumptions 1, 2, and 3 hold and assume that
the number of iterations $N-1$ of the SMD algorithm is fixed in advance with stepsizes given by
\begin{equation}\label{stepsizesmd}
\gamma_\tau=\gamma = \frac{D_{\omega, X} \sqrt{\mu(\omega)}}{\sqrt{2(M_2^2 + L^2)} \sqrt{N}},\;\tau=1,\ldots,N.
\end{equation}
Consider the approximation $g^N=\displaystyle \frac{1}{N}\displaystyle \sum_{\tau=1}^N  g(x_\tau,\xi_\tau)$
of $f(x_*)$. Then
\begin{equation} \label{upperboundmeanmd}
\mathbb{E}\left[ \Big| g^N - f( x_*) \Big|  \right] \leq \frac{ M_1+ \displaystyle \frac{D_{\omega, X}}{\sqrt{\mu( \omega )}} \sqrt{2(M_2^2 + L^2)} }{\sqrt{N}}.
\end{equation}
\end{lemma}
\begin{proof} It suffices to follow the proof of Lemma \ref{lemmersa1}, starting from
inequality \eqref{rsaeq2} which needs to be replaced by \eqref{mdeq2} for the Mirror Descent
algorithm.\hfill
\end{proof}
\begin{prop} \label{confidenceintervalmd}   Assume
that the number of iterations $N-1$ of the SMD algorithm is fixed in advance with stepsizes given by
\eqref{stepsizesmd}. Consider the approximation 
$g^N= \displaystyle \frac{1}{N}\displaystyle \sum_{\tau=1}^N  g(x_\tau,\xi_\tau)$
of $f(x_*)$. Then,
\begin{itemize}
\item[(i)] if Assumptions 1, 2, 3, and 4 hold, for any $\Theta>0$, we have
\begin{equation}\label{eqB3md}
\mathbb{P}\left(\Big|g^N-f(x_{*})\Big|>\frac{K_1(X) + \Theta K_2(X)}{\sqrt{N}} \right)\leq 4 \exp\{1\}\exp\{-\Theta\}
\end{equation}
where the constants $K_1(X)$ and $K_2(X)$ are given by
\begin{equation}\label{L1K2smd}
K_1(X)=\frac{D_{\omega, X} (M_2^2+2L^2)}{\sqrt{2 (M_2^2 + L^2 ) \mu( \omega )}} \mbox{ and }K_2(X)=\frac{D_{\omega, X} M_2^2}{\sqrt{2 (M_2^2 + L^2) \mu( \omega )}}+\frac{2 D_{\omega, X} M_2}{\sqrt{\mu(\omega)}} + M_1.
\end{equation}
\item[(ii)] If Assumptions 1, 2, 3, and 5 hold, then \eqref{eqB3md} holds with the right-hand side
replaced by $(3+\exp\{1\}) \exp\{-\frac{1}{4} \Theta^2 \}$.
\end{itemize}
\end{prop}
\begin{proof} It suffices to follow the proof of Proposition \ref{confidenceintervalRSA}, knowing that
inequality \eqref{rsaeq2} needs to be replaced by \eqref{mdeq2} for the Mirror Descent
algorithm. In particular, recalling that \eqref{eq447} holds, inequality
\eqref{eq110} becomes
$$
\begin{array}{rcl}
f^N-f(x_{*}) & \leq &
\displaystyle \frac{D_{\omega, X} (M_2^2+2L^2)}{\sqrt{2 (M_2^2 + L^2 ) \mu( \omega ) N}} + \frac{D_{\omega, X} M_2^2}{\sqrt{2 (M_2^2 + L^2) \mu( \omega )\,N}}\, \cA  + \frac{2 D_{\omega, X} M_2}{\sqrt{\mu( \omega )}N} \, \cB
\end{array}
$$
now with
$$
\cA = {1\over N M_2^2} \displaystyle \sum_{\tau=1}^N\|\Delta_\tau\|_*^2 \;\;\mbox{  and  }\;\;\cB =  {\sqrt{\mu(\omega)} \over 2 D_{\omega, X} M_2} \displaystyle  \sum_{\tau=1}^N
 \Delta_\tau^\transp (x_*-x_\tau ).
$$\hfill
\end{proof}
Similarly to Corollary \ref{corCIRSA}, we have the following corollary of Proposition \ref{confidenceintervalmd}:
\begin{cor}\label{corCIsmd} Let ${\tt{Up}}_1$ and ${\tt{Low}}_1$ be the upper and lower bounds given by respectively \eqref{upperbound} and
\eqref{lowerbound} now with 
$K_1(X)$ and $K_2(X)$ given by \eqref{L1K2smd} and $g^N$ given by \eqref{outputssmd}. Then if Assumptions 1, 2, 3, and 5 hold, for any $\Theta_1, \Theta_2, \Theta_3>0$, we have
\begin{equation}\label{cinfintsmdass5}
\mathbb{P}\left(f(x_* ) \in \Big[{\tt{Low}}_1(\Theta_2, \Theta_3, N) , {\tt{Up}}_1(\Theta_1, N)\Big]\right) \geq 1-e^{-\Theta_1^2/4}-e^{1-\Theta_2^2} - e^{-\Theta_2^2/4} - e^{-\Theta_3^2/4}   
\end{equation}
and parameters $\Theta_1, \Theta_2, \Theta_3$ can be chosen as in Remark \ref{choiceparametersCI} for
$\Big[{\tt{Low}}_1(\Theta_2, \Theta_3, N) , {\tt{Up}}_1(\Theta_1, N)\Big]$ to be a confidence interval with confidence level of at least
$1-\alpha$. If Assumptions 1, 2, 3, and 4 hold, then \eqref{cinfintsmdass5} holds with the term $e^{1-\Theta_2^2}$ replaced by $e^{1-\Theta_2}$.
\end{cor}

In the case when $f$ is uniformly convex
with convexity parameters $\rho$  and $\mu(f)$, \eqref{defpbriskneutral} has a unique optimal solution
$x_*$ and we can additionally bound from above $\mathbb{E}[\|x^N - x_*\|^\rho]$
by an $O(1/\sqrt{N})$ upper bound.
We recall that $f$ is uniformly convex on $X$ with convexity parameters $\rho \geq 2$  and $\mu(f)>0$ if
for all $t \in [0,1]$ and for all $x, y \in X$,
\begin{equation} \label{unifconvex}
f(t x + (1-t)y) \leq t f(x) + (1-t)f(y) -\frac{\mu( f )}{2}  t (1-t) (t^{\rho-1}+(1-t)^{\rho - 1}) \|x-y\|^\rho.
\end{equation}
A uniformly convex function with $\rho=2$ is called strongly convex. If a uniformly convex function $f$ is subdifferentiable at $x$, then
$$
\forall y \in  X, \;f(y) \geq f(x) + (y-x)^\transp f'(x) + \frac{\mu(f)}{2} \| y-x  \|^\rho
$$
and if $f$ is subdifferentiable at two points $x, y \in X$, then
$$
(y-x)^\transp (f'(y) - f'(x))  \geq \mu(f) \| y-x  \|^\rho.
$$
Note that if $g(\cdot, \xi)$ is uniformly convex for every $\xi$ then $f(x)=\mathbb{E}\Big[g(x, \xi)\Big]$
is uniformly convex with the same convexity parameters.

\begin{ex}
For problem \eqref{definstance1}, setting $V=\mathbb{E}[\xi \xi^T]$ and taking $\|\cdot\|=\|\cdot\|_1$, if $\lambda_0>0$, the objective function $f$ is uniformly convex with convexity parameters
$\rho=2$ and $\mu(f)=\frac{\alpha_1 (\lambda_{\min}(V) + \lambda_0 )}{n}$ where $\lambda_{\min}(V)$ is the smallest eigenvalue of $V$:
$$
\begin{array}{lll}
(f'(y)-f'(x))^\transp (y-x)& = &\alpha_1 (y-x)^\transp (V + \lambda_0 I)(y - x) \\
&\geq &\alpha_1 (\lambda_{\min}(V) + \lambda_0) \|y-x\|_2^2 \geq \frac{\alpha_1 (\lambda_{\min}(V) + \lambda_0 )}{n} \|y-x\|_1^2.
\end{array}
$$
\end{ex}
\begin{ex}
For problem \eqref{definstance2}, taking $\|\cdot\|=\|\cdot\|_2$, if $\lambda_0>0$ the objective function $f$ is uniformly convex with convexity parameters
$\rho=2$ and $\mu(f)= 2 \lambda_0$.
\end{ex}
\begin{ex}[Two-stage stochastic programs]
For the two-stage stochastic convex program defined in Section \ref{applitwostage}, 
if $f_1$ is uniformly convex on $X$ and if for every $\xi \in \Xi$ the function $f_2(\cdot, \cdot, \xi)$
is uniformly convex, then $f$ is uniformly convex on $X$.
For conditions ensuring strong convexity in some two-stage stochastic programs with complete recourse, we refer to \cite{romschultz93} and \cite{schultz94}.
\end{ex}

\begin{lemma}\label{controlqualunifconv} Let Assumptions 1, 2, and 3 hold
and assume that the number of iterations $N-1$ of the SMD algorithm is fixed in advance with stepsizes given by
\eqref{stepsizesmd}. Consider the approximation 
$g^N=\displaystyle \frac{1}{N}\sum_{\tau=1}^N  g(x_\tau,\xi_\tau)$
of $f(x_*)$ and assume that $f$ is uniformly convex. Then \eqref{upperboundmeanmd} holds
and
\begin{equation} \label{upperboundmeanmdsol}
\mathbb{E}\left[ \|x^N - x_*\|^\rho  \right] \leq \frac{D_{\omega, X} \sqrt{2(M_2^2 + L^2)}}{\mu( f )\sqrt{\mu( \omega )}\sqrt{N}}.
\end{equation}
\end{lemma}
\begin{proof} For every $\tau=1,\ldots,N$, since $x_{\tau} \in X$, the first order optimality conditions
give
$$
(x_{\tau} - x_{*})^\transp f'(x_* )   \geq 0.
$$
Using this inequality and the fact that $f$ is uniformly convex yields
\begin{equation} \label{unifconvexfirst}
\mu(f) \|x_{\tau} - x_*\|^{\rho}  \leq   (x_{\tau}-x_* )^\transp (f'( x_{\tau}) - f'(x_* )) \leq (x_{\tau} - x_{*})^\transp f'( x_{\tau} ).
\end{equation}
Next, note that since $\rho \geq 2$, the function
$\|x\|^\rho$ from $E$ to $\mathbb{R}_+$
is convex as a composition of the convex monotone function $x^\rho$
from $\mathbb{R}_+$ to $\mathbb{R}_+$ and of the convex function $\|x\|$ from $E$
to $\mathbb{R}_+$. It follows that
\begin{equation} \label{optsolution}
\begin{array}{lll}
\|x^N - x_*\|^\rho & = & \displaystyle  \left \lVert  \frac{1}{\Gamma_N} \sum_{\tau=1}^N \gamma_\tau (x_\tau - x_* ) \right \rVert^\rho \leq \displaystyle \sum_{\tau=1}^N \frac{\gamma_\tau}{\Gamma_N} \|x_\tau - x_* \|^\rho \\
& \leq & \displaystyle \frac{1}{\mu(f)} \sum_{\tau=1}^N \;\frac{\gamma_\tau}{\Gamma_N} (x_{\tau} - x_{*})^\transp f'( x_{\tau} )  \mbox{ using }\eqref{unifconvexfirst}.
\end{array}
\end{equation}
Finally, we prove \eqref{upperboundmeanmdsol} using the above inequality and \eqref{mdeq2}, and following the proof
of Lemma \ref{lemmersa1}. \hfill
\end{proof}

\section{Multistep Stochastic Mirror Descent} \label{mssmd}

The analysis of  the SMD algorithm of the previous section was done
taking $x_1= x_\omega$ as a starting point.
In the case when $f$ is uniformly convex, Algorithm 3 below is 
a multistep version of the Stochastic Mirror Descent algorithm
starting from an arbitrary point $y_1=x_1 \in X$.
A similar multistep algorithm was presented in \cite{nestioud2010}
for the {\em{method of dual averaging}}. The proofs of this section
are adaptations of the proofs of \cite{nestioud2010} to our setting.
However, in \cite{nestioud2010} the confidence intervals defined using the stochastic
method of dual averaging were not
computable whereas the confidence intervals to be given in this section
for the multistep SMD are computable.

We assume in this section that $f$ is uniformly convex, i.e., satisfies \eqref{unifconvex}.
For multistep Algorithm 3, at step $t$, Algorithm 2 is run for $N_t-1$ iterations starting from
$y_{t}$ instead of $x_{\omega}$ with steps that are constant along these iterations but that are decreasing with the algorithm step $t$.
The output $y_{t+1}$ of step $t$ is the initial
point for the next run of Algorithm 2, at step $t+1$.
To describe Algorithm 3, it is convenient to introduce
\begin{itemize}
\item[(1)] $x^{N}(x, \gamma)$: the approximate solution of \eqref{defpbriskneutral} computed as in
\eqref{outputssmd} where the points $x_1, \ldots, x_N$ are generated by Algorithm 2 run for $N-1$ iterations with constant step
$\gamma$ and using $x_1=x$ instead of $x_1=x_\omega$ as a starting point.;
\item[(2)] $g^{N}(x, \gamma)$: the approximation of the optimal value of \eqref{defpbriskneutral} computed as in
\eqref{outputssmd} where the points $x_1, \ldots, x_N$ are generated by Algorithm 2 run for $N-1$ iterations with constant step
$\gamma$ and using $x_1=x$ instead of $x_1=x_\omega$ as a starting point.
\end{itemize}
In Proposition \ref{propmdmultistep}, we provide an upper bound
for the mean error on the optimal value that is divided by two at each step.
We will assume that the prox-function is quadratically growing:\\
\par {\textbf{Assumption 6.}} There exists $0<M(\omega)<+\infty$ such that
\begin{equation}\label{proxquad}
V_{x}(y) \leq \frac{1}{2} M(\omega) \|x-y\|^2 \mbox{ for all }x,y \in X.
\end{equation}
Assumption 6 holds if $\omega$ is twice continuously differentiable on $X$ and in this case $M(\omega)$ can be related
to a uniform upper bound on the norm of the Hessian matrix of $\omega$. 

\begin{ex} When $\omega(x)=\omega_1(x)=\frac{1}{2}\|x\|_2^2$, we get $V_{x}(y)=\frac{1}{2}\|x-y\|^2$ and
Assumption 6 holds with $M(\omega)=1$ (this is the setting of RSA).
\end{ex}
Assumption 6 also holds for distance-generating functions $\omega_2$ and $\omega_3$ provided $X$ does not contain $0$:
\begin{ex}
For $X:=\{x \in \mathbb{R}^n : \sum_{i=1}^n x(i) = a, \;x(i) \geq b, i=1,\ldots,n\}$, with $0<b<a/n$, 
$\omega(x)=\omega_3(x)={1\over p\gamma}\sum_{i=1}^n|x(i)|^p\;\mbox{ with }
p=1+1/\ln(n) \mbox{ and }\gamma= 
{1\over \exp(1) \ln(n)}$, $\|\cdot\|=\|\cdot\|_1$ and $\|\cdot\|_*=\|\cdot\|_{\infty}$,
Assumption 6 is satisfied with
$M(\omega_3)=\frac{\exp (1)}{b^{1-1/\ln(n)} }$: indeed, since $\omega_3$ is twice continuously differentiable on $X$ with
$\omega_3''(x)=\frac{p-1}{\gamma}\mbox{diag}(x(1)^{p-2}, \ldots,x(n)^{p-2})$,
for every $x, y \in X$, there exists some $0<\tilde \theta<1$ such that
$$
\begin{array}{l}
V_x(y)  =  \omega_3(y) -\omega_3(x) -\omega_3'(x)^\transp (y-x) = \frac{1}{2} (y-x)^\transp \omega_3''(x + {\tilde \theta}(y-x)  )(y-x),
\end{array}
$$
which implies that
$$
\frac{\mu(\omega_3)}{2} \|y-x\|_1^2 =\frac{p-1}{2 \gamma a^{2-p}  n} \|y-x\|_1^2 \leq \frac{p-1}{2 \gamma a^{2-p}} \|y-x\|_2^2  \leq V_x( y ) 
 \leq \frac{p-1}{2 \gamma b^{2-p}} \|y-x\|_2^2 \leq \frac{p-1}{2 \gamma b^{2-p}} \|y-x\|_1^2, 
$$
where for the last inequality, we have used the fact that $\|y-x\|_2^2 \leq \|y-x\|_1^2$.
\end{ex}

\rule{\linewidth}{1pt}
\par {\textbf{Algorithm 3: multistep Stochastic Mirror Descent.}}\\
\par {\textbf{Initialization.}} Take $y_1=x_1 \in X$. Fix the number of steps $m$.\\
\par {\textbf{Loop.}} For $t=1,\ldots,m$,
\begin{enumerate}
\item[1)] Compute
\begin{equation}\label{defnt}
\displaystyle N_t=1+\left \lceil{\frac{2^{3+\frac{2(t-1)(\rho-1)}{\rho}} (L^2 + M_2^2) M(\omega) }{\mu^2(f) \mu(\omega) D_X^{2(\rho-1)}}} \right \rceil
\end{equation}
where $\left \lceil{x}\right \rceil$ is the smallest integer greater than or equal to $x$.
\item[2)] Compute $\gamma^t=\displaystyle \frac{D_X}{2^{\frac{t-1}{\rho}}\sqrt{N_t}} \sqrt{\frac{M(\omega) \mu( \omega )}{2(L^2 + M_2^2)}  }$.
\item[3)] Run Algorithm 2 (Stochastic Mirror Descent) for $N_t-1$ iterations, starting from $y_t$ instead of $x_\omega$, to compute
$y_{t+1}=x^{N_t}(y_t, \gamma^t)$ obtained using iterations \eqref{mdrecurrence}
with constant step $\gamma^t$ at each iteration.\\
\end{enumerate}
\par {\textbf{Outputs:}} $y_{m+1}=x^{N_m}(y_m, \gamma^m)$ and $g^{N_m}(y_m, \gamma^m)$.\\
\rule{\linewidth}{1pt}

If for Algorithm 2 (SMD algorithm), the initialization phase consists in
taking an arbitrary point $x_1$ in $X$ instead of $x_{\omega}$, analogues of
Lemmas \ref{controleespmd}, \ref{controlqualunifconv}, and of Proposition
\ref{confidenceintervalmd} can be obtained using Assumption 6 and replacing \eqref{mdeq2} by
the relation (see the proof of Lemma \ref{lemmaMD} for a justification):
\begin{equation} \label{mdeqbis}
\sum_{\tau=1}^N \; \gamma_{\tau} (x_{\tau}-x_{*})^\transp  f'(x_{\tau})  \leq
\frac{M(\omega)}{2} \|x_1-x_*\|^2 + \frac{1}{2 \mu(\omega)} \sum_{\tau=1}^N \; \gamma_{\tau}^2 \| G_{\tau}\|_*^2  +\sum_{\tau=1}^N \; \gamma_{\tau} (x_{*}-x_{\tau})^\transp \Delta_{\tau},
\end{equation}
which will be used in the sequel.
\begin{prop}\label{propmdmultistep} Let $y_{m+1}$ be the solution generated by Algorithm 3 after $m$ steps.
Assume that $f$ is uniformly convex and that
Assumptions 1, 2, 3, and 6 hold. Then
\begin{eqnarray}
&&\mathbb{E}\Big[ \|y_{m+1}-x_*\| \Big]  \leq  \displaystyle \frac{D_X}{2^{m/\rho}},\;\; \mathbb{E}\Big[ \Big|f(y_{m+1})-f(x_*) \Big| \Big]  \leq  \displaystyle \mu(f) \frac{D_X^\rho}{2^{m}}, \label{multistep1}\\
&&\mathbb{E}\Big[ \Big|g^{N_m}(y_m, \gamma^m)-f(x_*) \Big]  \leq    \mu(f) \frac{D_X^\rho}{2^{m}} + \frac{M_1}{\sqrt{N_m}}.    \label{assfivefirst}
\end{eqnarray}
\end{prop}
\begin{proof} We prove by induction
that $\mathbb{E}\Big[ \|y_{k}-x_*\| \Big]  \leq  D_k:=\frac{D_X}{2^{(k-1)/\rho}}$
and $\mathbb{E}\Big[ \|y_{k}-x_*\|^{\rho} \Big]  \leq  D_k^{\rho}$
for $k=1,\ldots,m+1$.  For $k=1$, the inequality holds. Assume that
it holds for some $k<m+1$. Using \eqref{mdeqbis} and following the proof of Lemmas  \ref{controleespmd}
and \ref{controlqualunifconv}, we obtain
\begin{eqnarray}
\mathbb{E}\Big[ \|x^{N_k}(y_k, \gamma^k)-x_*\|^{\rho} \Big]  & \leq & \frac{D_k}{\mu(f) \sqrt{N_k}} \sqrt{\frac{2(L^2+ M_2^2) M(\omega)}{\mu(\omega)}}, \label{mdmstep1}\\
\mathbb{E}\Big[ f(x^{N_k}(y_k, \gamma^k))-f(x_*) \Big]&=&\mathbb{E}\Big[ f(y_{k+1})-f(x_*) \Big] \nonumber \\
& \leq & \frac{D_k}{\sqrt{N_k}} \sqrt{\frac{2(L^2+ M_2^2) M(\omega)}{\mu(\omega)}}.\label{mdmstep2}
\end{eqnarray}
For \eqref{mdmstep1}, we have used the fact that
$$
\mathbb{E}\Big[ \|y_{k}-x_*\|^2 \Big] = \mathbb{E}\Big[ \left( \|y_{k}-x_*\|^{\rho} \right)^{2/\rho} \Big] \leq  
\left( \mathbb{E}\Big[ \|y_{k}-x_*\|^{\rho} \Big] \right)^{2/\rho} \leq  D_k^2,
$$
which holds using the induction hypothesis and Jensen inequality.
Plugging 
$$
N_k \geq 8\,\frac{2^{\frac{2(k-1)(\rho-1)}{\rho}} (L^2 + M_2^2) M(\omega) }{\mu^2(f) \mu(\omega) D_X^{2(\rho-1)}} = \frac{8 (L^2 + M_2^2) M(\omega) }{\mu^2(f) \mu(\omega) D_k^{2(\rho-1)}}
$$ into \eqref{mdmstep1} gives
$$
\mathbb{E}\Big[ \|y_{k+1}-x_*\|^{\rho} \Big] =
\mathbb{E}\Big[ \|x^{N_k}(y_k, \gamma^k)-x_*\|^{\rho} \Big]  \leq D_k \frac{D_k^{\rho-1}}{2} = D_{k+1}^{\rho}.
$$
Since for $\rho \geq 2$, the function $x^{1/\rho}$ is concave, using Jensen inequality we conclude
that $\mathbb{E}\Big[ \|y_{k+1}-x_*\| \Big] \leq D_{k+1}$ which achieves the induction.
Next, using \eqref{mdmstep2}, we obtain
$\mathbb{E}\Big[ \Big|f(y_{k+1})-f(x_*) \Big| \Big] \leq \mu(f) D_{k+1}^\rho$.
Finally, we prove \eqref{assfivefirst} using \eqref{multistep1} and following the end of the proof of Lemma \ref{lemmersa1}.
\hfill
\end{proof}
\begin{cor} Let $y_{m+1}$ be the solution generated by Algorithm 3 after $m$ steps.
Assume that $f$ is uniformly convex and that
Assumptions 1, 2, 3, and 6 hold. Then for any $\Theta>0$,
$
\mathbb{P}\Big(\|y_{m+1}-x_{*}\|^{\rho} > 2^{-\frac{m}{2}} \Theta \Big) \leq \frac{D_X^\rho}{\Theta} 2^{-\frac{m}{2}}.
$
\end{cor}

If at most $N$ calls to the oracle are allowed, Algorithm 3 becomes Algorithm 4.\\
\rule{\linewidth}{1pt}
\par {\textbf{Algorithm 4: multistep Stochastic Mirror Descent with no more than $N$ calls to the oracle.}}\\
\par {\textbf{Initialization.}} Take $y_1=x_1 \in X$, set $\tt{Steps}=1$, ${\tt{Nb}}_{\tt{Call}}=N_1 -1$
and fix the maximal number of calls $N$ to the oracle.\\
\par {\textbf{Loop.}} {\textbf{While}} ${\tt{Nb}_{\tt{Call}}} \leq N$,
\begin{enumerate}
\item[1)] Compute $\gamma^{\tt{Steps}}=\displaystyle \frac{D_X}{2^{\frac{\tt{Steps}-1}{\rho}} \sqrt{N_{\tt{Steps}}}} \sqrt{\frac{M(\omega) \mu( \omega )}{2(L^2 + M_2^2)}  }$
with $N_{\tt{Steps}}$ given by \eqref{defnt}.
\item[2)] Run Algorithm 2 (Stochastic Mirror Descent) for $N_{\tt{Steps}}-1$ iterations
with $N_{\tt{Steps}}$ given by \eqref{defnt}, starting from $y_{\tt{Steps}}$ instead of $x_\omega$, to compute
$y_{\tt{Steps}+1}=x^{N_{\tt{Steps}}}(y_{\tt{Steps}}, \gamma^{\tt{Steps}})$ obtained using iterations \eqref{mdrecurrence}
with constant step $\gamma^{\tt{Steps}}$ at each iteration.\\
\item[3)] $\tt{Steps} \leftarrow \tt{Steps} + 1$, $\tt{Nb}_{\tt{Call}} \leftarrow \tt{Nb}_{\tt{Call}} + N_{\tt{Steps}}-1$.
\end{enumerate}
\par \hspace*{0.9cm} {\textbf{ End while}}\\
\par $\tt{Steps} \leftarrow \tt{Steps} - 1$.\\
\par {\textbf{Outputs:}} $y_{\tt{Steps}+1}=x^{N_{{\tt{Steps}}}}(y_{{\tt{Steps}}}, \gamma^{{\tt{Steps}}})$ and $g^{N_{{\tt{Steps}}}}(y_{{\tt{Steps}}}, \gamma^{{\tt{Steps}}})$.\\
\rule{\linewidth}{1pt}
\begin{prop}\label{propmdmultistep2} Let $y_{\tt{Steps}+1}$ be the solution generated by Algorithm 4.
Assume that $f$ is uniformly convex and that $N$ is sufficiently large, namely that
\begin{equation} \label{nlarge}
N > 1+ \frac{2(2^{\beta} +1)}{\beta \ln 2}\ln \left( 1+ \frac{(2^{\beta}-1)}{A(f, \omega)} N \right),
\end{equation}
where $A(f, \omega)= \frac{8 (L^2 + M_2^2) M(\omega) }{\mu^2(f) \mu(\omega) D_X^{2(\rho-1)}}$
and where $1 \leq \beta=2\frac{\rho-1}{\rho}<2$.
If Assumptions 1, 2, 3, and 6 hold then
\begin{equation}\label{convmsmdesp} 
\begin{array}{lll}
\mathbb{E}\Big[ \|y_{\tt{Steps}+1}-x_*\|^\rho \Big] & \leq & \displaystyle D_X^\rho \left[ \frac{2^{\beta+1} A(f, \omega)}{(2^\beta - 1)(N-1) + 2A(f, \omega)}\right]^{1/\beta},\vspace*{0.1cm}\\
\mathbb{E}\Big[ \Big|f(y_{{\tt{Steps}}+1})-f(x_*) \Big| \Big] & \leq & \displaystyle \mu(f) D_X^\rho \left[ \frac{2^{\beta+1} A(f, \omega)}{(2^\beta - 1)(N-1) + 2A(f, \omega)}\right]^{1/\beta},
\end{array}
\end{equation}
and $\mathbb{E}\Big[ \Big|g^{N_{\tt{Steps}}}(y_{\tt{Steps}}, \gamma^{\tt{Steps}})-f(x_*) \Big| \Big]$ is bounded from above by
$$\displaystyle \mu(f) D_X^\rho \left[ \frac{2^{\beta+1} A(f, \omega)}{(2^\beta - 1)(N-1) + 2A(f, \omega)}\right]^{1/\beta} + \frac{M_1}{\sqrt{N_{\tt{Steps}}}}.$$
\end{prop}
\begin{proof} In the proof of Proposition \ref{propmdmultistep}, we have shown that
\begin{equation} \label{firstproofmdmstep2}
\mathbb{E}\Big[ \|y_{\tt{Steps}+1}-x_*\|^\rho \Big]  \leq  \displaystyle \frac{D_X^\rho}{2^{{\tt{Steps}}}}\;\mbox{ and }\;
\mathbb{E}\Big[ \Big|f(y_{{\tt{Steps}}+1})-f(x_*) \Big| \Big]  \leq  \displaystyle \mu(f) \frac{D_X^\rho}{2^{{\tt{Steps}}}}.
\end{equation}
Denoting for short $A(f, \omega)$ by $A$, we will show that
\begin{equation} \label{ineqtwomd}
\frac{1}{2^{{\tt{Steps}}}} \leq \left[ \frac{2^{\beta+1} A}{(2^\beta - 1)(N-1) + 2A}\right]^{1/\beta},
\end{equation}
which, plugged into \eqref{firstproofmdmstep2}, will prove the proposition. Let us check that
\eqref{ineqtwomd} indeed holds.
By definition of $N_t$ and of the number of steps of Algorithm 4, we have
$$
{\tt{Steps}}+1 + \frac{2^{\beta ({\tt{Steps}+1) }}-1}{2^{\beta}-1}A  =\sum_{t=1}^{{\tt{Steps}+1}} \,(1 + 2^{(t-1) \beta} A ) > \sum_{t=1}^{{\tt{Steps}+1}} \, (N_t -1) > N
$$
which can be written
\begin{equation} \label{ineqthreemd}
\frac{2^{\beta\,{\tt{Steps} }}}{2^{\beta}-1}A > \frac{1}{2^{\beta}} \left(N-{\tt{Steps}}-1+\frac{A}{2^{\beta}-1}\right),
\end{equation}
and
\begin{equation} \label{ineqmdbis}
N \geq \sum_{t=1}^{{\tt{Steps}}} \,(N_t - 1)  \geq  \sum_{t=1}^{{\tt{Steps}}} \,2^{(t-1) \beta} A  =  \frac{2^{\beta\,{\tt{Steps} }}-1}{2^{\beta}-1}A.
\end{equation}
From \eqref{ineqmdbis}, we obtain an upper bound on the number of steps:
\begin{equation} \label{ineqfivemd}
{\tt{Steps}} \leq \frac{\ln \left( 1+ \frac{(2^{\beta}-1)}{A} N \right)}{\beta \ln 2}.
\end{equation}
Combining \eqref{ineqthreemd}, \eqref{ineqmdbis}, and \eqref{ineqfivemd} gives
\begin{eqnarray}
&&\frac{-A}{2^{\beta}-1} + \frac{1}{2^{\beta}}\left(N-1+\frac{A}{2^{\beta}-1}\right)  \leq  \frac{{\tt{Steps}}}{2^{\beta}} + \sum_{t=1}^{{\tt{Steps}}} \, (N_t -1) \nonumber\\
&& \leq  \frac{{\tt{Steps}}}{2^{\beta}} + \sum_{t=1}^{{\tt{Steps}}} \, (1+2^{(t-1) \beta} A) \leq  {\tt{Steps}}(1+\frac{1}{2^{\beta}}) + \frac{2^{\beta\,{\tt{Steps} }}-1}{2^{\beta}-1}A   \nonumber\\
&& \leq  \frac{\ln \left( 1+ \frac{(2^{\beta}-1)}{A} N \right)}{\beta \ln 2}(1+\frac{1}{2^{\beta}}) + \frac{2^{\beta\,{\tt{Steps} }}-1}{2^{\beta}-1}A. \label{lastmdmstep4}
\end{eqnarray}
Plugging \eqref{nlarge} into \eqref{lastmdmstep4} and rearranging the terms gives \eqref{ineqtwomd}.\hfill
\end{proof}
Proposition \ref{propmdmultistep2} gives an $O(1/N^{\rho/2(\rho-1)})$ upper bound for
$\mathbb{E}\Big[ \Big|f(y_{{\tt{Steps}}+1})-f(x_*) \Big| \Big]$, which is tighter, since
$\frac{1}{\beta}=\rho/2(\rho-1)>\frac{1}{2}$, than the upper bounds obtained in the previous sections in the convex case.
When $\rho=2$, we obtain the rate $O(1/N)$ which is the best known convergence rate for
stochastic methods for minimizing strongly convex functions; see \cite{ghadimilan}, \cite{nediclee}, \cite{rakhlinshamirsrid}.
Finally, we provide a confidence interval for the optimal value of 
\eqref{defpbriskneutral}, obtained using the following multistep modified version of Algorithm 3 (a confidence interval
can also be obtained for the optimal value of \eqref{defpbriskneutral} using a similar modified version of Algorithm 4):\\
\rule{\linewidth}{1pt}
\par {\textbf{Algorithm 3': variant of Algorithm 3.}}
\par Algorithm 3 with the following modification: for each step $t$, when Algorithm 2 is run for $N_{t}-1$ iterations, 
the proximal mapping used in \eqref{mdrecurrence} is now defined by replacing
in \eqref{defprox} the set $X$ by $X \cap B(y_t, \frac{D_X}{2^{(t-1)/\rho}})$.\\
\rule{\linewidth}{1pt}
\begin{prop}\label{propmdmultistep3} Let $y_{m+1}$ be the solution generated by Algorithm 3'.
 Assume that $f$ is uniformly convex,
fix $\Theta>0$, and assume that $N_k$ is sufficiently large for $k=1,\ldots,m$, namely that
\begin{equation} \label{nlarge2}
N_k \geq  2^{\Big[2k-2\frac{(k-1)}{\rho}\Big]} \Big(K_1(X) + \Theta K_2(X)\Big)^2 
\end{equation}
with
$$
\begin{array}{l}
K_1(X)= \sqrt{\frac{M(\omega)}{2 \mu(\omega)(L^2 + M_2^2)}}\left( \frac{2L^2 + M_2^2}{\mu(f) D_X^{\rho-1}}\right)  \mbox{ and }\\
K_2(X)=\left( M_2^2 \sqrt{\frac{M(\omega)}{2 \mu(\omega) (L^2 + M_2^2)}}+2M_2 \right)\frac{1}{\mu(f) D_X^{\rho-1}}.
\end{array}
$$
Then if Assumptions 1, 2, 3, 4,  and 6 hold, we have
$$
\begin{array}{l}
\mathbb{P}\Big( \Big| g^{N_m}(y_m, \gamma^{m})-f(x_*) \Big| > \frac{\mu(f) D_X^{\rho}}{2^{m}} + \Theta \frac{M_1}{\sqrt{N_m}} \Big) \leq 2 m \exp\{1-\Theta \} + 2 \exp\{-\frac{1}{4} \Theta^2 \}.
\end{array}
$$
\end{prop}
\begin{proof} Let us fix $\Theta>0$.
Denoting by $x_{\tau}, \tau=1,\ldots, N_m$, the last $N_m$ points generated by the algorithm
and setting $f^{N_m}=\frac{1}{N_m} \sum_{\tau=1}^{N_m} f(x_\tau)$, following the proof of Proposition 
\ref{confidenceintervalRSA}, we have $\mathbb{P}\Big(\Big|g^{N_m}(y_m, \gamma^{m})-f^{N_m}\Big| >\Theta \frac{M_1}{\sqrt{N_m}}\Big) \leq 2 \exp\{-\frac{1}{4} \Theta^2 \}$.
We now show that
\begin{equation} \label{algo3prime}
\begin{array}{l}
\mathbb{P}\Big( \Big| f^{N_m}-f(x_*) \Big| > \frac{\mu(f) D_X^{\rho}}{2^{m}}  \Big) \leq 2 m \exp\{1-\Theta \},
\end{array}
\end{equation}
which will achieve the proof of the proposition.
The proof is by induction on the number of steps of the algorithm.
The induction hypothesis is that for some step $k \in \{1,\ldots,m\}$ and  for all $\ell=1,\ldots,k$,
there is a set $S_\ell$ of probability 1 if $\ell=1$ and at least $1-2 \exp\{1-\Theta \}$
otherwise such that on $\cap_{\ell=1}^{k} S_\ell$, we have $\|y_k-x_*\| \leq D_k=\frac{D_X}{2^{(k-1)/\rho}}$.
For $k=1$, the result holds. Assume now the induction hypothesis for some $k \in \{1,\ldots,m\}$.
We intend to show that \eqref{algo3prime} holds with $m$ substituted by $k$ and that
there is a set $S_{k+1}$ of probability at least $1-2 \exp\{1-\Theta \}$ such that on
on $\cap_{\ell=1}^{k+1} S_\ell$, we have $\|y_{k+1}-x_*\| \leq D_{k+1}=\frac{D_X}{2^{k/\rho}}$.
Denoting now by $x_{\tau}, \tau=1,\ldots, N_k$, the points generated at the $k$-th step of the algorithm,
using \eqref{mdeqbis} and the fact that $\|G_\tau\|_*^2 \leq 2(L^2 + \|\Delta_\tau\|_*^2)$, we have
for $f^{N_k}-f(x_*)$ the upper bound
\begin{equation}\label{eqqualityfinal}
\begin{array}{l}
 \displaystyle \frac{1}{N_k \gamma^k}\left[\frac{M(\omega)}{2} \|y_k-x_*\|^2 + \frac{1}{\mu(\omega)} \sum_{\tau=1}^{N_k} \; (\gamma^k)^2 (L^2 + \|\Delta_\tau\|_*^2)  +\sum_{\tau=1}^{N_k} \; \gamma^k 
 \Delta_{\tau}^\transp (x_{*}-x_{\tau}) \right] \\
\displaystyle \leq U_k:=\frac{M(\omega) D_k^2}{2 N_k \gamma^k} +\frac{L^2 \gamma^k}{\mu(\omega)} + \frac{\gamma^k M_2^2}{\mu(\omega)} \mathcal{A}_k  +\frac{2 D_k M_2}{N_k} \mathcal{B}_k
\end{array}
\end{equation}
on $\displaystyle \cap_{\ell=1}^{k} S_\ell$ where 
$$
\mathcal{A}_k=\frac{1}{N_k M_2^2} \sum_{\tau=1}^{N_k} \|\Delta_\tau\|_*^{2} \mbox{ and } \mathcal{B}_k=\frac{1}{2 D_k M_2} \sum_{\tau=1}^{N_k} \Delta_\tau^\transp  (x_{*}-x_{\tau}).
$$
Observe that on $\cap_{\ell=1}^{k} S_\ell$, we have $\|x_* - y_k\| \leq D_k$ and by definition of $x_\tau$,
we have $\|y_k - x_{\tau}\| \leq D_k$ for $\tau=1,\ldots,N_k$. It follows that
we can follow the proof of Proposition \ref{confidenceintervalRSA} to show that for any $\Theta > 0$,
$$
\mathbb{P}\Big(\mathcal{A}_k>\Theta \Big) \leq \exp\{1-\Theta \} \mbox{ and }\mathbb{P}\Big(\mathcal{B}_k>\Theta \sqrt{N_k}\Big) \leq \exp\{-\frac{1}{4} \Theta^2 \}.
$$
Thus there is a set $S_{k+1}$ of probability at least $1-2\exp\{1-\Theta \}$ such that on $S_{k+1}$, we have
$\mathcal{A}_k \leq \Theta$ and $\mathcal{B}_k \leq \Theta \sqrt{N_k}$. 
Next, on $\cap_{\ell=1}^{k+1} S_\ell$, plugging into \eqref{eqqualityfinal}
the upper bounds $\Theta$ and $\Theta \sqrt{N_k}$ for respectively
$\mathcal{A}_k$ and $\mathcal{B}_k$, using the definition of $\gamma^k$, and the lower bound \eqref{nlarge2} on $N_k$, we obtain for
$f^{N_k}-f(x_*)$ the upper bound $\frac{\mu(f) D_X^{\rho}}{2^{k}}=\mu(f) D_{k+1}^\rho$.
Observing that
$\mathbb{P}(\cap_{\ell=1}^{k+1} S_\ell ) \geq 1-2k\exp\{1-\Theta \}$,
we have shown \eqref{algo3prime} with step $m$ substituted by step $k$.
Finally, using \eqref{optsolution}, we have on $\cap_{\ell=1}^{k+1} S_\ell$
for $\|y_{k+1}-x_{*}\|^\rho$ the upper bound $\frac{U_k}{\mu(f)}$
where $U_k$ is defined in \eqref{eqqualityfinal}. Since we have
just shown that on $\cap_{\ell=1}^{k+1} S_\ell$, $U_k$ is bounded from above by
$\mu(f) D_{k+1}^\rho$, this achieves the induction step.\hfill
\end{proof}
Similarly to Corollary \ref{corCIRSA}, we can combine the upper bound $g^{N_m}(y_m, \gamma^{m}) + \frac{\Theta_1 M_1}{\sqrt{N_m}}$ with the
lower bound from Proposition \ref{propmdmultistep3} to obtain a less conservative (smaller, for fixed confidence level) confidence interval
for $f( x_* )$:
\begin{cor}\label{lastcormstep} Let $y_{m+1}$ be the solution generated by Algorithm 3'.
 Assume that $f$ is uniformly convex,
fix $\Theta>0$, and assume that $N_k$ is sufficiently large for $k=1,\ldots,m$, namely that
\eqref{nlarge2} holds.
Then if Assumptions 1, 2, 3, 4,  and 6 hold, for any $\Theta_1, \Theta_2>0$ we have
$$
\begin{array}{l}
\mathbb{P}\left(f( x_* ) \in   \Big[ g^{N_m}(y_m, \gamma^{m})- \frac{\mu(f) D_X^{\rho}}{2^{m}} - \frac{\Theta_2 M_1}{\sqrt{N_m}} , g^{N_m}(y_m, \gamma^{m}) + \frac{\Theta_1 M_1}{\sqrt{N_m}}\Big] \right) \\
\geq 1 -2m \exp\{1-\Theta \} - \exp\{-\Theta_1^4 /4 \} - \exp\{-\Theta_2^4 /4 \}. 
\end{array}
$$
\end{cor}
\begin{proof} It suffices to use Proposition \ref{propmdmultistep3} and to follow the proof of Corollary \ref{corCIRSA}.\hfill
\end{proof}

\section{Numerical experiments} \label{numsim}

\subsection{Comparison of the confidence intervals from Section \ref{riskneutral} and from \cite{nemlansh09}}

We compare the coverage probabilities and the computational time of two confidence intervals with confidence level
at least $1-\alpha=0.9$ on the optimal value of \eqref{definstance1} and \eqref{definstance2}, built using a sample
$\xi^N=(\xi_1,\ldots,\xi_N)$ of size $N$ of $\xi$:
\begin{enumerate}
\item the (non-asymptotic) confidence interval ${\mathcal C}_{\tt{SMD\,1}}=\Big[{\tt{Low}}_1(\Theta_2, \Theta_3, N) , {\tt{Up}}_1(\Theta_1, N)\Big]$
proposed in Section \ref{smd} with $\Theta_1, \Theta_2, \Theta_3$ as in Corollary \ref{corCIsmd}. 
\item The (non-asymptotic) confidence interval ${\mathcal C}_{\tt{SMD\,2}}=\Big[{\tt{Low}}_2(\Theta_2, N) , {\tt{Up}}_1(\Theta_1, N)\Big]$
proposed in \cite{nemlansh09} where\footnote{Note that parameter $D_{\omega,X}$ in \cite{nemlansh09}
is parameter $D_{\omega, X}$ given by \eqref{defDomega} divided by $\sqrt{2}$.} 
\begin{equation}\label{cinemshlan}
{\tt{Low}}_2(\Theta_2, N)={\underbar f}^N - \frac{1}{\sqrt{N}}\left( 
\left(\frac{1}{2 \theta} + 2 \theta  \right)   \frac{D_{\omega, X} M_{*}}{\sqrt{\mu( \omega)}}  + \Theta_2\left[M_1 + \Big[ 8+ \frac{2 \theta}{\sqrt{N}} \Big] \frac{D_{\omega, X} M_{*} }{\sqrt{\mu(\omega)}}  \right]
\right)
\end{equation}
with
$\displaystyle {\underbar f}^N = \min_{x \in X} \frac{1}{N} \sum_{t=1}^N \Big[ g(x_t, \xi_t) + G(x_t, \xi_t)^\transp (x - x_t ) \Big]$, taking
for $x_1, \ldots, x_N$, the sequence of points generated by the SMD algorithm with constant step $\gamma= \frac{\theta \sqrt{\mu(\omega)} D_{\omega, X}}{M_*  \sqrt{N}}$. 
In this expression, $M_*$ satisfies $\mathbb{E}\Big[ \exp\{\|G(x, \xi)\|_*^2 / M_*^2 \}  \Big] \leq \exp\{1\}$ for all $x \in X$.
Using Theorem 1 of \cite{nemlansh09}, we have $\mathbb{P}(f(x^*) < {\tt{Low}}_2(\Theta_2, N)) \leq 6 \exp\{-\Theta_2^2/3 \}  + \exp\{-\Theta_2^2/12 \}  + \exp\{-0.75\Theta_2 \sqrt{N} \}$.
Recalling that $\mathbb{P}(f(x^*) > {\tt{Up}}_1(\Theta_1, N)) \leq \exp\{-\Theta_1^2/4 \}$, it follows that we can take $\Theta_1 = 2 \sqrt{\ln(2/\alpha)}$ and
$\Theta_2$ satisfying $6 \exp\{-\Theta_2^2/3 \}  + \exp\{-\Theta_2^2/12 \}  + \exp\{-0.75\Theta_2 \sqrt{N} \}=\alpha/2$.
\end{enumerate}
All simulations were implemented in Matlab using Mosek Optimization Toolbox \cite{mosek}.

\subsubsection{Comparison of the confidence intervals on a risk-neutral problem}

We consider problem \eqref{definstance1} with $\alpha_0=0.1, \alpha_1 = 0.9, \lambda_0 =b=0, a=1$,
$n \in \{40, 60, 80, 100\}$, and where $\xi$ is a random vector with i.i.d. Bernoulli entries: 
$\Prob(\xi_i=1)=\Psi_i, \;\Prob(\xi_i=-1)=1-\Psi_i$, with $\Psi_i$ randomly drawn over $[0,1]$.
It follows that $f(x)=\alpha_0 \mu^\transp x + \frac{\alpha_1}{2} x^\transp V x$ where 
$\mu_i=\mathbb{E}[\xi_i ]=2\Psi_i - 1$ and $V_{i, j}=\mathbb{E}[\xi_i ]\mathbb{E}[\xi_j ]=(2\Psi_i - 1)(2\Psi_j - 1)$ for $i \neq j$
while $V_{i,i}= \mathbb{E}[ \xi_i^2 ]=1$. 
For SMD, we take $\|\cdot\|=\|\cdot\|_1$ and  for the distance-generating function the entropy function $\omega(x)=\omega_2(x)=\sum_{i=1}^n x(i) \ln( x( i ) )$.
We (first) take $\theta=1$ in \eqref{cinemshlan}, meaning that ${\mathcal C}_{\tt{SMD\,2}}$
is obtained running SMD with constant step
$\gamma= \frac{\sqrt{\mu(\omega)} D_{\omega, X}}{M_*  \sqrt{N}}$ where $M_* = |\alpha_0| + \alpha_1$.
We simulate 500 instances of this problem and compute for each instance the confidence intervals ${\mathcal C}_{\tt{SMD\,1}}$ and 
${\mathcal C}_{\tt{SMD\,2}}$. 
The coverage probabilities of the two non-asymptotic confidence intervals are equal to one for all parameter combinations.

 We report in Table \ref{ratioex1table1} the mean ratio of the widths of the 
non-asymptotic confidence intervals. 
Interestingly, we observe that the confidence interval ${\mathcal C}_{\tt{SMD\,1}}$ we proposed in Section \ref{riskneutral} is less conservative than ${\mathcal C}_{\tt{SMD\,2}}$:
in these experiments, the mean length of the width of ${\mathcal C}_{\tt{SMD\,2}}$ divided by the width of ${\mathcal C}_{\tt{SMD\,1}}$ varies between 3.80 and 3.85, as can be seen in Table \ref{ratioex1table1}.

\begin{table}
\centering
\begin{tabular}{|c||c|c|c|c|}
\hline
Sample& 
\multicolumn{4}{c|}{${|{\mathcal C}_{{\tt{SMD\,2}}}|/ |{\mathcal C}_{{\tt{SMD\,1}}}|}$, problem size $n$}   \\
\cline{2-5}
size $N$ &40&60&80&100\\
\hline
1\,000 &  3.82 & 3.83 & 3.84 & 3.85  \\
\hline
5\,000&  3.81  & 3.82    &  3.83    & 3.85  \\
\hline
10\,000&  3.80 & 3.82  & 3.83 & 3.84  \\
\hline
\end{tabular}
\caption{Average ratio of the widths of the confidence intervals for problem \eqref{definstance1}.}
\label{ratioex1table1}
\end{table}

Another advantage of ${\mathcal C}_{\rm SMD\, 1}$ is that it tends to be computed more quickly (see Table \ref{comptimern1} for problem sizes $n=40$, $60$, $80$, and $100$), especially when the 
problem size $n$ increases (see Table \ref{comptimern2} for $n=1 000$, $2 000$, $5000$, and $10\,000$), due to the fact that ${\mathcal C}_{\rm SMD\, 1}$ is computed using an analytic formula
while solving an (additional) optimization problem of size $n$ is required to compute  ${\mathcal C}_{\rm SMD\, 2}$.

\begin{table}
\centering
\begin{tabular}{|c||c|c|c|c|}
\hline
Confidence& \multicolumn{4}{c|}{Problem size $n$}\\
\cline{2-5}
interval&  40&60&80&100\\
\hline
${\mathcal C}_{{\tt{SMD\,1}}}$, $N = 1\,000$ &0.075  &0.091  & 0.094  & 0.109     \\
\hline
${\mathcal C}_{{\tt{SMD\,2}}}$, $N= 1\,000$ & 0.080 & 0.100  &  0.104 & 0.118       \\
\hline
${\mathcal C}_{{\tt{SMD\,1}}}$, $N = 10\,000$ &0.61  &0.62  & 0.61  & 0.70     \\
\hline
${\mathcal C}_{{\tt{SMD\,2}}}$, $N= 10\,000$ & 0.59 & 0.61  &  0.64 & 0.73       \\
\hline
\end{tabular}
\caption{Average computational time (in seconds) of a confidence interval estimated computing 500 confidence intervals for problem \eqref{definstance1}.}
\label{comptimern1}
\end{table}

\begin{table}
\centering
\begin{tabular}{|c||c|c|c|c|}
\hline
Confidence& \multicolumn{4}{c|}{Problem size $n$}\\
\cline{2-5}
interval&  1000&2000&5000&10\,000\\
\hline
${\mathcal C}_{{\tt{SMD\,1}}}$, $N = 100$ & 0.426 & 1.353 & 6.099  &  23.902      \\
\hline
${\mathcal C}_{{\tt{SMD\,2}}}$, $N= 100$ & 0.435 & 1.378 & 6.153  &  24.171      \\
\hline
\end{tabular}
\caption{Average computational time (in seconds) of a confidence interval estimated computing 50 confidence intervals for problem \eqref{definstance1}.}
\label{comptimern2}
\end{table}

We now fix a problem size $n=100$ and compute $100$ realizations of the confidence intervals on the optimal value
of that problem. On the top left plot of Figure \ref{SMDpb1}, we report the optimal value as well as the approximate
optimal values $g^N$ using variants {\tt{SMD 1}} and {\tt{SMD 2}} of SMD for three sample sizes: $N=1000, 5000$, and $10\,000$.
On the remaining plots of this figure, the upper and lower bounds of confidence intervals $\mathcal{C}_{\tt{SMD\,1}}$ and $\mathcal{C}_{\tt{SMD\,2}}$ 
are reported for sample sizes $N=1000, 5000$, and $10\,000$. We observe that the upper limits of $\mathcal{C}_{\tt{SMD\,1}}$ and $\mathcal{C}_{\tt{SMD\,2}}$
are very close (though not identical since the SMD variants {\tt{SMD 1}} and {\tt{SMD 2}} use different steps).
When the sample size $N$ increases, $g^N$ gets closer to  the optimal value and  the upper (resp. lower) limits tend to decrease (resp. increase).
 In this figure, we also see that ${\mathcal C}_{\tt{SMD\,1}}$ lower limit is
much larger than ${\mathcal C}_{\tt{SMD\,2}}$ lower limit (in accordance with the results of Table \ref{ratioex1table1}).
We also note that {\tt{SMD 1}} and {\tt{SMD 2}} lower bounds appear to be almost straight lines for these simulations. This comes from the fact
that the random part $g^N$ in these bounds is quite small compared to the deterministic part (remaining terms).

\begin{figure}[H]
\centering
\begin{tabular}{ll}
 \includegraphics[scale=0.5]{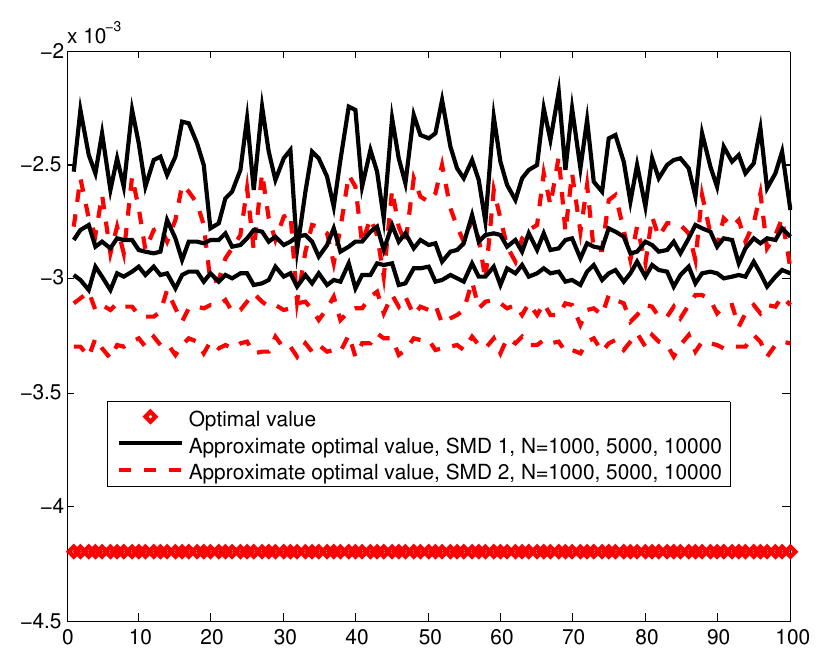}&
 \includegraphics[scale=0.5]{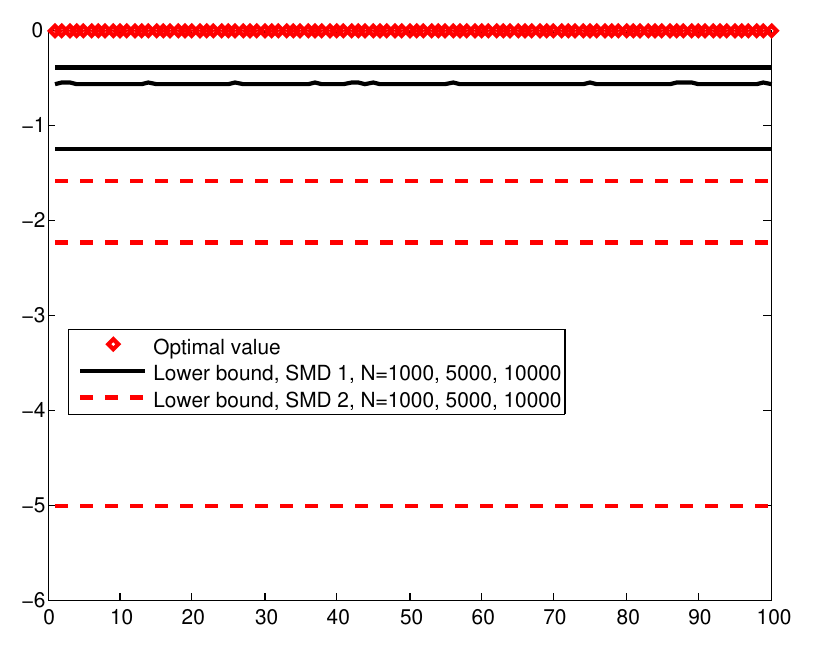}\\
 \includegraphics[scale=0.5]{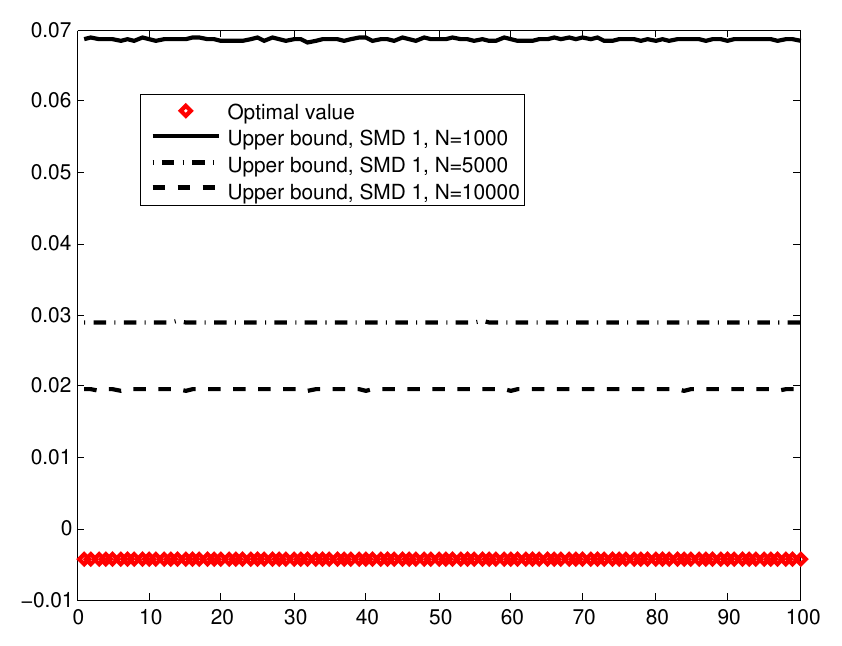}&
 \includegraphics[scale=0.5]{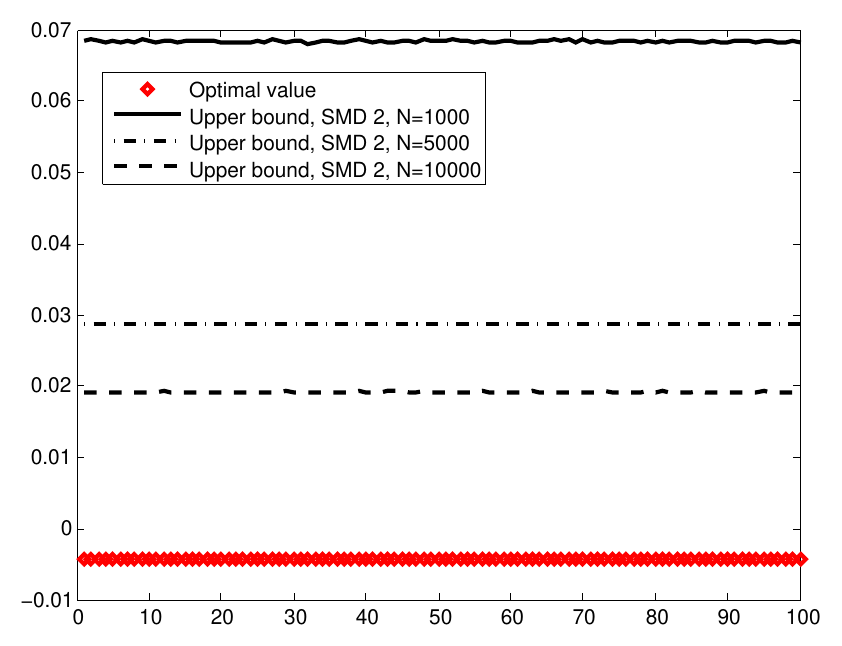}
\end{tabular} 
\caption{Approximate optimal value, upper and lower bounds for ${\mathcal C}_{\tt{SMD\,1}}$ and ${\mathcal C}_{\tt{SMD\,2}}$, on 100 instances of problem \eqref{definstance1} of size $n=100$.}
\label{SMDpb1}
\end{figure}

Finally, we consider for parameter $\theta$ involved in the computation of ${\mathcal C}_{\tt SMD\, 2}$ the range of values
$0.005$, $0.01$, $0.05$, $0.1$, $0.5$, $1$, $5$, $10$ considered in \cite{nemlansh09}. For these values of $\theta$, the average ratios of 
${\mathcal C}_{\tt SMD\, 2}$ and ${\mathcal C}_{\tt SMD\, 1}$ widths are given in Table \ref{ratioex1table2}.
These average ratios are all above $3.79$ and as high as $11.04$ for $(\theta, N, n) = (0.005, 1000, 100)$, which shows
again that ${\mathcal C}_{{\tt{SMD\,2}}}$ is much more conservative than the interval ${\mathcal C}_{{\tt{SMD\,1}}}$
proposed in Section \ref{smd} for this range of values of $\theta$.

\begin{table}
\centering
\begin{tabular}{|l||c|c|c|c|}
\hline
& \multicolumn{4}{c|}{Problem size $n$}\\
\cline{2-5}
\hspace*{2cm}(Ratio, $\theta$, $N$)&  40&60&80&100\\
\hline
${|{\mathcal C}_{{\tt{SMD\,2}}}|/ |{\mathcal C}_{{\tt{SMD\,1}}}|}$, $\theta=0.005$, $N= 1\,000$& 10.99 & 11.01 & 11.03  &11.04  \\
\hline
${|{\mathcal C}_{{\tt{SMD\,2}}}|/ |{\mathcal C}_{{\tt{SMD\,1}}}|}$, $\theta=0.01$, $N = 1\,000$  & 7.39 & 7.39 & 7.40  & 7.40    \\
\hline
${|{\mathcal C}_{{\tt{SMD\,2}}}|/ |{\mathcal C}_{{\tt{SMD\,1}}}|}$, $\theta=0.05$, $N= 1\,000$ & 4.45 & 4.45 & 4.45  & 4.46       \\
\hline
${|{\mathcal C}_{{\tt{SMD\,2}}}|/ |{\mathcal C}_{{\tt{SMD\,1}}}|}$, $\theta=0.1$, $N= 1\,000$& 4.06 & 4.07 & 4.07  &4.08  \\
\hline
${|{\mathcal C}_{{\tt{SMD\,2}}}|/ |{\mathcal C}_{{\tt{SMD\,1}}}|}$, $\theta=0.5$, $N = 1\,000$  & 3.79 & 3.81 &  3.81 & 3.82    \\
\hline
${|{\mathcal C}_{{\tt{SMD\,2}}}|/ |{\mathcal C}_{{\tt{SMD\,1}}}|}$, $\theta=1$, $N= 1\,000$ & 3.82 & 3.84 & 3.85  &  3.85      \\
\hline
${|{\mathcal C}_{{\tt{SMD\,2}}}|/ |{\mathcal C}_{{\tt{SMD\,1}}}|}$, $\theta=5$, $N= 1\,000$ & 4.36 & 4.38 & 4.39  & 4.40       \\
\hline
${|{\mathcal C}_{{\tt{SMD\,2}}}|/ |{\mathcal C}_{{\tt{SMD\,1}}}|}$, $\theta=10$, $N= 1\,000$ & 5.07 & 5.10 &  5.11 &  5.12      \\
\hline
\end{tabular}
\caption{Average ratio of the widths of confidence intervals ${\mathcal C}_{{\tt{SMD\,1}}}$ and ${\mathcal C}_{{\tt{SMD\,2}}}$, problem \eqref{definstance1}.}
\label{ratioex1table2}
\end{table}

\if{
We take $\|\cdot\|=\|\cdot\|_1$, $\|\cdot\|_{*}=\|\cdot\|_{\infty}$, and as in \cite{ioudnemgui15}, \cite[Section 5.7]{MLOPT},
the distance-generating function  
\begin{equation}\label{dgf.B_;bzc}
\omega(x)={1\over p\gamma}\sum_{i=1}^n|x(i)|^p\;\mbox{ with }
p=1+1/\ln(n) \mbox{ and }\gamma= 
{1\over \exp(1) \ln(n)}.
\end{equation}
For every $x \in X$, since $p \rightarrow \|x\|_p$ is nonincreasing and $p>1$, we get
$\|x\|_p \leq \|x\|_1=a$ and $\max_{x \in X} \omega(x) \leq \frac{a^p}{p \gamma}$.
Next, using H\"older's inequality, for $x \in X$ we have $a=\sum_{i=1}^n x_i \leq n^{1/q} \|x\|_p$
where $\frac{1}{p} + \frac{1}{q}=1$. We deduce that
$\min_{x \in X} \omega(x) \geq \frac{a^p}{p \gamma n^{1/\ln(n)}}$ and
that $D_{\omega,X} \leq \sqrt{\frac{2 a^p}{p \gamma}(1 - n^{-1/\ln(n)})}$.
We also observe that $D_X \leq \sqrt{2}(a- nb)$ and that $\mu(\omega)=\frac{\exp(1)}{n a^{2-p}}$:
$$
(\omega'(x) - \omega'(x))^\transp (x-y)  =  \frac{1}{\gamma} \sum_{i=1}^n (y_i - x_i)(\varphi(y_i) - \varphi(x_i ))
 =  \frac{1}{\gamma} \sum_{i=1}^n  \varphi'( c_i ) (y_i - x_i)^2 
$$
for some $0<c_i \leq a$
where $\varphi(x)=x^{p-1}$. Since $\varphi'(c_i) \geq \varphi'(a)=(p-1)a^{p-2}$, we obtain that
$(\omega'(x) - \omega'(x))^\transp (x-y) \geq \mu(\omega)\|y-x\|_1^2$ with $\mu(\omega)=\frac{\exp(1)}{n a^{2-p}}$.
Now with the notation of the previous sections, it is shown in \cite{ioudnemgui15} that for this objective function
$f$, Assumption 4 is satisfied with
$M_1=2|\alpha_0|  + \alpha_1$ and $M_2 = 2(|\alpha_0| + \alpha_1)$.
%see \cite{MLOPT},
%In what follows, we take $a=100$ and $b=10/n$.
 Next, observe that Assumption 6 is satisfied with
$M(\omega)=\frac{\exp (1)}{b^{1/\ln(n) -1} }$: indeed, since $\omega$ is twice continuously differentiable on $X$ with
$\omega''(x)=\frac{p-1}{\gamma}\mbox{diag}(x_1^{p-2}, \ldots,x_n^{p-2})$,
for every $x, y \in X$, there exists some $0<\tilde \theta<1$ such that
$$
\begin{array}{l}
V_x(y)  =  \omega(y) -\omega(x) -\omega'(x)^\transp (y-x) = \frac{1}{2} (y-x)^\transp \omega''(x + {\tilde \theta}(y-x)  )(y-x),
\end{array}
$$
which implies that
$$
\frac{\mu(\omega)}{2} \|y-x\|_1^2 =\frac{p-1}{2 \gamma a^{2-p}  n} \|y-x\|_1^2 \leq \frac{p-1}{2 \gamma a^{2-p}} \|y-x\|_2^2  \leq V_x( y ) 
 \leq \frac{p-1}{2 \gamma b^{2-p}} \|y-x\|_2^2 \leq \frac{p-1}{2 \gamma b^{2-p}} \|y-x\|_1^2, 
$$
where for the last inequality, we have used the fact that $\|y-x\|_2^2 \leq \|y-x\|_1^2$.
Note that for this instance the multistep SMD can be applied since the objective function $f$ is uniformly convex with convexity parameters
$\rho=2$ and $\mu(f)=\frac{\alpha_1 (\lambda_{\min}(V) + \lambda_0 )}{n}$:
$$
(f'(y)-f'(x))^\transp (y-x) = \alpha_1 (y-x)^\transp (V + \lambda_0 I)(y - x) \geq \alpha_1 (\lambda_{\min}(V) + \lambda_0) \|y-x\|_2^2 \geq \frac{\alpha_1 (\lambda_{\min}(V) + \lambda_0 )}{n} \|y-x\|_1^2.
$$
In this context, each iteration of the SMD algorithm can be performed efficiently using Newton's method:
setting $x_{+}=\Prox_x(\zeta)$ and $z=\zeta - \omega'(x)$, $x_{+}$ is the solution of the optimization problem
$
\min_{y \in X}  \sum_{i=1}^n (1/ p \gamma) y_i^p + z_i y_i. 
$
Hence, there are Lagrange multiplers $\mu \geq 0$  and $\nu$ such that $\mu_i (b - x_{+}(i)) = 0$,
$(1/\gamma)x_{+}(i)^{p-1} + z_i - \nu - \mu_i=0$ for $i=1,\ldots,n$, and $\sum_{i=1}^n x_{+}(i) =a$.
If $x_{+}(i)> b$ then $\mu_i=0$ and $\nu -z_i=(1/\gamma)x_{+}(i)^{p-1}>b^{p-1}/\gamma$, i.e.,
$
x_{+}(i) =\max( (\gamma(\nu - z_i))^{\frac{1}{p-1}}, b ).
$
If $x_{+}(i)=b$ then $\mu_i \geq 0$ can be written  $(1/\gamma)x_{+}(i)^{p-1}=\frac{1}{\gamma } b^{p-1} \geq \nu-z_i$. 
It follows that in all cases 
$x_{+}(i) = \max( (\gamma(\nu - z_i))^{\frac{1}{p-1}}, b )$. Plugging this relation into $\sum_{i=1}^n x_{+}(i) =a$, computing 
$x_{+}$ 
amounts to finding a root of the function
$f(\nu)=\sum_{i=1}^n \max( (\gamma(\nu - z_i))^{\frac{1}{p-1}}, b ) -a$.

We now consider problem \eqref{definstance1} with $a=1$ and $b=0$. We take
$\|\cdot \| = \| \cdot \|_2$, $\| \cdot  \|_* = \| \cdot \|_2$, and the distance generating function $\omega(x)=\frac{1}{2} \|x\|_2^2$.
In this situation, we compute $\mu(\omega)=M(\omega)=1, \mu(f)=\alpha_1 ( \lambda_{\min}(V) + \lambda_0 )$, $\rho=2$,  $x_{\omega}= 0$, $D_{\omega, X}=1$, and  $D_X \leq  1$.
Observing that we can take for this problem $G(x,\xi)=\alpha_0 \xi + \alpha_1 (\xi \xi^\transp + \lambda_0 I ) x$, we also have for $x \in X$ that
$$
\begin{array}{lll}
\|G(x,\xi) - \mathbb{E}[G(x,\xi)]\|_2   &\leq &  |\alpha_0| \|\xi - \mathbb{E}[\xi]   \|_2 + \alpha_1 \|(\xi \xi^\transp -V)x\|_2 \\
&\leq & 2 |\alpha_0| \sqrt{n} + \alpha_1 n \|\xi \xi^\transp -V\|_{\infty}^2 \leq  2 |\alpha_0| \sqrt{n} + 4 \alpha_1 n,
\end{array}
$$
$$
\begin{array}{lll}
\|\mathbb{E}[G(x,\xi)]\|_2   &\leq &  \alpha_0 \|\mathbb{E}[\xi]\|_2 + \alpha_1 n \|\mathbb{E}[\xi \xi^\transp ] + \lambda_0 I \|_{\infty}^2 \|x\|_2^2 \\
& \leq & \alpha_0 \sqrt{n} + \alpha_1 (1+ \lambda_0)^2 n
\end{array}
$$
and Assumptions 1 and 4 hold with $M_1=2|\alpha_0|+ \alpha_1$, $M_2=2 |\alpha_0| \sqrt{n} + 4 \alpha_1 n$, and $L=\alpha_0 \sqrt{n} + \alpha_1 (1+ \lambda_0)^2 n$.

}\fi

\subsubsection{Comparison of the confidence intervals on a risk-averse problem}

We reproduce the experiments of the previous section for problem \eqref{definstance2}
with $\|\cdot\| = \|\cdot\|_{*} = \|\cdot\|_2$ and the distance-generating function $\omega(x)=\omega_1(x) = \frac{1}{2}\|x\|_2^2$.
We take $M_*=\sqrt{\alpha_1^2 (1-\frac{1}{\varepsilon})^2   + n (\alpha_0 + \frac{\alpha_1}{\varepsilon})^2}$,
and two sets of values for $(\alpha_0, \alpha_1, \varepsilon)$:
$(\alpha_0, \alpha_1, \varepsilon)=(0.9,0.1,0.9)$ and the more risk-averse variant $(\alpha_0, \alpha_1, \varepsilon)=(0.1,0.9,0.1)$.

For these problems, we first discretize $\xi$, generating a sample of size $10^5$ which becomes the sample space.
We compute the optimal value of \eqref{definstance2} using this sample and sample from this set of scenarios to generate the problem
instances.

For different problem and sample sizes, we generate again 500 instances.
Coverage probabilities of the non-asymptotic confidence intervals are equal to one for all parameter combinations. 
The time required to compute these confidence intervals is given in Table \ref{ctimecvar1}
while the the average ratios of the widths of $\mathcal{C}_{\tt{SMD\,2}}$ and $\mathcal{C}_{\tt{SMD\,1}}$ are reported in Table 
\ref{ctimecvar2}. 

\begin{table}
\centering
\begin{tabular}{|c||c|c|c|c||c|c|c|c|}
\hline
Confidence interval and& \multicolumn{4}{c||}{$\varepsilon=0.1$, problem size }&\multicolumn{4}{c|}{$\varepsilon=0.9$, problem size}\\
\cline{2-5}\cline{6-9}
sample size $N$ &41&61&81&101&41&61&81&101\\
\hline
${\mathcal C}_{{\tt{SMD\,1}}}$, $N = 100$& 0.057& 0.069& 0.073& 0.094 & 0.058 & 0.065 &0.071& 0.082 \\
\hline
${\mathcal C}_{{\tt{SMD\,2}}}$, $N = 100$& 0.057& 0.064& 0.069& 0.094 & 0.055 & 0.062 &0.066& 0.074 \\
\hline
${\mathcal C}_{{\tt{SMD\,1}}}$, $N = 10\,000$& 5.74 & 6.13 &  6.92 & 7.47  & 5.79 & 6.59 &7.22& 7.97 \\
\hline
${\mathcal C}_{{\tt{SMD\,2}}}$, $N = 10\,000$& 5.97 & 6.41  & 7.22 & 7.81 & 5.85 & 6.63 &7.28& 8.00 \\
\hline
\end{tabular}
\caption{CVaR optimization (problem \eqref{definstance2}). Average computational time (in seconds) of a confidence interval estimated computing 500 confidence intervals.}
\label{ctimecvar1}
\end{table}

\begin{table}
\centering
\begin{tabular}{|c||c|c|c|c||c|c|c|c|}
\hline
Ratio and& \multicolumn{4}{c||}{$\varepsilon=0.1$, problem size }&\multicolumn{4}{c|}{$\varepsilon=0.9$, problem size}\\
\cline{2-5}\cline{6-9}
sample size $N$ &41&61&81&101&41&61&81&101\\
\hline
${|{\mathcal C}_{{\tt{SMD\,2}}}|/ |{\mathcal C}_{{\tt{SMD\,1}}}|}, N= 100$ & 2.29& 2.30& 2.31& 2.31 & 2.29 & 2.30 &2.31& 2.32 \\
\hline
${|{\mathcal C}_{{\tt{SMD\,2}}}|/ |{\mathcal C}_{{\tt{SMD\,1}}}|}, N= 10\,000$&2.30& 2.30& 2.30& 2.31 & 2.31 & 2.31 &2.32& 2.31 \\
\hline
\end{tabular}
\caption{CVaR optimization (problem \eqref{definstance2}). Average ratio of the widths of the confidence intervals ${\mathcal C}_{{\tt{SMD\,2}}}$ and ${\mathcal C}_{{\tt{SMD\,1}}}$.}
\label{ctimecvar2}
\end{table}

We observe again on this problem that  $\mathcal{C}_{\tt{SMD\,2}}$ is much more conservative than
$\mathcal{C}_{\tt{SMD\,1}}$ and for $N = 10\,000$ that $\mathcal{C}_{\tt{SMD\,1}}$  is computed quicker than 
$\mathcal{C}_{\tt{SMD\,1}}$ for all problem sizes. When  $\epsilon$ is small and more weight is given
to the CVaR, the optimization problem becomes more difficult, i.e., we need a large sample size to obtain a solution of good
quality. This can be seen in Figures \ref{SMDpb2} and \ref{SMDpb2bis}.

\begin{figure}
\centering
\begin{tabular}{ll}
\includegraphics[scale=0.5]{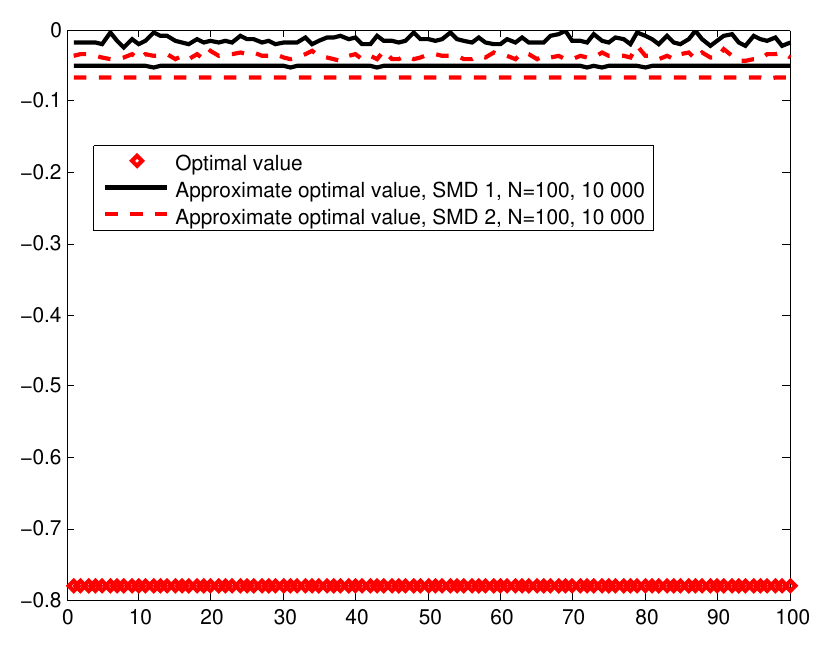}&  \includegraphics[scale=0.5]{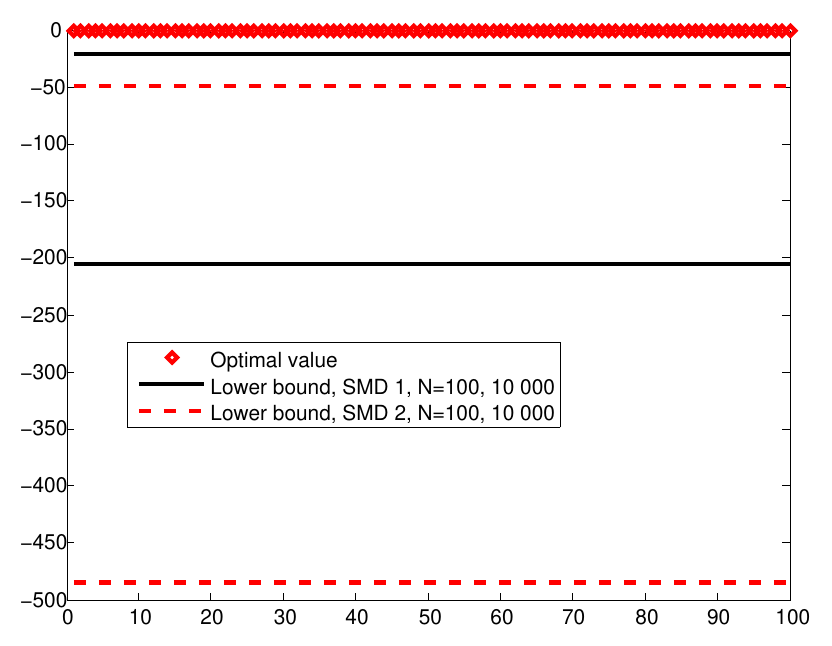}\\
\includegraphics[scale=0.5]{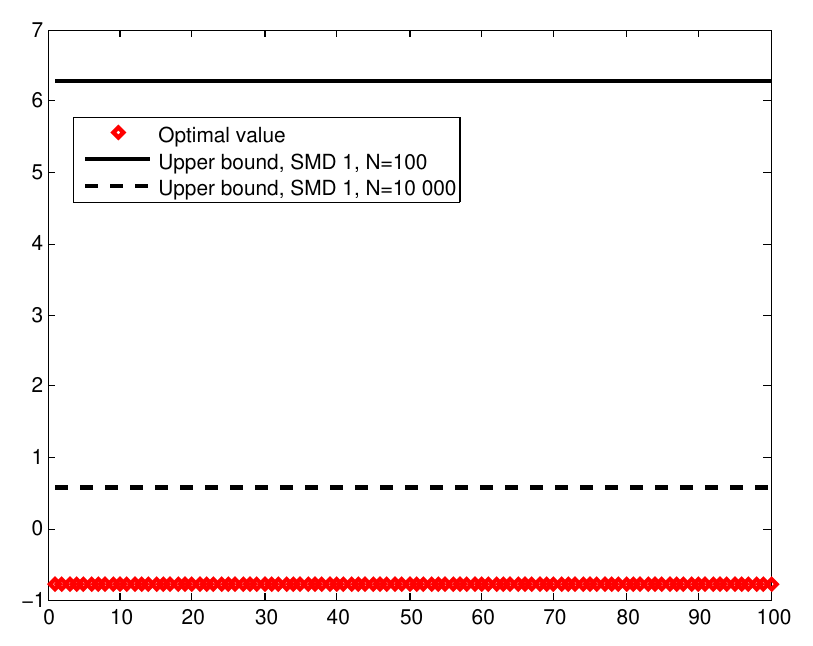} & \includegraphics[scale=0.5]{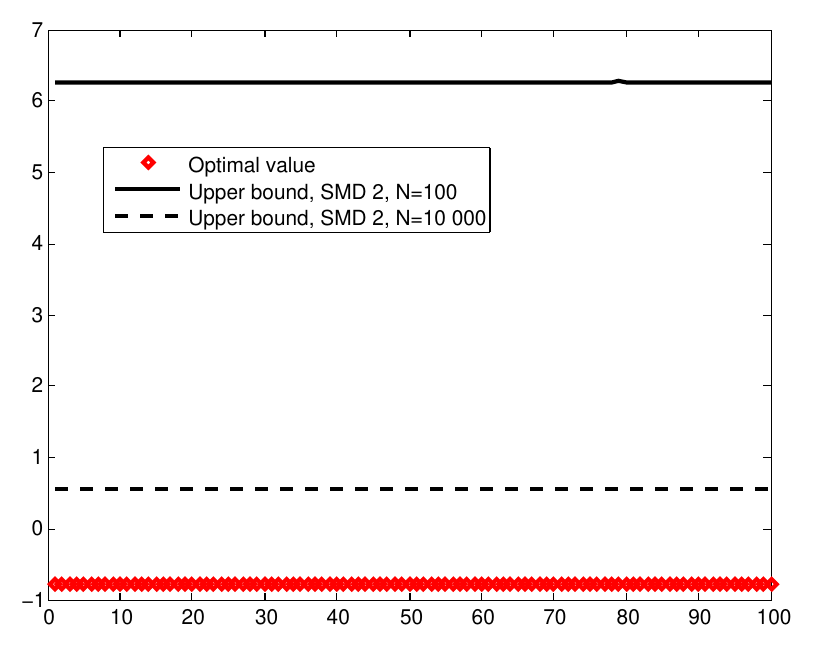}
\end{tabular}
\caption{CVaR optimization (problem \eqref{definstance2}). Approximate optimal value $g^N$, upper and lower bounds of $\mathcal{C}_{\tt{SMD\,2}}$ and $\mathcal{C}_{\tt{SMD\,1}}$ on 100 instances, problem size $n=100$
and  $\varepsilon=0.1$.}
\label{SMDpb2}
\end{figure}

On the top left plots of Figures \ref{SMDpb2} and \ref{SMDpb2bis}, for a problem of size $n=100$, we plot 100 realizations of 
the approximate
optimal values $g^N$ using variants {\tt{SMD 1}} and {\tt{SMD 2}} of SMD for two sample sizes: $N=100$
and $N=10\,000$ ($\varepsilon=0.1$ for Figure \ref{SMDpb2} and $\varepsilon=0.9$ for  Figure \ref{SMDpb2}).
For fixed sample size $N$, for $\varepsilon=0.9$ these realizations are much closer to the optimal value than
for $\varepsilon=0.1$.
On the remaining plots of Figure \ref{SMDpb2} and \ref{SMDpb2bis}, we report the upper and lower bounds of confidence intervals $\mathcal{C}_{\tt{SMD\,1}}$ 
and $\mathcal{C}_{\tt{SMD\,2}}$.
We observe again that (i) upper (resp. lower) bounds decrease (resp. increase) when the sample size increases,
(ii) $\mathcal{C}_{\tt{SMD\,2}}$ and $\mathcal{C}_{\tt{SMD\,1}}$ upper bounds are very close, and 
(iii) $\mathcal{C}_{\tt{SMD\,1}}$ lower bound is much larger than $\mathcal{C}_{\tt{SMD\,2}}$ lower bound (reflecting the fact
that $\mathcal{C}_{\tt{SMD\,2}}$ is much more conservative than $\mathcal{C}_{\tt{SMD\,1}}$).
Additionally, we observe that when $\epsilon$ is small ($\varepsilon=0.1$) and more weight is given to the CVaR ($\alpha_1=0.9$)
the upper and lower bounds become more distant to the optimal value, i.e., the width of the confidence intervals increases.

\begin{figure}
\centering
\begin{tabular}{ll}
\includegraphics[scale=0.5]{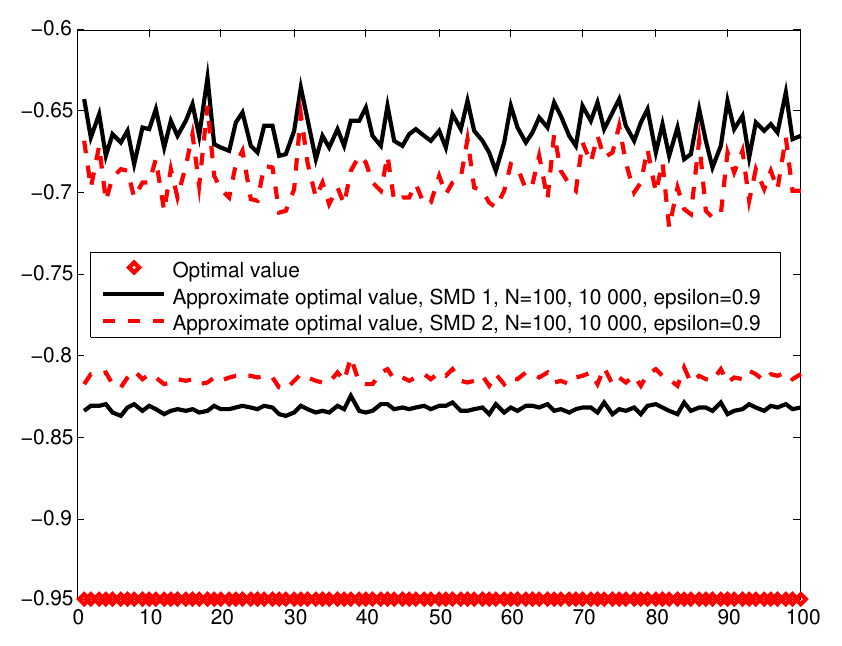} & \includegraphics[scale=0.5]{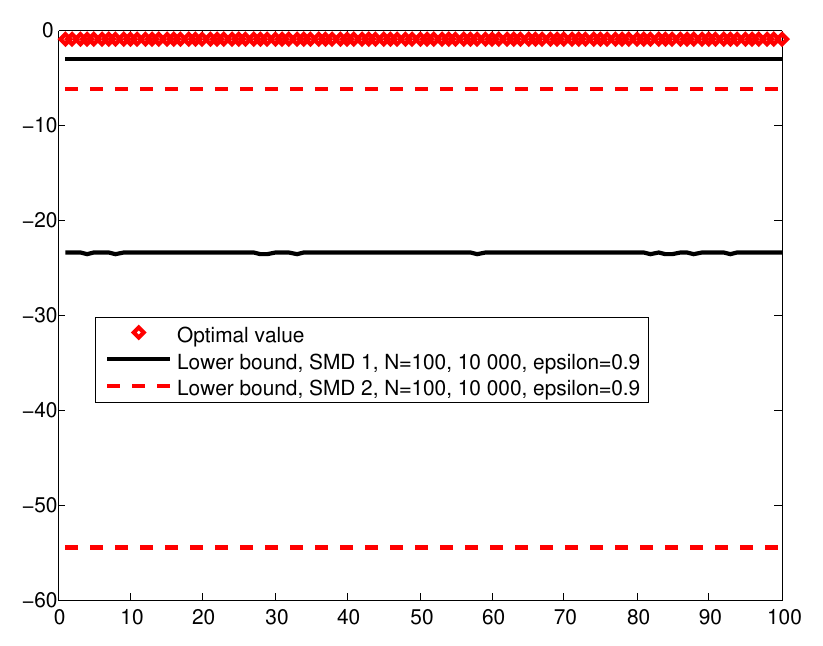}\\ 
\includegraphics[scale=0.5]{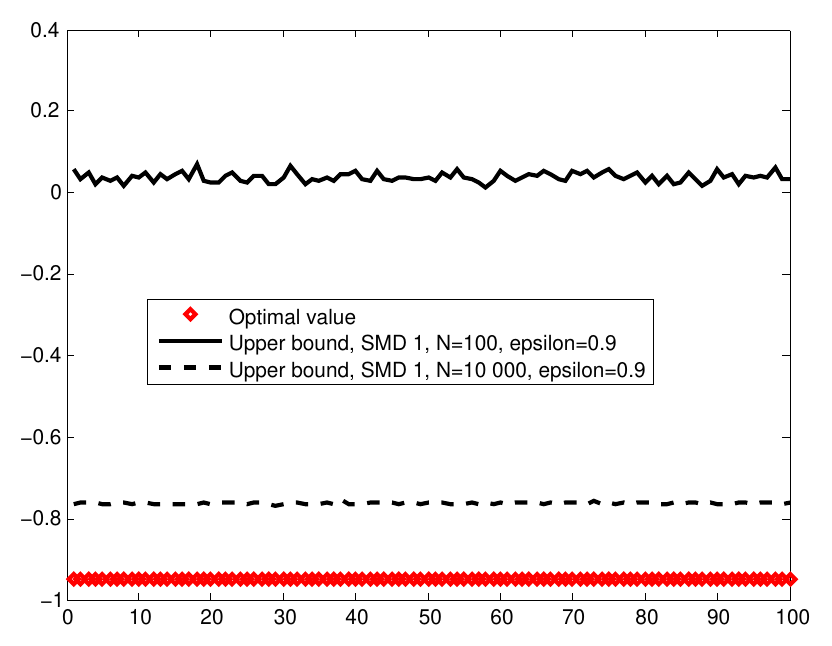} & \includegraphics[scale=0.5]{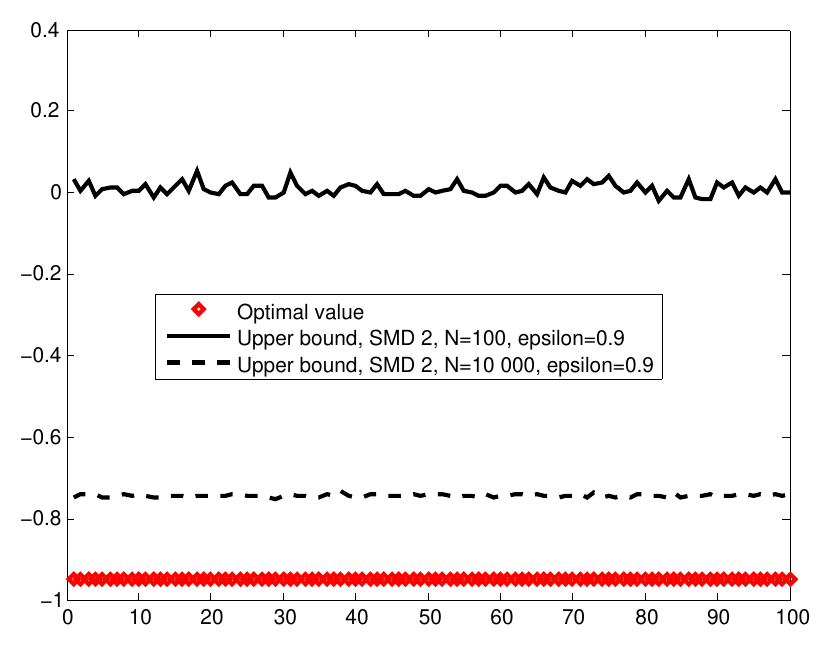}
\end{tabular}
\caption{CVaR optimization (problem \eqref{definstance2}). Approximate optimal value $g^N$, upper and lower bounds of $\mathcal{C}_{\tt{SMD\,2}}$ and $\mathcal{C}_{\tt{SMD\,1}}$ on 100 instances, problem size $n=100$
and $\varepsilon=0.9$.}
\label{SMDpb2bis}
\end{figure}

To conclude, confidence intervals ${\mathcal C}_{\tt SMD\, 1}$ and ${\mathcal C}_{\tt SMD\, 2}$
cannot be compared directly because both the constants involved and the steps used to generate the points
$x_1,\ldots,x_N$, are different. However, we hypothesize that the optimization in ${\tt SMD\, 2}$
results in both the conservativeness and the computation time difference.

\subsection{Comparing the multistep and nonmultistep variants of SMD to solve problem \eqref{definstance1}}

We solve various instances of problem \eqref{definstance1} (with $a=1, b=0$) using SMD and its multistep version defined in Section \ref{mssmd} taking $\omega(x)=\omega_2(x)=\frac{1}{2}\|x\|_2^2$.
These algorithms in this case are the RSA and multistep RSA. We fix the parameters $\alpha_1=0.9, \alpha_0=0.1, \lambda_0=4, x_1=[1;0;\ldots;0]], D_X=\sqrt{2}$, and recall
that $\mu(\omega)=\mu(\omega_2)=M(\omega_2)=\mu(f)=1, \rho=2$, $L=|\alpha_0 | \sqrt{n} + \alpha_1 (\sqrt{n} +  \lambda_0 )$, $M_1=2|\alpha_0|+0.5\alpha_1$, and $M_2=2 \sqrt{n} ( |\alpha_0| +  \alpha_1 )$.
In this and the next section, $\xi$ is again a random vector with i.i.d. Bernoulli entries: 
$\Prob(\xi_i=1)=\Psi_i, \;\Prob(\xi_i=-1)=1-\Psi_i$, with $\Psi_i$ randomly drawn over $[0,1]$.

We first take $n=100$ and choose the number of iterations using Proposition \ref{propmdmultistep2}, namely we take
$N=\lceil 1+78 A(f,\omega_2) \rceil=312\,248$  which ensures that for the MSSMD algorithm $\mathbb{E}\Big[ \Big|f(y_{{\tt{Steps}}+1})-f(x_*) \Big| \leq 0.1$.
(we also check that for this value of $N$, relation \eqref{nlarge} (an assumption of Proposition \ref{propmdmultistep2}) holds).
For this value of $N$, the values of $\gamma^t$ for each iteration of the MSSMD algorithm as well as the constant value of
$\gamma$ for the SMD algorithm are represented in the left plot of Figure \ref{mssmd1}.
We observe that the MSRSA algorithm starts with larger steps (when we are still far from the optimal solution) and ends with smaller steps 
(when we get closer to the optimal solution) than the RSA algorithm.
We run each algorithm 50 times and report in the middle plot of Figure \ref{mssmd1} the average (over the 50 runs) of the approximate optimal values computed along the
iterations with both algorithms. We also report in the right  plot of Figure \ref{mssmd1} the average (over these 50 runs) of the value of the objective function
at the SMD and MSSMD solutions.

More precisely, for each run of the SMD algorithm, for iteration $i$ the approximate optimal
value is $g^i=\frac{1}{i} \sum_{k=1}^i g(x_k, \xi_k)$ (defined in Algorithm 1)
while for iteration $j$ of the $i$-th step of the MSSMD algorithm,
the approximate optimal value is $g^{i, j}=\frac{1}{j} \sum_{k=1}^j g(x_{i, k}, \xi_{i, k})$
(defined in Algorithm 3)
where $\xi_{i, k}$ and $x_{i, k}$ are respectively the $k$-th realization of $\xi$ and
the $k$-th point generated for that step $i$ (of course, for a given run, the same samples are used for SMD and MSSMD).

We observe that we get better (lower) approximations of the optimal value using the MSRSA algorithm.
After a large number of iterations, the algorithms provide very close
approximations of the optimal value (themselves close to the optimal value of the problem), which is in agreement with the results of Sections \ref{riskneutral} 
and \ref{mssmd} which state that for both algorithms the approximate optimal values converge
in probability to the optimal value of the problem.
However, it is observed that the MSRSA algorithm provides an approximate solution of good quality much quicker
than the RSA algorithm.

\begin{figure}
\centering
\begin{tabular}{lll}
 \includegraphics[scale=0.4]{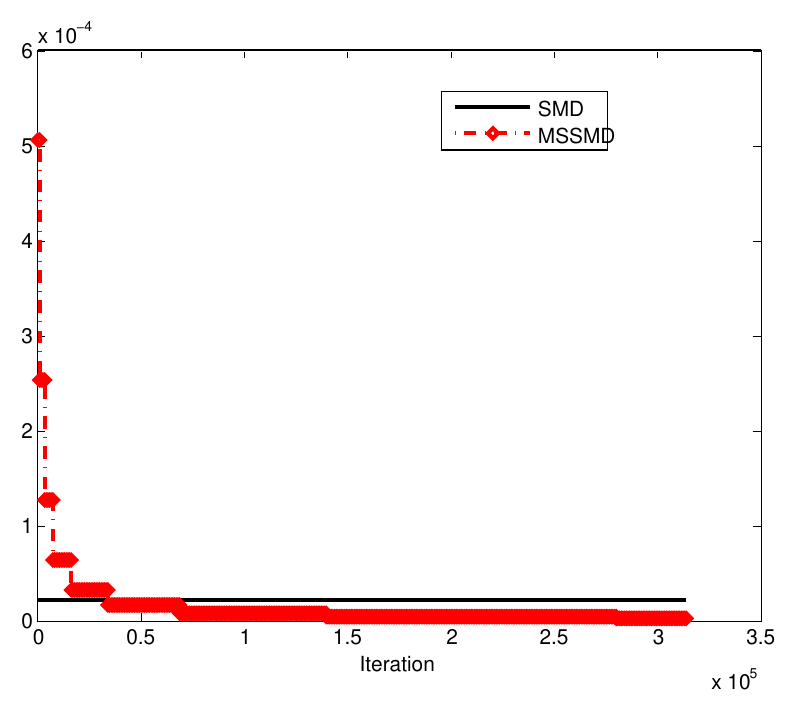}& \includegraphics[scale=0.4]{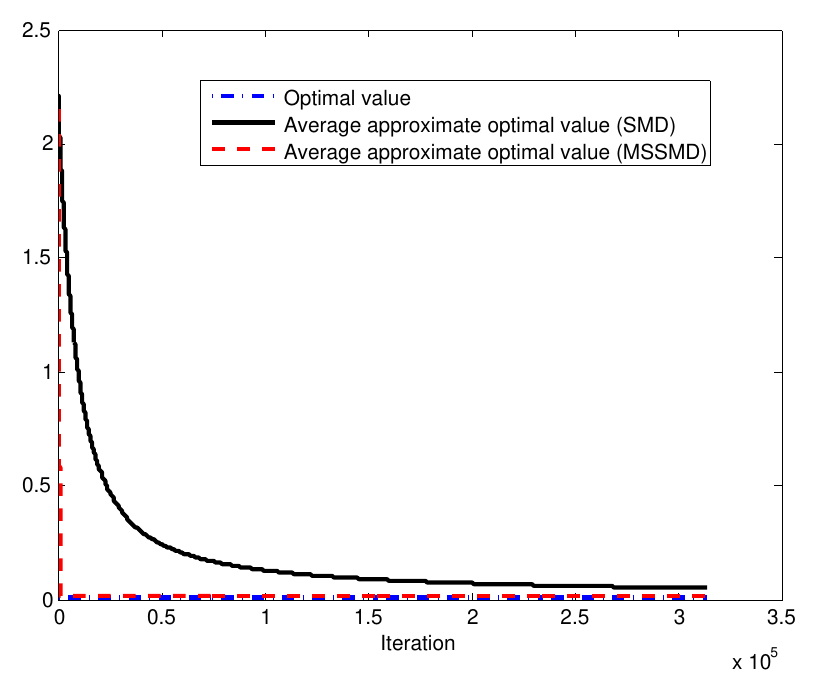}
\includegraphics[scale=0.4]{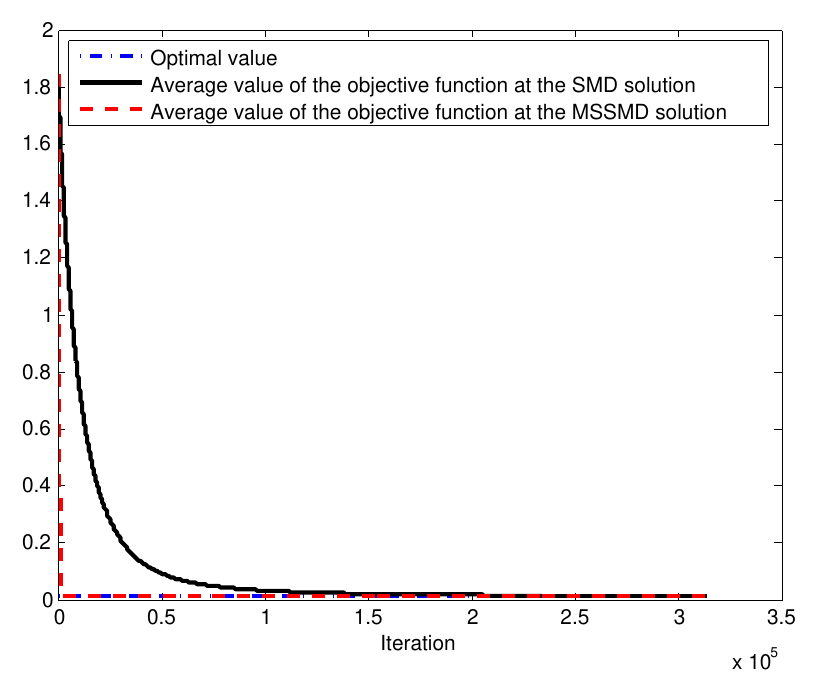}
 \end{tabular}
\caption{Steps (left plot), average (computed over 50 runs) approximate optimal values (middle plot), and average (computed over 50 runs) value of the objective function
at the solution (right plot) along the iterations of
the SMD and MSSMD algorithms run on problem \eqref{definstance1} with $n=100$, $N=312\,248$.}
\label{mssmd1}
\end{figure}

We also observe that if the value of the sample size $N=312\,248$ chosen based on Proposition \ref{propmdmultistep2} indeed allows us
to solve the problem with a good accuracy, it is very conservative. In a second series of experiments, we choose various problem
sizes $n$ and smaller sample sizes $N$, namely $(n,N)=(50, 1000), (n,N)=(100, 1000), (n,N)=(500, 10\,000)$, and $(n,N)=(1000, 10\,000)$,
still observing solutions of good quality.
For these values of the pair $(n,N)$, the values of the steps used for the SMD and MSSMD algorithms are reported in Figure \ref{stepsra1}.
Here again the MSRSA algorithm starts with larger steps and ends with smaller steps.

\begin{figure}[H]
\centering
\begin{tabular}{ll}
 \includegraphics[scale=0.6]{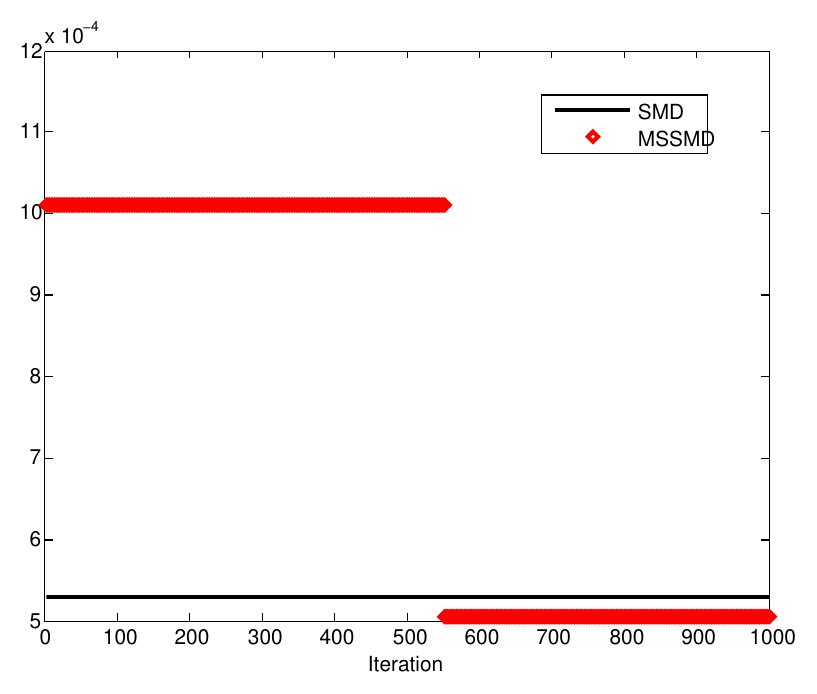}& \includegraphics[scale=0.6]{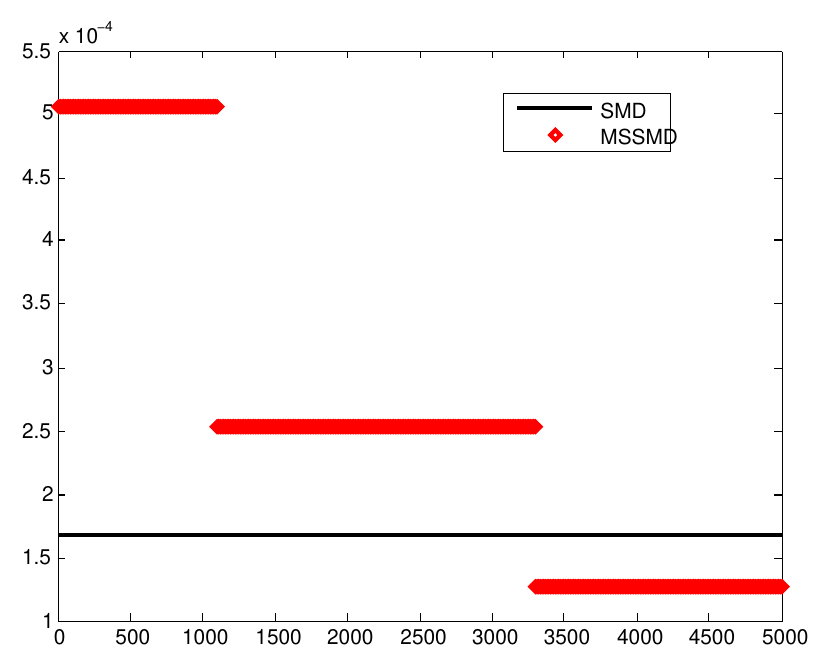}\\
\includegraphics[scale=0.6]{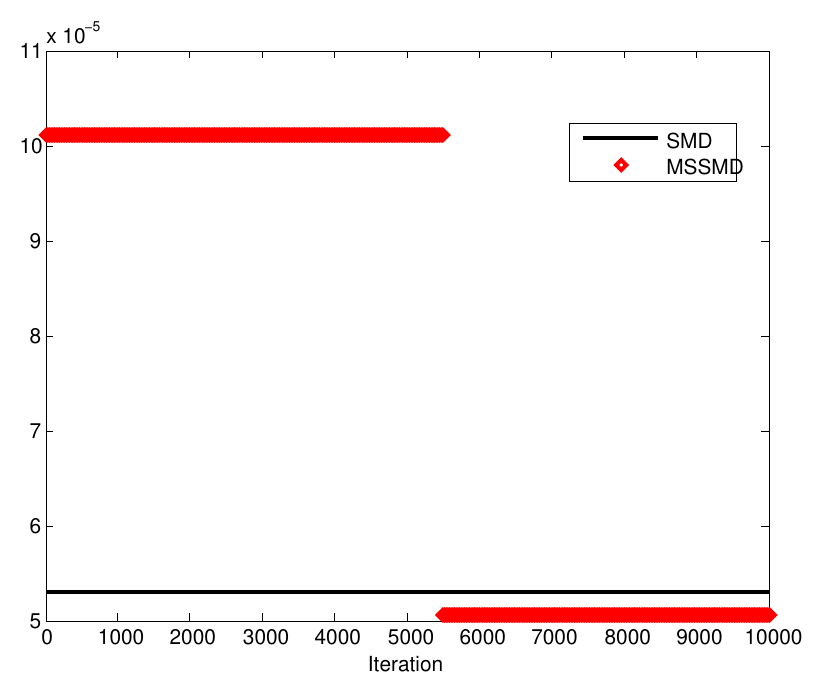}& \includegraphics[scale=0.6]{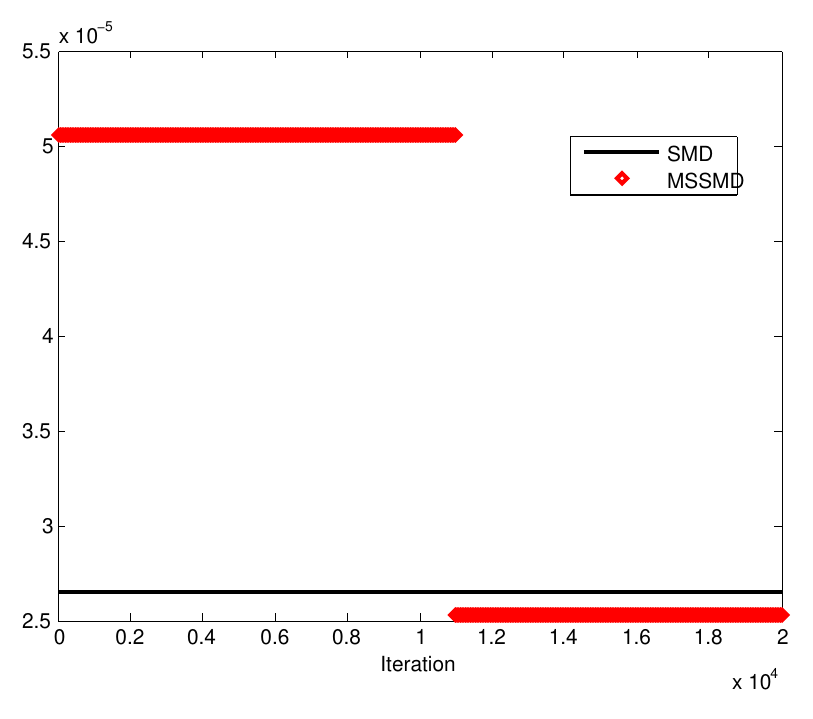}
 \end{tabular}
\caption{Steps used for the SMD and MSSMD algorithms to solve problem \eqref{definstance1} with 
$(n, N)=(50, 1000)$ (top left plot),
$(n, N)=(100, 5000)$ (top right plot), 
$(n, N)=(500, 10\,000)$ (bottom left), $(n, N)=(1000, 10\,000)$ (bottom right).}
\label{stepsra1}
\end{figure}

The average (over 50 runs) of the approximate optimal value and of the value of the objective function at the SMD and MSSMD solutions
are reported in Figures \ref{mssmdra1} and \ref{mssmdra1bis}. We still observe on these simulations that MSSMD allows us to obtain a solution of good quality much
quicker than SMD and ends up with a better solution, even when only two different step sizes are used for MSSMD.

\begin{figure}
\centering
\begin{tabular}{ll}
 \includegraphics[scale=0.5]{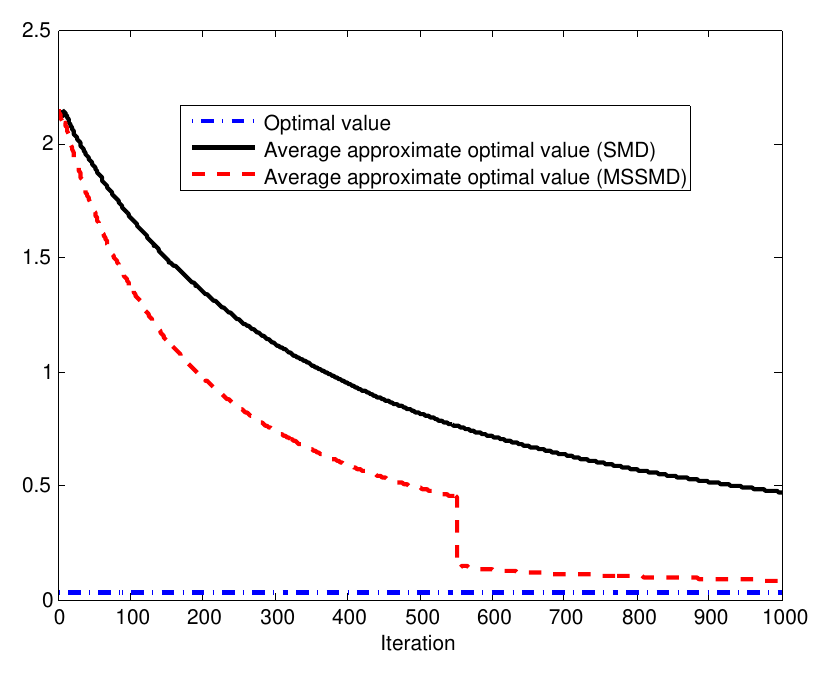}&\includegraphics[scale=0.5]{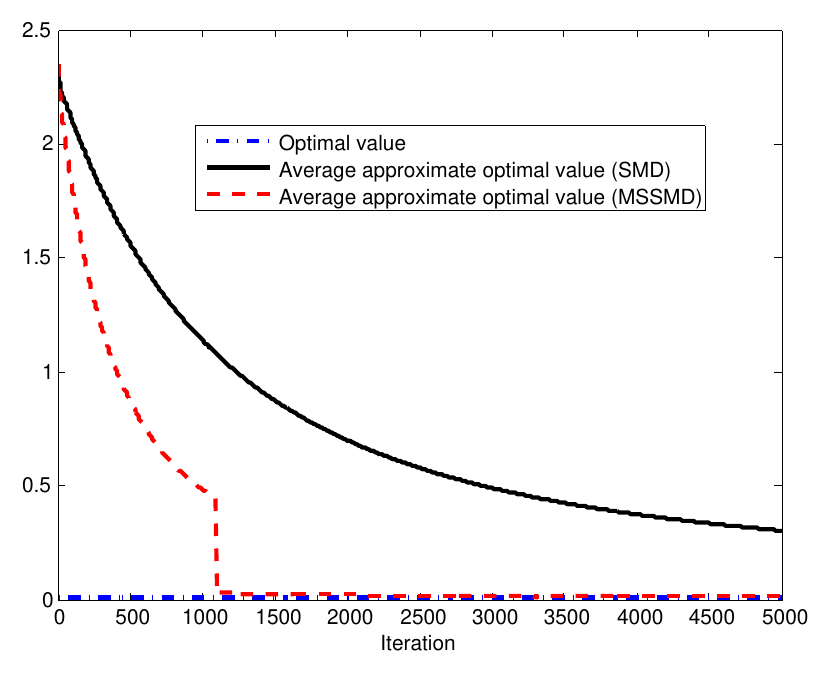}\\ 
 \includegraphics[scale=0.5]{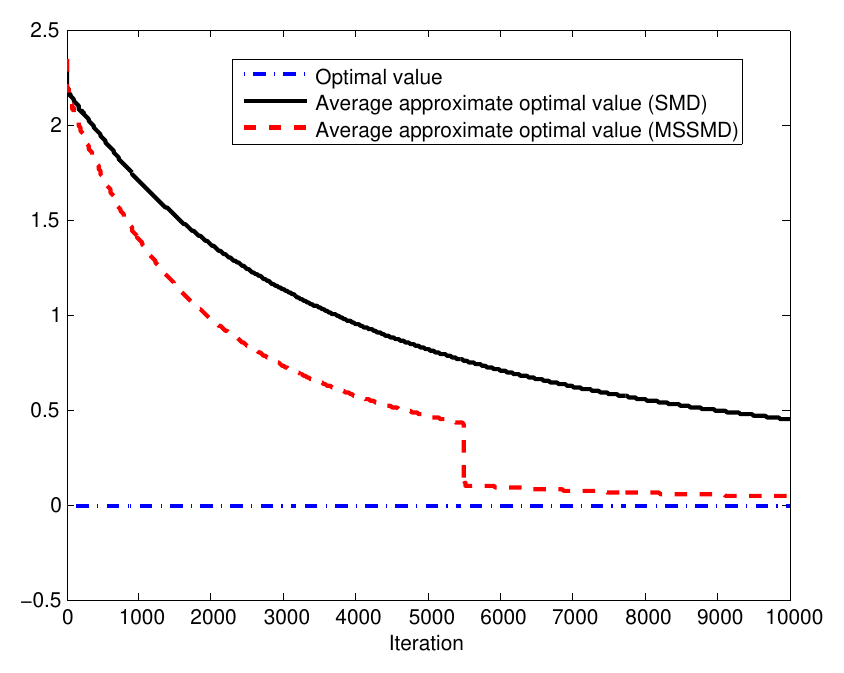}&\includegraphics[scale=0.5]{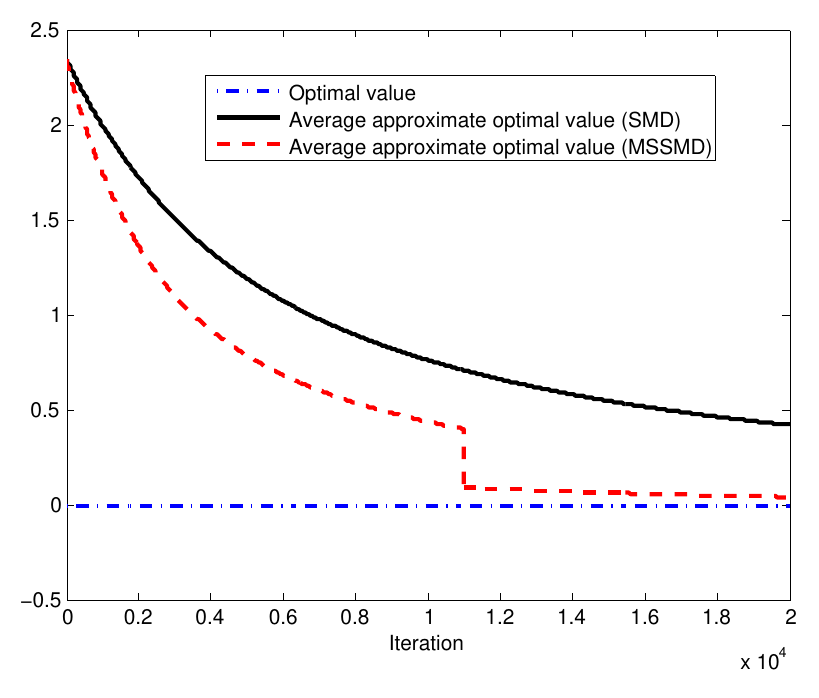}
 \end{tabular}
\caption{Average over 50 realizations of the
approximate optimal values computed by the SMD and MSSMD algorithms
to solve \eqref{definstance1}. Top left: $(n, N)=(50, 1000)$, top right: $(n, N)=(100, 5000)$, bottom left: $(n, N)=(500, 10\,000)$, bottom right: $(n,N)=(1000, 10\,000)$.}
\label{mssmdra1}
\end{figure}

\begin{figure}
\centering
\begin{tabular}{ll}
\includegraphics[scale=0.5]{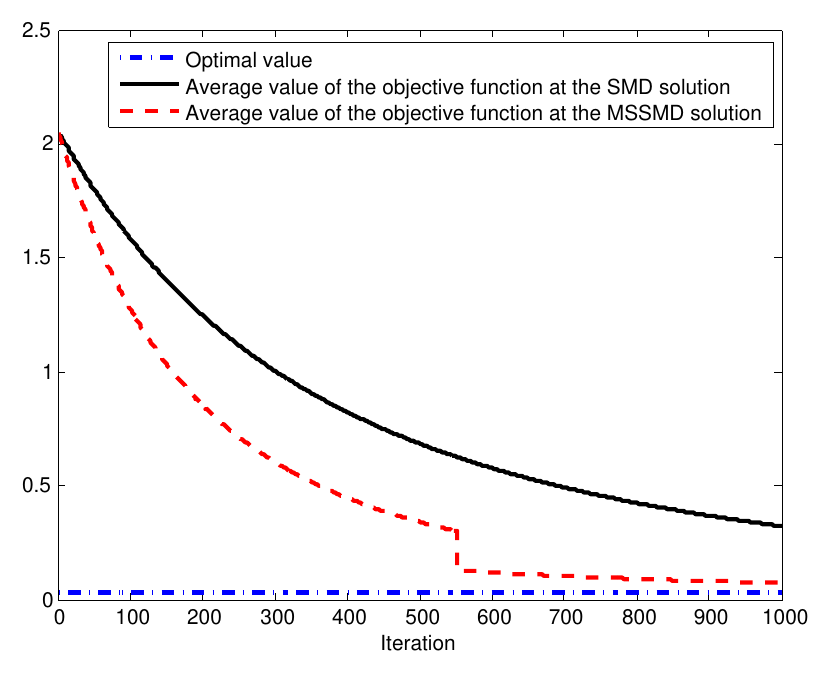} & \includegraphics[scale=0.5]{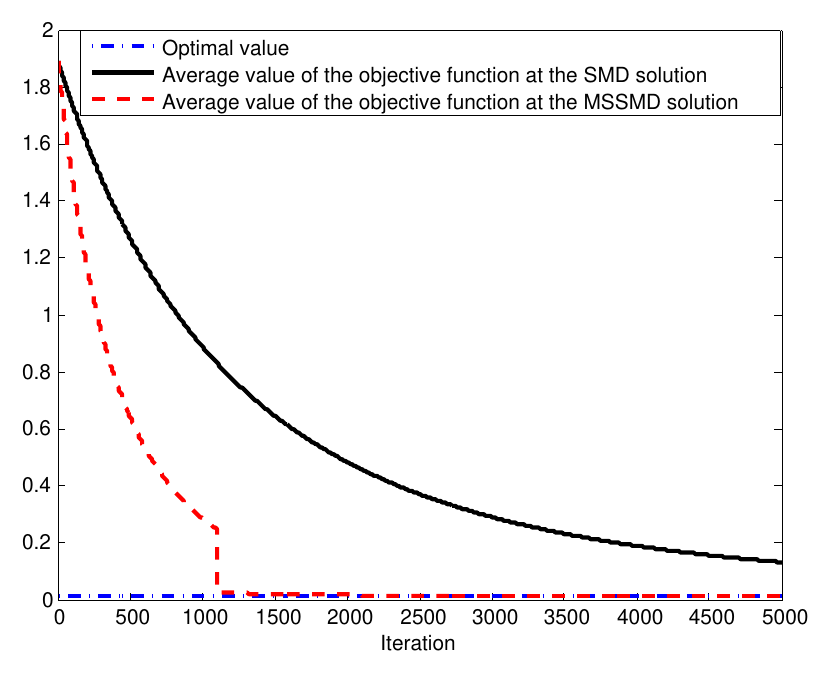}\\ 
\includegraphics[scale=0.5]{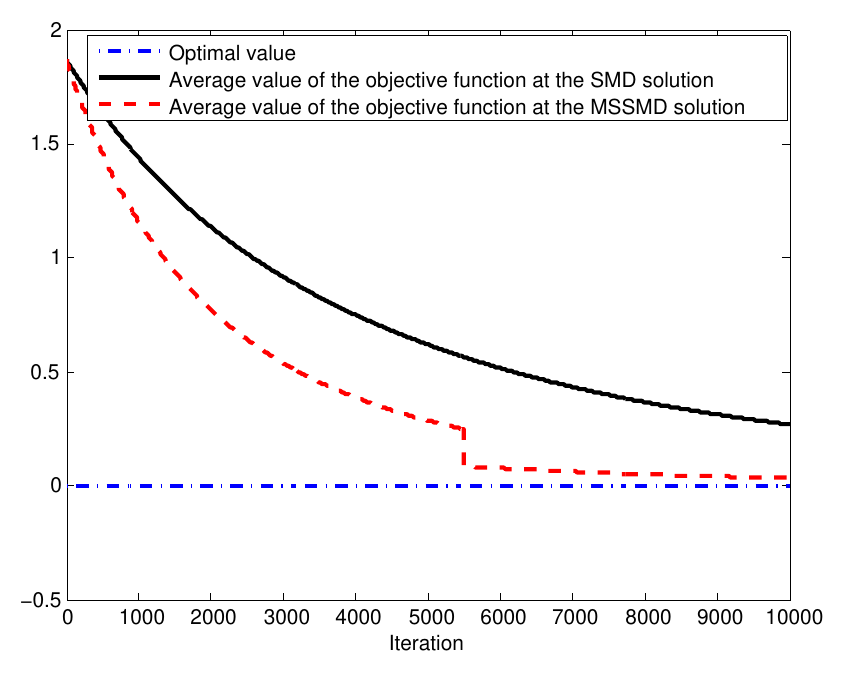}&\includegraphics[scale=0.5]{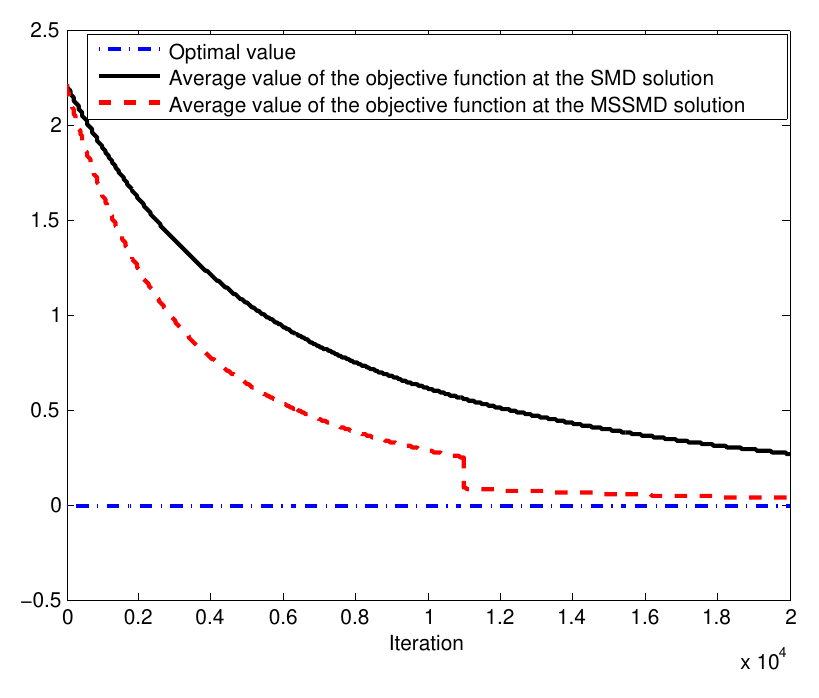}
 \end{tabular}
\caption{Average over 50 realizations of the values of the objective function at the approximate solutions computed by the SMD and MSSMD algorithms
to solve \eqref{definstance1}. Top left: $(n, N)=(50, 1000)$, top right: $(n, N)=(100, 5000)$, bottom left: $(n, N)=(500, 10\,000)$, bottom right: $(n, N)=(1000, 10\,000)$.}
\label{mssmdra1bis}
\end{figure}
\newpage

\subsection{Comparing the multistep and nonmultistep variants of SMD to solve problem \eqref{definstance3}}

We reproduce the experiment of the previous section running 50 times SMD and MSSMD on problem \eqref{definstance3}
taking $\omega(x)=\omega_2(x)=\frac{1}{2}\|x\|_2^2$, $\varepsilon=0.9$, $\alpha_1=0.1, \alpha_0=0.9, \lambda_0=1, x_1=[0;1;0;\ldots;0]], D_X=\sqrt{3}$, and recall
that $\mu(\omega)=\mu(\omega_2)=M(\omega_2)=\mu(f)=1, \rho=2$, 
$L=\sqrt{\alpha_1^2 (1-\frac{1}{\varepsilon})^2   + n (\alpha_0 + \frac{\alpha_1}{\varepsilon})^2}+2\lambda_0$,
$M_1=2(\alpha_0 + \frac{\alpha_1}{\varepsilon})$, and
$M_2=\sqrt{\left(\frac{\alpha_1}{\varepsilon}\right)^2  + 4n \left( \alpha_0 + \frac{\alpha_1}{\varepsilon} \right)^2 }$.
We consider again four combinations for the pair $(n,N)$: $(n,N)=(50, 1000), (100, 1000), (500, 10\,000)$, and $(1000, 10\,000)$.

The steps used along the iterations of the SMD and MSSMD algorithms are reported in Figure~\ref{stepsra2}. 
\begin{figure}[H]
\centering
\begin{tabular}{ll}
 \includegraphics[scale=0.6]{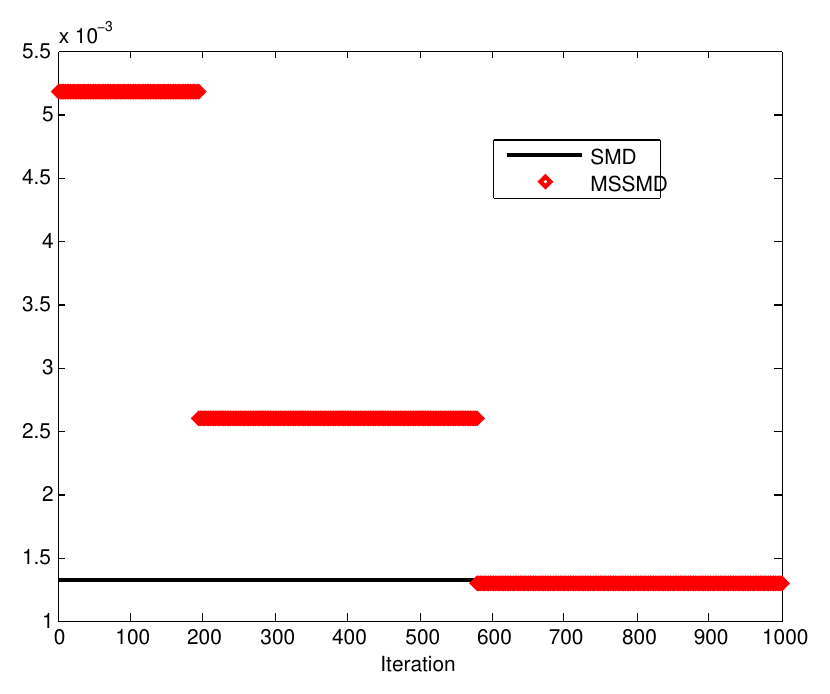}& \includegraphics[scale=0.6]{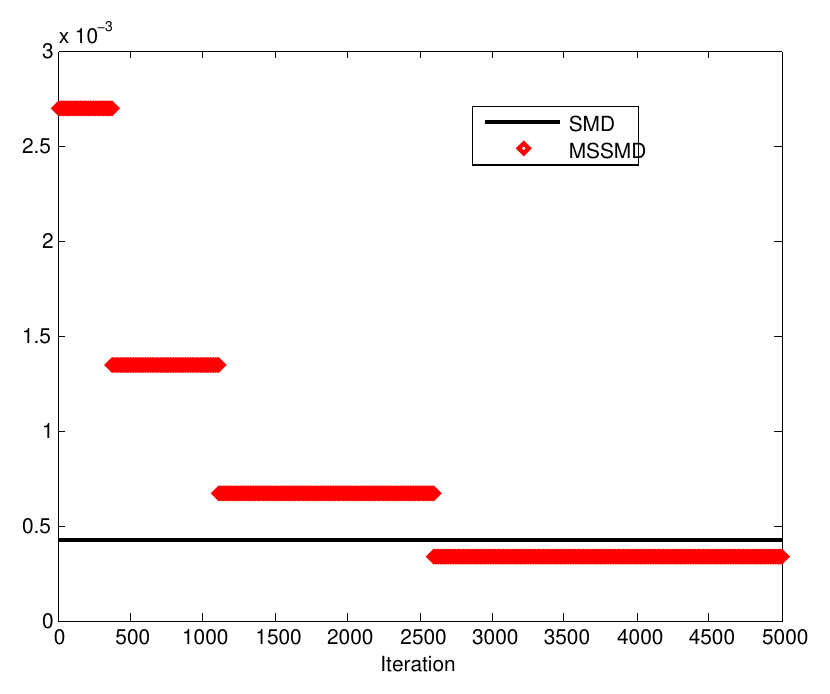}\\
\includegraphics[scale=0.6]{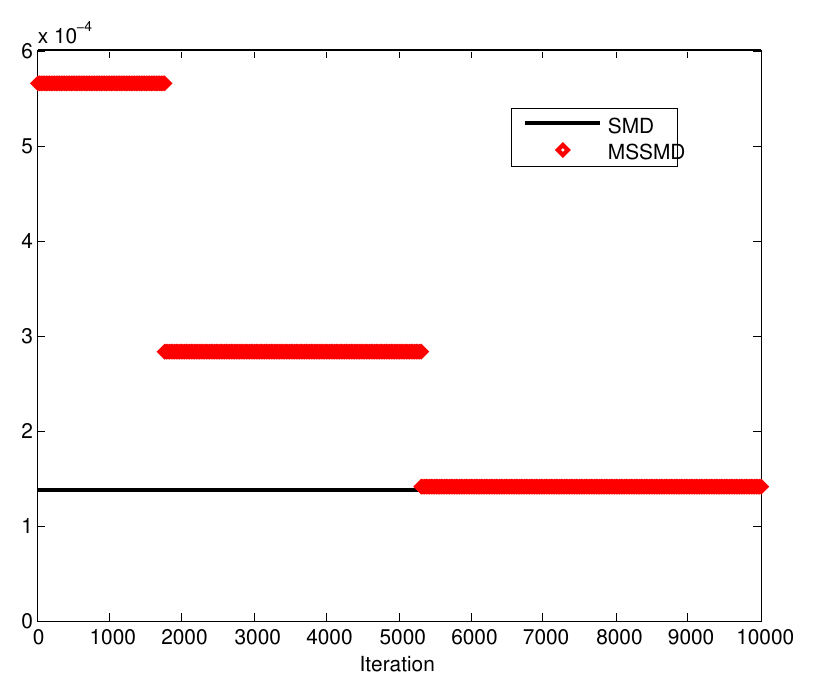}& \includegraphics[scale=0.6]{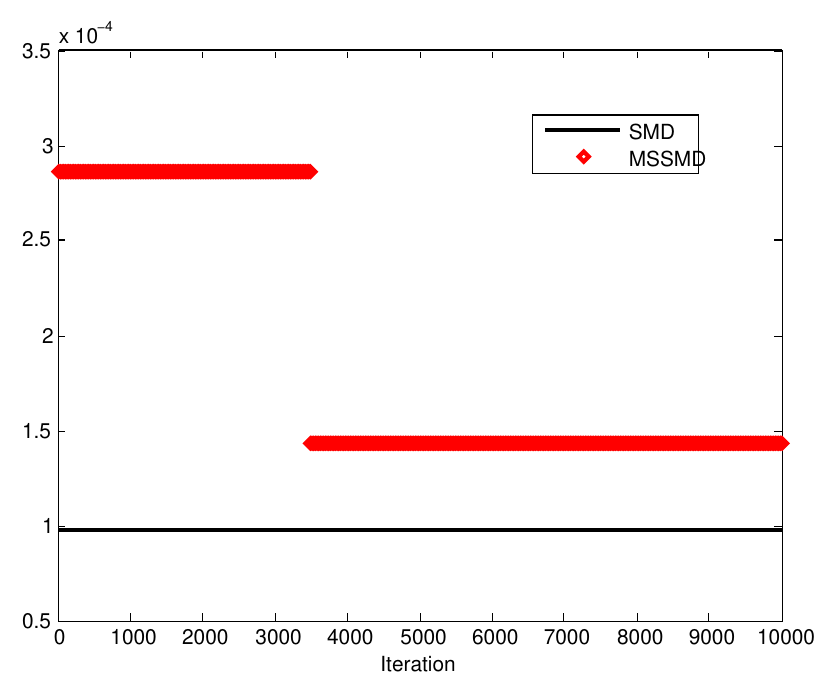}
 \end{tabular}
\caption{Steps used for the SMD and MSSMD algorithms to solve problem \eqref{definstance3} with 
$(n, N)=(50, 1000)$ (top left plot),
$(n, N)=(100, 5000)$ (top right plot), 
$(n, N)=(500, 10\,000)$ (bottom left), $(n, N)=(1000, 10\,000)$ (bottom right).}
\label{stepsra2}
\end{figure}

The average (computed running the algorithms 50 times) of the approximate optimal values and of the value of the objective
function at the approximate solutions are reported in Figures \ref{mssmdra2} and \ref{mssmdra2bis}.
In these experiments we observe again that MSSMD approximate solutions are better along the iterations and
at the end of the optimization process.

\begin{figure}
\centering
\begin{tabular}{ll}
\includegraphics[scale=0.52]{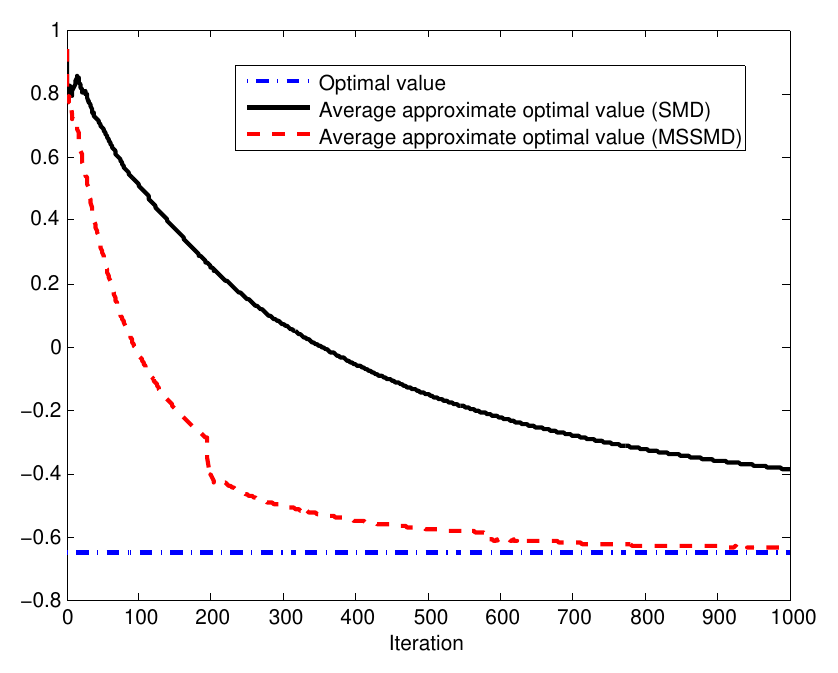}&\includegraphics[scale=0.52]{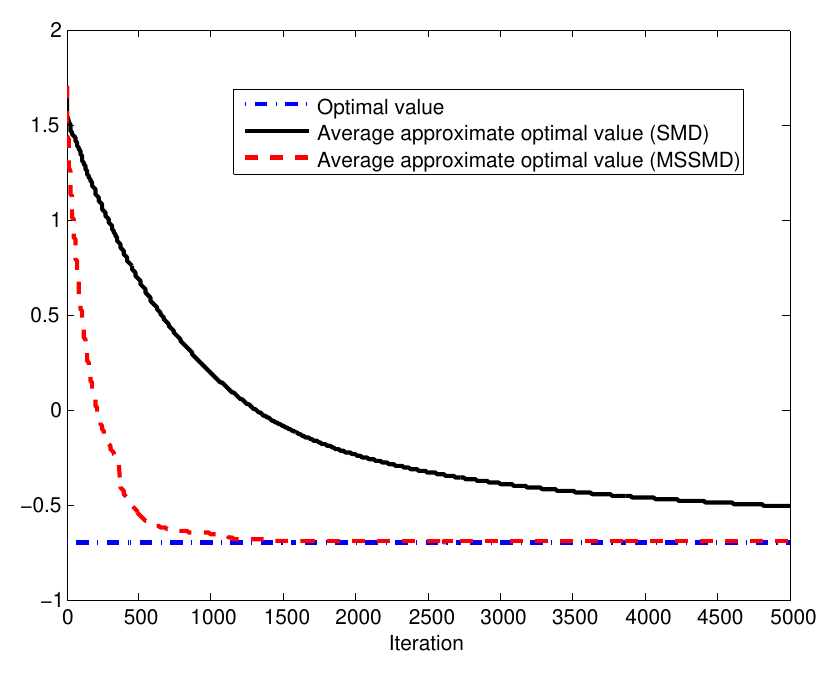}\\ 
\includegraphics[scale=0.52]{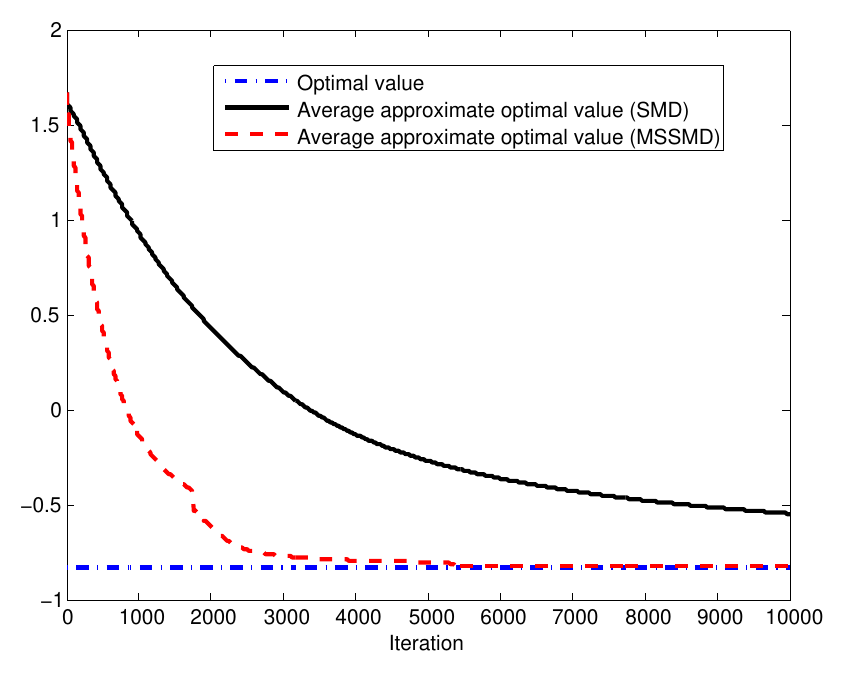}&\includegraphics[scale=0.52]{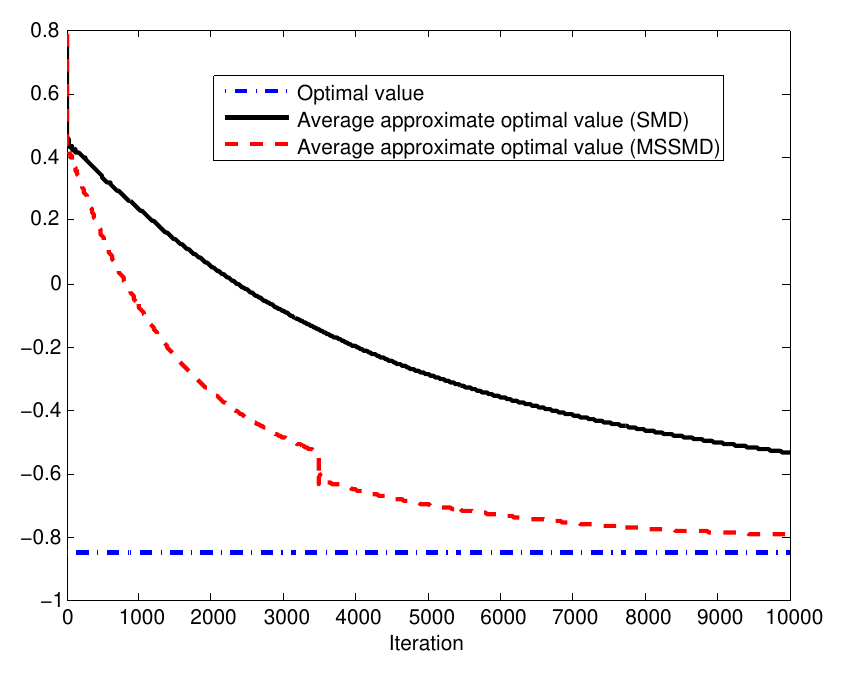}
\end{tabular}
\caption{Average over 50 realizations of the
approximate optimal values computed by the SMD and MSSMD algorithms
to solve \eqref{definstance3}. Top left: $(n, N)=(50, 1000)$, top right: $(n, N)=(100, 5000)$, bottom left: $(n, N)=(500, 10\,000)$, bottom right: $(n, N)=(1000, 10\,000)$.}
\label{mssmdra2}
\end{figure}

\section{Conclusion and future work} \label{applirmfuture}

We derived a new confidence interval on the optimal value of a convex stochastic program using the SMD algorithm that has the advantage
of being quicker to compute and much less conservative than previous confidence intervals.

We introduced a multistep extension of the SMD algorithm and derived a computable nonasymptotic confidence interval on the optimal
value of a risk-averse stochastic program, expressed in terms of EPRM, using this algorithm.
We have shown (using two stochastic optimization problems) that the multistep SMD algorithm can obtain ``good'' solutions much quicker that the SMD algorithm.

Our work is applicable to obtain confidence intervals on the risk measure value of a distribution on the basis
of a sample from this distribution, if this risk measure is an EPRM.

The analysis presented in this paper can be extended in several ways. 

First, numerical tests could be performed to analyze the quality of the confidence intervals given by Corollary \ref{lastcormstep}
for multistep SMD. 
Other algorithms could be considered to solve \eqref{pbpartcase} and the corresponding 
confidence intervals derived.  More general classes of problems, for instance
involving integer variables, could also be analyzed.

Next, we could take a law invariant coherent risk measure 
for $\mathcal{R}$ in \eqref{pbpartcase}. In this situation,
asymptotic confidence intervals on the optimal value of \eqref{pbpartcase} could be obtained
combining the Central Limit Theorem for risk measures given in 
\cite{nancypflug}, the Delta theorem, and the Functional Central Limit Theorem.

\begin{figure}
\centering
\begin{tabular}{ll}
\includegraphics[scale=0.52]{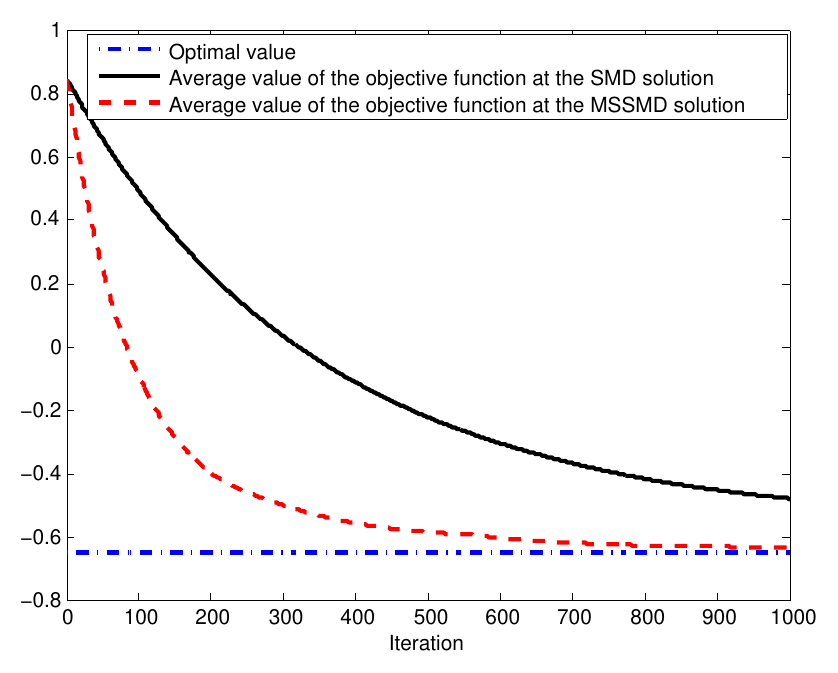} & \includegraphics[scale=0.52]{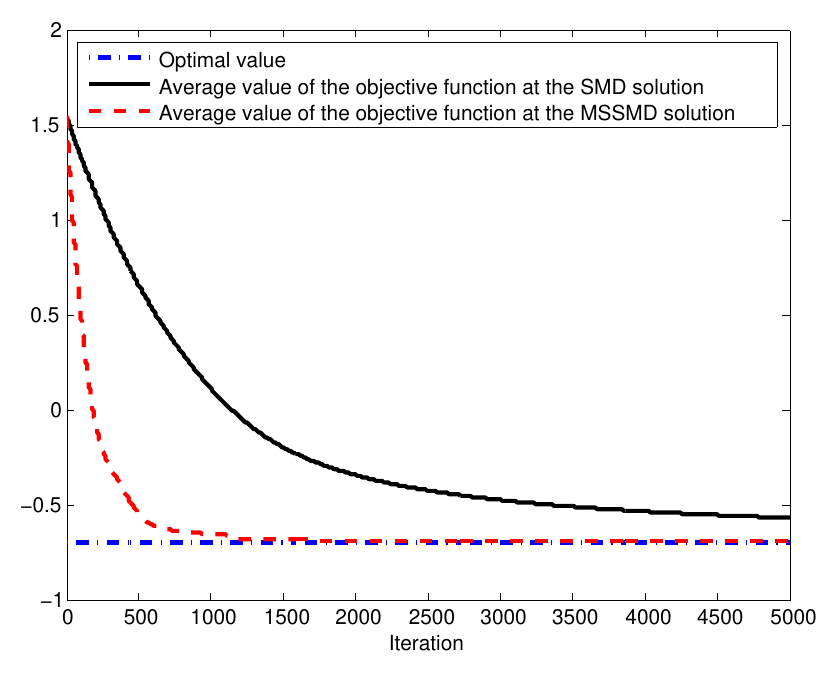}\\ 
\includegraphics[scale=0.52]{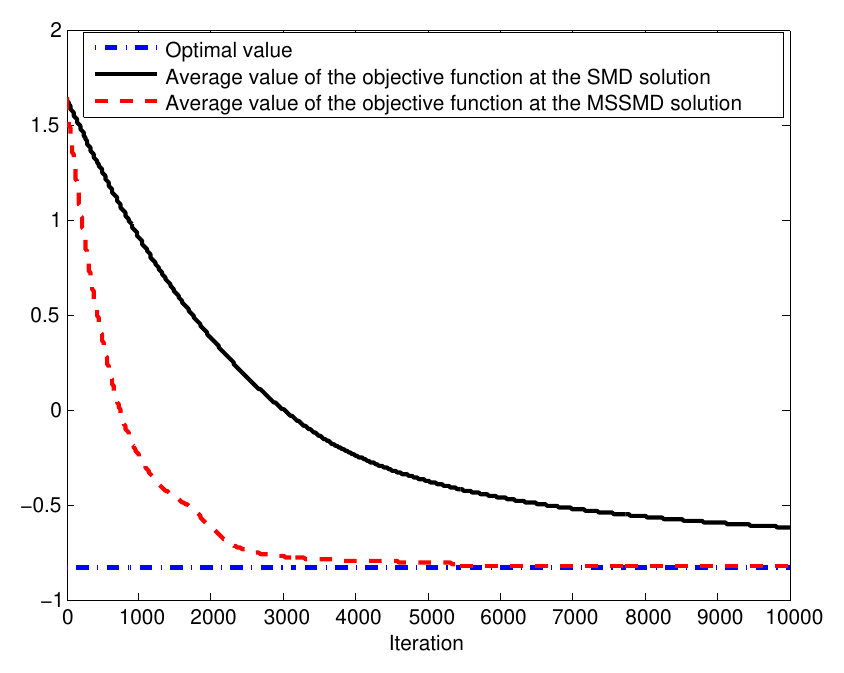} & \includegraphics[scale=0.52]{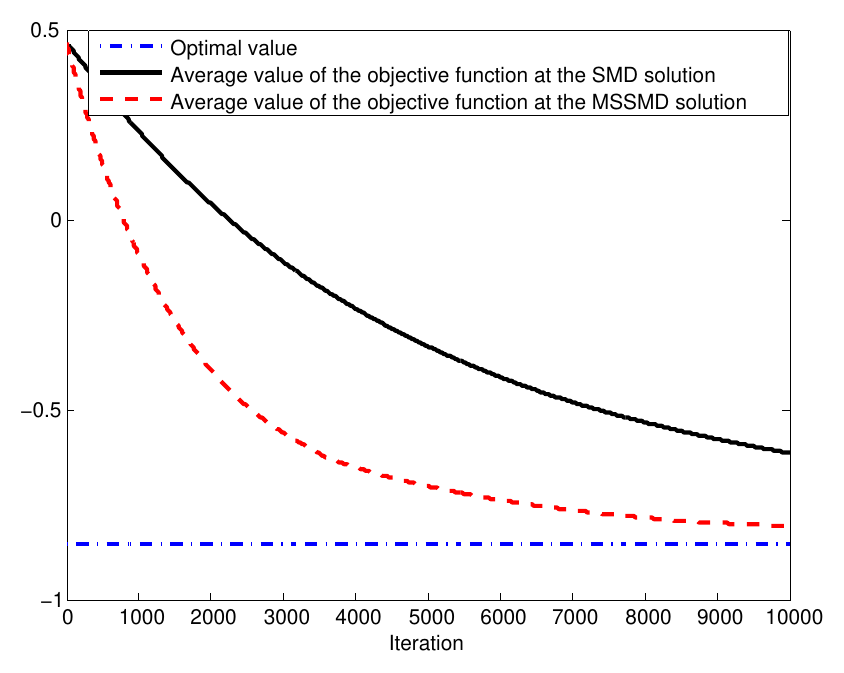}
\end{tabular}
\caption{Average over 50 realizations of the values of the objective function at the approximate solutions (right plots) computed by the SMD and MSSMD algorithms
to solve \eqref{definstance3}. Top left: $(n, N)=(50, 1000)$, top right: $(n, N)=(100, 5000)$, bottom left: $(n, N)=(500, 10\,000)$, bottom right: $(n, N)=(1000, 10\,000)$.}
\label{mssmdra2bis}
\end{figure}

Finally, our analysis can be used to study the following problem: defining
\begin{equation} \label{defpb1}
\rho_i(\xi)=
\left\{
\begin{array}{l}
\min \; f(x):=\mathcal{R}_i\left[ g(x,\xi) \right],\\
x \in X
\end{array}
\right.
\end{equation}
for an EPRM $\mathcal{R}_i$ and given 
samples from the distributions of random vectors $\xi_1, \ldots, \xi_m$, 
our developments can be used to compare the optimal values $\rho_i(\xi_i), \,i=1,\ldots,m$, studying the
following {\em{statistical tests}}:
\begin{equation}  \label{testseveralrm} 
\begin{array}{ll}
(a) \; H_0 :\; \rho_1(\xi_1) = \rho_2(\xi_2) = \ldots = \rho_m(\xi_m)   & \mbox{against }{\overline{H_0}},\\
(b) \; H_0^i :\; \rho_i(\xi_i) \leq \rho_j(\xi_j),\;\;1 \leq j \neq i \leq m   & \mbox{against }{\overline{H_0^{i}}},\\
(c) \; H_0 :\; \rho_1(\xi_1) \leq \rho_2(\xi_2) \leq \ldots \leq \rho_m(\xi_m)   & \mbox{against }{\overline{H_0}},
\end{array}
\end{equation}
where ${\overline{H_0}}$ is the complement of $H_0$.
Without assuming the independence of $\xi_1, \ldots, \xi_m$, 
a special case of \eqref{testseveralrm} is obtained taking a singleton $X=\{x^*_{i}\}$ for the set $X$
defining $\rho_i$, fixing the risk measure $\mathcal{R}_i = \mathcal{R}$ and the distribution $\xi_i = \xi$. 
Setting $\eta_i = g(x^*_{i}, \xi)$, test \eqref{testseveralrm} boils down in this case to 
\begin{equation} \label{testrmintro}
\begin{array}{ll}
\begin{array}{ll}
(a) \; H_0 :\; \mathcal{R}(\eta_1) = \mathcal{R}(\eta_2) = \ldots = \mathcal{R}(\eta_m)   & \mbox{against }{\overline{H_0}},\\
(b) \; H_0^i :\; \mathcal{R}(\eta_i) \leq \mathcal{R}(\eta_j),\;\;1 \leq j \neq i \leq m   & \mbox{against }{\overline{H_0^{i}}},\\
(c) \; H_0 :\; \mathcal{R}(\eta_1) \leq \mathcal{R}(\eta_2) \leq \ldots \leq \mathcal{R}(\eta_m)  & \mbox{against }{\overline{H_0}}.
\end{array}
\end{array}
\end{equation}
These tests are useful when we want to choose among $m$ candidate solutions $x^*_{1}, \ldots, x^*_{m}$ for the problem
$$
\left\{
\begin{array}{l}
\min \; f(x):=\mathcal{R}\left[ g(x,\xi) \right],\\
x \in X,
\end{array}
\right.
$$
the best one (the one with the smallest risk measure value), using risk measure $\mathcal{R}$ to rank the distributions $\eta_i$.

\vspace*{0.5cm}
\par {\textbf{Acknowledgments.}} The author would like to thank Arkadi Nemirovski, Anatoli Juditsky, and
Alexander Shapiro for helpful discussions.
The author's research was 
partially supported by an FGV grant, CNPq grant 307287/2013-0, 
FAPERJ grants E-26/110.313/2014 and E-26/201.599/2014.

\addcontentsline{toc}{section}{References}
\bibliographystyle{plain}
\bibliography{Risk_Averse_SP_Polyhedral}

\begin{thebibliography}{10}

\bibitem{mosek}
E.~D. Andersen and K.~D. Andersen.
\newblock {\em The MOSEK optimization toolbox for MATLAB, manual. Version 7.0},
  2013.
\newblock \url{http://docs.mosek.com/7.0/toolbox/}.

\bibitem{baymorton06}
G.~Bayraksan and D.P. Morton.
\newblock Assessing solution quality in stochastic programs.
\newblock {\em Math. Program.}, 108:495--514, 2006.

\bibitem{baymorton11}
G.~Bayraksan and D.P. Morton.
\newblock A sequential sampling procedure for stochastic programming.
\newblock {\em Oper. Res.}, 59:898--913, 2011.

\bibitem{pierrelouisbayak12}
G.~Bayraksan and P.~Pierre-Louis.
\newblock Fixed-width sequential stopping rules for a class of stochastic
  programs.
\newblock {\em SIAM J. Optim.}, 22:1518--1548, 2012.

\bibitem{Chiralaksanakulmorton04}
A.~Chiralaksanakul and D.P. Morton.
\newblock Assessing policy quality in multi-stage stochastic programming.
\newblock {\em Stochastic Programming E-Print Series}, 12, 2004.

\bibitem{dupawets1988}
J.~Dupa\v cova and R.J.-B Wets.
\newblock Asymptotic behavior of statistical estimators and of optimal
  solutions of stochastic optimization problems.
\newblock {\em Ann. Stat.}, 16:1517--1549, 1988.

\bibitem{rom2}
A.~Eichhorn and W.~R\"omisch.
\newblock Polyhedral risk measures in stochastic programming.
\newblock {\em SIAM J. Optim.}, 16:69--95, 2005.

\bibitem{eichrom07}
A.~Eichorn and W.~R\"omisch.
\newblock Stochastic integer programming: Limit theorems and confidence
  intervals.
\newblock {\em Math. Oper. Res.}, 32:118--135, 2007.

\bibitem{ghadimilan}
S.~Ghadimi and G.~Lan.
\newblock Optimal stochastic approximation algorithms for strongly convex
  stochastic composite optimization i: A generic algorithmic framework.
\newblock {\em SIAM Journal on Optimization}, 22:1469--1492, 2012.

\bibitem{guigues2014cvsddp}
V.~Guigues.
\newblock {Convergence analysis of sampling-based decomposition methods for
  risk-averse multistage stochastic convex programs}.
\newblock {\em Available on arXiv at http://arxiv.org/abs/1408.4439}, 2014.

\bibitem{ioudnemgui15}
V.~Guigues, A.~Juditsky, and A.~Nemirovski.
\newblock Non-asymptotic confidence bounds for the optimal value of a
  stochastic program.
\newblock {\em Available on arXiv at http://arxiv.org/abs/1601.07592}, 2016.

\bibitem{guiguesromisch10}
V.~Guigues and W.~R\"omisch.
\newblock Sampling-based decomposition methods for multistage stochastic
  programs based on extended polyhedral risk measures.
\newblock {\em SIAM J. Optim.}, 22:286--312, 2012.

\bibitem{nestioud2010}
A.~Juditsky and Y.~Nesterov.
\newblock Primal-dual subgradient methods for minimizing uniformly convex
  functions.
\newblock {\em Available on arXiv at http://arxiv.org/abs/1401.1792}, 2010.

\bibitem{MLOPT}
Anatoli Juditsky and Arkadi Nemirovski.
\newblock First order methods for nonsmooth convex large-scale optimization,
  {I}: general purpose methods.
\newblock In S.~Sra, S.~Nowozin, and S.J. Wright, editors, {\em Optimization
  for Machine Learning}, pages 121--148. MIT Press, 2011.

\bibitem{kingrock1993}
A.J. King and R.T. Rockafellar.
\newblock Asymptotic theory for solutions in statistical estimation and
  stochastic programming.
\newblock {\em Math. Oper. Res.}, 18:148--162, 1993.

\bibitem{kleywegt2001}
A.J. Kleywegt, A.~Shapiro, and T.~Homem de~Mello.
\newblock {The sample average approximation method for stochastic discrete
  optimization}.
\newblock {\em SIAM J. Optim.}, 12:479–502, 2001.

\bibitem{ahmedshapiro2002}
A.J. Kleywegt, A.~Shapiro, and T.~Homem de~Mello.
\newblock The sample average approximation method for stochastic programs with
  integer recourse.
\newblock {\em Optimization OnLine}, 2002.

\bibitem{nemlansh09}
G.~Lan, A.~Nemirovski, and A.~Shapiro.
\newblock Validation analysis of mirror descent stochastic approximation
  method.
\newblock {\em Math. Program.}, 134:425--458, 2012.

\bibitem{makmortonwood}
W.K. Mak, D.P. Morton, and R.K. Wood.
\newblock {Monte Carlo bounding techniques for determining solution quality in
  stochastic programs}.
\newblock {\em Oper. Res. Lett.}, 24:47--56, 1999.

\bibitem{nediclee}
A.~Nedich and S.~Lee.
\newblock On stochastic subgradient mirror-descent algorithm with weighted
  averaging.
\newblock {\em SIAM Journal on Optimization}, 24:84--107, 2014.

\bibitem{nemjudlannem09}
A.~Nemirovski, A.~Juditsky, G.~Lan, and A.~Shapiro.
\newblock Robust stochastic approximation approach to stochastic programming.
\newblock {\em SIAM J. Optim.}, 19:1574--1609, 2009.

\bibitem{pereira}
M.V.F. Pereira and L.M.V.G Pinto.
\newblock Multi-stage stochastic optimization applied to energy planning.
\newblock {\em Math. Program.}, 52:359--375, 1991.

\bibitem{pflug1995asy}
G.~Pflug.
\newblock Asymptotic stochastic programs.
\newblock {\em Math. Oper. Res.}, 20:769--789, 1995.

\bibitem{pflug99}
G.~Pflug.
\newblock Stochastic programs and statistical data.
\newblock {\em Ann. Oper. Res.}, 85:59--78, 1999.

\bibitem{nancypflug}
G.~Pflug and N.~Wozabal.
\newblock Asymptotic distribution of law-invariant risk functionals.
\newblock {\em Finance Stoch}, 14:397--418, 2010.

\bibitem{polyak90}
B.T. Polyak.
\newblock New stochastic approximation type procedures.
\newblock {\em Automat. i Telemekh (English translation: Automation and Remote
  Control)}, 7:98--107, 1990.

\bibitem{polyakjud92}
B.T. Polyak and A.~Juditsky.
\newblock Acceleration of stochastic approximation by averaging.
\newblock {\em SIAM J. Contr. and Optim.}, 30:838--855, 1992.

\bibitem{rakhlinshamirsrid}
A.~Rakhlin, O.~Shamir, and K.~Sridharan.
\newblock Making gradient descent optimal for strongly convex stochastic
  optimization.
\newblock {\em 29th International Conference on Machine Learning (ICML)}, 2012.

\bibitem{monroe51}
H.~Robbins and S.~Monroe.
\newblock {A stochastic approximation method}.
\newblock {\em Annals of Math. Stat.}, 22:400--407, 1951.

\bibitem{ury2}
R.T. Rockafellar and S.~Uryasev.
\newblock {Conditional Value-at-Risk for general loss distributions}.
\newblock {\em J. Bank. Financ.}, 26(7):1443--1471, 2002.

\bibitem{romdelta2005}
W.~R\"omisch.
\newblock Delta method, infinite dimensional.
\newblock {\em In S. Kotz, C. B. Read, N. Balakrishnan, B. Vidakovic, eds.
  Extended entry, Encyclopedia of Statistical Sciences, 2nd ed., Wiley, New
  York.}, 2005.

\bibitem{romschultz93}
W.~R\"omisch and R.~Schultz.
\newblock Stability of solutions for stochastic programs with complete
  recourse.
\newblock {\em Math. Oper. Res.}, 18:590--609, 1993.

\bibitem{schultz94}
R.~Schultz.
\newblock Strong convexity in stochastic programs with complete recourse.
\newblock {\em Journal of Computational and Applied Mathematics}, 56:3--22,
  1994.

\bibitem{shap1989}
A.~Shapiro.
\newblock Asymptotic properties of statistical estimators in stochastic
  programming.
\newblock {\em Ann. Statist.}, 17:841--858, 1989.

\bibitem{shap1991}
A.~Shapiro.
\newblock Asymptotic analysis of stochastic programs.
\newblock {\em Ann. Oper. Res.}, 30:169--186, 1991.

\bibitem{shap2000}
A.~Shapiro and T.~Homem de~Mello.
\newblock {On rate of convergence of optimal solutions of Monte Carlo
  approximations of stochastic programs}.
\newblock {\em SIAM J. Optim.}, 11:70--86, 2000.

\bibitem{shadenrbook}
A.~Shapiro, D.~Dentcheva, and A.~Ruszczy\'nski.
\newblock {\em {Lectures on Stochastic Programming: Modeling and Theory}}.
\newblock SIAM, Philadelphia, 2009.

\bibitem{talagrand1994}
M.~Talagrand.
\newblock {Sharper bounds for Gaussian and empirical processes}.
\newblock {\em Ann. Probab.}, 22:28--76, 1994.

\bibitem{talagrand1996}
M.~Talagrand.
\newblock {The Glivenko-Cantelli problem, ten years later}.
\newblock {\em J. Theoret. Probab.}, 9:371--384, 1996.

\end{thebibliography}

\section*{Appendix}

We have collected in the Appendix two proofs, essentially known, see \cite{nemjudlannem09}.

\begin{proof}[Proof of Lemma \ref{lemmalargedevexp}.]
We first show that for any $\gamma>0$ and $\tau=1,\ldots,N$, we have
\begin{equation}\label{expineq0}
\bE_{|\tau-1}\Big[ \exp\{\gamma\eta_\tau\} \Big] \leq\exp\{\gamma^2\}.
\end{equation}
Let us fix $0<\gamma \leq 1$.
Observing that
\begin{equation}\label{expineq}
e^x \le x+e^{x^2} \mbox{ for every }x \in \mathbb{R},
\end{equation}
we obtain
$$
\begin{array}{lll}
\bE_{|\tau-1}\Big[\exp\{\gamma\eta_\tau\}\Big] & \leq & \bE_{|\tau-1}\Big[\gamma \eta_\tau \Big] +  \bE_{|\tau-1}\Big[\exp\{\gamma^2\eta_\tau^2\}\Big]\vspace*{0.1cm}\\
& \leq & \bE_{|\tau-1}\Big[\exp\{\gamma^2\eta_\tau^2\}\Big] \mbox{ using }\eqref{martin}\\
& \leq &    \bE_{|\tau-1}\Big[(\exp\{\eta_\tau^2\})^{\gamma^2}\Big]\leq
\left(\bE_{|\tau-1}\Big[\exp\{\eta_\tau^2\}\Big]\right)^{\gamma^2},
\end{array}
$$
where the last inequality is Jensen inequality applied to the concave function $x^{\gamma^2}$.
Plugging \eqref{martin} into the above inequality shows that \eqref{expineq0} holds for $0<\gamma \leq 1$.

For $\gamma>1$,
$$
\begin{array}{lll}
\bE_{|\tau-1}\Big[ \exp\{\gamma\eta_\tau\}\Big] & \leq & \bE_{|\tau-1}\Big[ \exp\{{1\over 2}\gamma^2+{1\over 2}\eta_\tau^2\}\Big]\vspace*{0.1cm}\\
&\leq & \exp\{\frac{\gamma^2}{2}\} \sqrt{ \bE_{|\tau-1}\Big[ \exp\{ \eta_\tau^2 \}\Big]  } \leq \exp\{ \frac{\gamma^2 + 1}{2}\} \leq \exp\{ \gamma^2 \},
\end{array}
$$
where we have used \eqref{martin} for the third inequality and the fact that $\gamma>1$ for the last one.
We have thus shown that \eqref{expineq0} holds for every $\gamma>0$.
As a result, for $\gamma>0$, setting $S_\tau=\sum_{s=1}^\tau\eta_s$, we have
$$
\begin{array}{lll}
\bE \Big[ \exp\{\gamma S_\tau\} \Big] &=&\bE\Big[ \exp\{\gamma S_{\tau-1}\} \bE_{|\tau-1}\Big[ \exp\{\gamma\eta_\tau\} \Big] \Big]\\
&\leq & \exp\{\gamma^2\}\bE\Big[ \exp\{\gamma S_{\tau-1}\}\Big] \mbox{ using }\eqref{expineq0}.
\end{array}
$$
It follows that for $\gamma>0$
\begin{equation}\label{bornesupstau}
\begin{array}{lll}
\bE \Big[ \exp\{\gamma S_\tau\} \Big] &\leq & \exp\{\gamma^2 (\tau - 1)\} \bE\Big[ \exp\{\gamma \eta_1 \}\Big] \leq \exp\{\gamma^2 \tau \} \mbox{ using }\eqref{expineq0}.
\end{array}
\end{equation}
Next, for $\gamma>0$,
$$
\begin{array}{lll}
\mathbb{P}\Big( S_N >\Theta \, \sqrt{N}\Big) & =   & \mathbb{P}\Big( \exp\{ \gamma S_N\}  >  \exp\{ \Theta \, \sqrt{N} \gamma\} \Big)\\
& \leq  & \displaystyle \min_{\gamma>0} \; \exp\{-\Theta \sqrt{N} \gamma\} \bE \Big[ \exp\{\gamma S_N \} \Big]  \mbox{ using Chernoff bound,}\\
& \leq  &  \exp\{\displaystyle \min_{\gamma>0} \Big[ \gamma^2 N - \Theta \sqrt{N} \gamma \Big] \} = \exp\{-\Theta^2/4\}  \mbox{ using }\eqref{bornesupstau}.
\end{array}
$$
This achieves the proof of inequality \eqref{neqlemma2}.\hfill
\end{proof}

\begin{proof}[Proof of Lemma \ref{lemmaMD}.]
Invoking (\ref{eq448}), we get
$$
\forall y \in X: \gamma_\tau e_\tau^\transp (u_{\tau+1}-y ) \leq V_{u_\tau}(y)-V_{u_{\tau+1}}(y)-V_{u_\tau}(u_{\tau+1}),
$$
whence for all $x \in X$, we have
$$
\begin{array}{rcl}
\gamma_\tau  e_\tau^\transp( u_{\tau}-y) &\leq& V_{u_\tau}(y)-V_{u_{\tau+1}}(y)+\Big[\gamma_\tau e_\tau^\transp (u_\tau-u_{\tau+1}) -V_{u_\tau}(u_{\tau+1})\Big]\vspace*{0.1cm}\\
&\leq& V_{u_\tau}(y)-V_{u_{\tau+1}}(y)+\Big[\gamma_\tau \|e_\tau\|_*\|u_\tau-u_{\tau+1}\| -{\mu(\omega)\over 2}\|u_\tau-u_{\tau+1}\|^2 \Big]\\
&\leq& V_{u_\tau}(y)-V_{u_{\tau+1}}(y)+ \displaystyle \frac{\gamma_\tau^2\|e_\tau\|_*^2}{2 \mu( \omega )},\\
\end{array}
$$
where we have used (\ref{strong}) for the second inequality.
Summing up the resulting inequalities over $\tau=1,...,N$, and taking into account that $V_{u_{N+1}}(y)\geq0$ by (\ref{strong}) and
$V_{u_1}(y)\leq {1\over 2} D_{\omega, X}^2$
by (\ref{note}) (recall that $u_1=x_\omega$), we arrive at (\ref{lemmaMDineq}).\hfill
\end{proof}

\end{document}